\documentclass[12pt,reqno]{amsart}
\usepackage{latexsym, amsfonts, amsmath, amsthm, amssymb,amscd,color}
\usepackage{pstricks}
\usepackage{pst-optic}
\allowdisplaybreaks

\newtheorem{thm}{Theorem}[section]
\newtheorem{rmk}{Remark}[section]
\newtheorem{prop}[thm]{Proposition}

\newtheorem{cor}[thm]{Corollary}

\newtheorem{lem}[thm]{Lemma}

\numberwithin{equation}{section}

  \def\xI{0.866025} \def\xII{1.73205} \def\xIII{2.59808} \def\xIV{3.4641} \def\xV{4.33013} \def\xVI{5.19615} \def\xVII{6.062175}
 \def\pav{\rput{60}(0,0){\psline(0,0)(0,3)} \rput{-60}(0,0){\psline(0,0)(0,4)} \rput{0}(-\xIII,1.5){\psline(0,0)(0,5)} \rput{0}(\xIV,2){\psline(0,0)(0,5)} \rput{60}(\xIV,7){\psline(0,0)(0,3)}
 \rput{120}(0.866025,8.5){\psline(0,0)(0,4)} \rput{60}(\xI,0.5){\psline(0,0)(0,1)\psline(0,2)(0,4)} \rput{60}(\xII,1){\psline(0,1)(0,3)\psline(0,4)(0,5)} \rput{60}(\xIII,1.5){\psline(0,1)(0,3)\psline(0,5)(0,6)}
 \rput{60}(\xIV,2){\psline(0,2)(0,4)\psline(0,6)(0,7)} \rput{60}(\xIV,3){\psline(0,1)(0,2)\psline(0,3)(0,4)\psline(0,5)(0,6)} \rput{60}(\xIV,4){\psline(0,1)(0,3)\psline(0,5)(0,6)}
 \rput{60}(\xIV,5){\psline(0,1)(0,3)\psline(0,4)(0,5)} \rput{60}(\xIV,6){\psline(0,0)(0,2)\psline(0,3)(0,4)}
 \rput{120}(\xIV,3){\psline(0,0)(0,3)\psline(0,4)(0,5)} \rput{120}(\xIV,4){\psline(0,0)(0,2)\psline(0,3)(0,5)} \rput{120}(\xIV,5){\psline(0,0)(0,1)\psline(0,2)(0,3)\psline(0,4)(0,6)}
 \rput{120}(\xIV,6){\psline(0,0)(0,1)\psline(0,2)(0,3)\psline(0,4)(0,6)} \rput{120}(\xIV,7){\psline(0,1)(0,2)\psline(0,3)(0,6)} \rput{120}(\xIII,7.5){\psline(0,1)(0,5)}
 \rput{120}(\xII,8){\psline(0,0)(0,3)\psline(0,4)(0,5)} \rput{0}(-\xII,1){\psline(0,0)(0,4)\psline(0,5)(0,6)} \rput{0}(-\xI,0.5){\psline(0,0)(0,1)\psline(0,2)(0,6)} \rput{0}(0,0){\psline(0,1)(0,2)\psline(0,3)(0,7)}
 \rput{0}(\xI,0.5){\psline(0,0)(0,1)\psline(0,2)(0,4)\psline(0,5)(0,7)} \rput{0}(\xII,1){\psline(0,0)(0,2)\psline(0,3)(0,5)\psline(0,6)(0,7)} \rput{0}(\xIII,1.5){\psline(0,0)(0,4)\psline(0,5)(0,6)}}

 \def\centpav{ \rput{-60}(0,0){\psline(0,0)(0,4)} \rput{0}(\xIV,2){\psline(0,0)(0,5)} \rput{0}(-\xIV,2){\psline(0,0)(0,5)} \rput{-60}(-\xIV,7){\psline(0,0)(0,4)}
 \rput{60}(\xIV,7){\psline(0,0)(0,4)} \rput{-60}(-\xIV,7){\psline(0,0)(0,4)}
 \rput{60}(0,0){\psline(0,0)(0,4)} \rput{60}(\xI,0.5){\psline(0,1)(0,5)} \rput{60}(\xII,1){\psline(0,2)(0,6)} \rput{60}(\xIII,1.5){\psline(0,2)(0,3)\psline(0,4)(0,7)}
 \rput{60}(\xIV,2){\psline(0,2)(0,3)\psline(0,4)(0,5)\psline(0,6)(0,8)} \rput{60}(\xIV,3){\psline(0,1)(0,2)\psline(0,3)(0,4)\psline(0,5)(0,7)} \rput{60}(\xIV,4){\psline(0,1)(0,3)\psline(0,5)(0,7)}
 \rput{60}(\xIV,5){\psline(0,1)(0,3)\psline(0,4)(0,6)} \rput{60}(\xIV,6){\psline(0,0)(0,2)\psline(0,3)(0,5)}
 \rput{-60}(-\xI,0.5){\psline(0,1)(0,5)} \rput{-60}(-\xII,1){\psline(0,2)(0,6)} \rput{-60}(-\xIII,1.5){\psline(0,2)(0,5)\psline(0,6)(0,7)} \rput{-60}(-\xIV,2){\psline(0,3)(0,6)\psline(0,7)(0,8)}
 \rput{-60}(-\xIV,3){\psline(0,2)(0,5)\psline(0,6)(0,7)} \rput{-60}(-\xIV,4){\psline(0,2)(0,6)} \rput{-60}(-\xIV,5){\psline(0,1)(0,2)\psline(0,3)(0,6)} \rput{-60}(-\xIV,6){\psline(0,0)(0,1)\psline(0,2)(0,5)}
 \rput{0}(-\xIII,1.5){\psline(0,0)(0,4)\psline(0,5)(0,6)} \rput{0}(-\xII,1){\psline(0,0)(0,4)\psline(0,5)(0,6)} \rput{0}(-\xI,0.5){\psline(0,0)(0,3)\psline(0,4)(0,6)}
 \rput{0}(0,0){\psline(0,0)(0,2)\psline(0,3)(0,4)\psline(0,5)(0,7)} \rput{0}(\xI,0.5){\psline(0,0)(0,2)\psline(0,3)(0,4)\psline(0,5)(0,7)} \rput{0}(\xII,1){\psline(0,0)(0,2)\psline(0,3)(0,5)\psline(0,6)(0,7)}
 \rput{0}(\xIII,1.5){\psline(0,0)(0,4)\psline(0,5)(0,6)}
 \pspolygon[fillstyle=solid,fillcolor=lightgray](0,4)(-\xI,4.5)(0,5)(\xI,4.5)}

 \def\semipavf{ \rput{60}(0,0){\psline(0,0)(0,7)} \rput{60}(0,1){\psline(0,0)(0,3)\psline(0,4)(0,7)} \rput{60}(0,2){\psline(0,0)(0,1)\psline(0,2)(0,4)\psline(0,5)(0,7)} \rput{60}(0,3){\psline(0,0)(0,2)\psline(0,3)(0,6)}
 \rput{60}(0,4){\psline(0,0)(0,1)\psline(0,3)(0,6)} \rput{60}(0,5){\psline(0,0)(0,1)\psline(0,2)(0,4)\psline(0,5)(0,6)} \rput{60}(0,6){\psline(0,0)(0,3)\psline(0,4)(0,5)} \rput{60}(0,7){\psline(0,1)(0,3)\psline(0,4)(0,5)}
 \rput{60}(0,8){\psline(0,1)(0,4)} \rput{60}(0,9){\psline(0,0)(0,2)\psline(0,3)(0,4)} \rput{60}(0,10){\psline(0,0)(0,3)} \rput{60}(0,11){\psline(0,0)(0,2)} \rput{60}(0,12){\psline(0,0)(0,1)}
 \rput{120}(0,1){\psline(0,0)(0,1)} \rput{120}(0,2){\psline(0,0)(0,2)} \rput{120}(0,3){\psline(0,1)(0,3)} \rput{120}(0,4){\psline(0,0)(0,2)\psline(0,3)(0,4)} \rput{120}(0,5){\psline(0,1)(0,4)}
 \rput{120}(0,6){\psline(0,1)(0,3)\psline(0,4)(0,5)} \rput{120}(0,7){\psline(0,0)(0,3)\psline(0,4)(0,5)} \rput{120}(0,8){\psline(0,0)(0,1)\psline(0,2)(0,4)\psline(0,5)(0,6)}
 \rput{120}(0,9){\psline(0,0)(0,1)\psline(0,3)(0,5)\psline(0,6)(0,7)} \rput{120}(0,10){\psline(0,1)(0,2)\psline(0,3)(0,5)\psline(0,6)(0,7)}\rput{120}(0,11){\psline(0,0)(0,1)\psline(0,2)(0,4)\psline(0,5)(0,7)}
 \rput{120}(0,12){\psline(0,0)(0,3)\psline(0,4)(0,7)} \rput{120}(0,13){\psline(0,0)(0,7)}
 \rput{0}(0,0){\psline(0,2)(0,3)\psline(0,4)(0,6)\psline(0,7)(0,10)} \rput{0}(-\xI,0.5){\psline(0,2)(0,3)\psline(0,4)(0,8)\psline(0,9)(0,10)}
 \rput{0}(-\xII,1){\psline(0,1)(0,2)\psline(0,3)(0,5)\psline(0,6)(0,8)\psline(0,9)(0,10)} \rput{0}(-\xIII,1.5){\psline(0,1)(0,4)\psline(0,6)(0,9)}
 \rput{0}(-\xIV,2){\psline(0,0)(0,1)\psline(0,2)(0,4)\psline(0,5)(0,7)\psline(0,8)(0,9)} \rput{0}(-\xV,2.5){\psline(0,0)(0,2)\psline(0,3)(0,5)\psline(0,6)(0,8)}
 \rput{0}(-\xVI,3){\psline(0,0)(0,2)\psline(0,3)(0,7)} \rput{0}(-\xVII,3.5){\psline(0,0)(0,6)}
 \def\completion{\rput{-60}(0,0){\psline(0,0)(0,1)}\rput{-120}(0,1){\psline(0,1)(0,0)}} \rput{0}(0,0){\completion} \rput{0}(0,1){\completion} \rput{0}(0,3){\completion} \rput{0}(0,6){\completion} \rput{0}(0,10){\completion}
 \rput{0}(0,11){\completion} \rput{0}(0,12){\completion} \psline[linestyle=dotted](0,-0.5)(0,14) \uput[0](0,13.5){$\ell$}}

\def\semipava{ \rput{60}(0,0){\psline(0,0)(0,7)} \rput{60}(0,1){\psline(0,0)(0,3)\psline(0,4)(0,7)} \rput{60}(0,2){\psline(0,0)(0,1)\psline(0,2)(0,4)\psline(0,5)(0,7)} \rput{60}(0,3){\psline(0,0)(0,2)\psline(0,3)(0,6)}
 \rput{60}(0,4){\psline(0,0)(0,2)\psline(0,3)(0,6)} \rput{60}(0,5){\psline(0,0)(0,2)\psline(0,2)(0,4)\psline(0,5)(0,6)} \rput{60}(0,6){\psline(0,0)(0,3)\psline(0,4)(0,5)} \rput{60}(0,7){\psline(0,1)(0,3)\psline(0,4)(0,5)}
 \rput{60}(0,8){\psline(0,1)(0,4)} \rput{60}(0,9){\psline(0,0)(0,2)\psline(0,3)(0,4)} \rput{60}(0,10){\psline(0,0)(0,3)} \rput{60}(0,11){\psline(0,0)(0,2)} \rput{60}(0,12){\psline(0,0)(0,1)}
 \rput{120}(0,1){\psline(0,0)(0,1)} \rput{120}(0,2){\psline(0,0)(0,2)} \rput{120}(0,3){\psline(0,1)(0,3)} \rput{120}(0,4){\psline(0,1)(0,2)\psline(0,3)(0,4)} \rput{120}(0,5){\psline(0,2)(0,4)}
 \rput{120}(0,6){\psline(0,2)(0,3)\psline(0,4)(0,5)} \rput{120}(0,7){\psline(0,0)(0,1)\psline(0,2)(0,3)\psline(0,4)(0,5)} \rput{120}(0,8){\psline(0,0)(0,1)\psline(0,2)(0,4)\psline(0,5)(0,6)}
 \rput{120}(0,9){\psline(0,0)(0,1)\psline(0,3)(0,5)\psline(0,6)(0,7)} \rput{120}(0,10){\psline(0,1)(0,2)\psline(0,3)(0,5)\psline(0,6)(0,7)} \rput{120}(0,11){\psline(0,0)(0,1)\psline(0,2)(0,4)\psline(0,5)(0,7)}
 \rput{120}(0,12){\psline(0,0)(0,3)\psline(0,4)(0,7)} \rput{120}(0,13){\psline(0,0)(0,7)}
 \rput{0}(0,0){\psline(0,2)(0,6)\psline(0,7)(0,10)} \rput{0}(-\xI,0.5){\psline(0,2)(0,8)\psline(0,9)(0,10)} \rput{0}(-\xII,1){\psline(0,1)(0,2)\psline(0,3)(0,5)\psline(0,6)(0,8)\psline(0,9)(0,10)}
 \rput{0}(-\xIII,1.5){\psline(0,1)(0,4)\psline(0,6)(0,9)} \rput{0}(-\xIV,2){\psline(0,0)(0,1)\psline(0,2)(0,4)\psline(0,5)(0,7)\psline(0,8)(0,9)} \rput{0}(-\xV,2.5){\psline(0,0)(0,2)\psline(0,3)(0,5)\psline(0,6)(0,8)}
 \rput{0}(-\xVI,3){\psline(0,0)(0,2)\psline(0,3)(0,7)} \rput{0}(-\xVII,3.5){\psline(0,0)(0,6)}
 \def\completion{\rput{-60}(0,0){\psline(0,0)(0,1)}\rput{-120}(0,1){\psline(0,1)(0,0)}} \rput{0}(0,0){\completion} \rput{0}(0,1){\completion} \rput{0}(0,6){\completion} \rput{0}(0,10){\completion}
 \rput{0}(0,11){\completion} \rput{0}(0,12){\completion} \psline[linestyle=dotted](0,-0.5)(0,14) \uput[0](0,13.5){$\ell$}
 \pspolygon[fillstyle=solid,fillcolor=lightgray](-\xII,6)(-\xIII,6.5)(-\xII, 7)}

 \def\semipavb{
 \rput{60}(0,0){\psline(0,0)(0,7)} \rput{60}(0,1){\psline(0,0)(0,3)\psline(0,4)(0,7)} \rput{60}(0,2){\psline(0,0)(0,1)\psline(0,2)(0,4)\psline(0,5)(0,7)} \rput{60}(0,3){\psline(0,0)(0,2)\psline(0,3)(0,6)}
 \rput{60}(0,4){\psline(0,1)(0,2)\psline(0,3)(0,6)} \rput{60}(0,5){\psline(0,1)(0,2)\psline(0,3)(0,4)\psline(0,5)(0,6)} \rput{60}(0,6){\psline(0,1)(0,2)\psline(0,4)(0,5)} \rput{60}(0,7){\psline(0,1)(0,2)\psline(0,4)(0,5)}
 \rput{60}(0,8){\psline(0,1)(0,2)\psline(0,3)(0,4)} \rput{60}(0,9){\psline(0,1)(0,2)\psline(0,3)(0,4)} \rput{60}(0,10){\psline(0,1)(0,3)} \rput{60}(0,11){\psline(0,0)(0,2)} \rput{60}(0,12){\psline(0,0)(0,1)}
 \rput{120}(0,1){\psline(0,0)(0,1)} \rput{120}(0,2){\psline(0,0)(0,2)} \rput{120}(0,3){\psline(0,1)(0,3)} \rput{120}(0,4){\psline(0,0)(0,2)\psline(0,3)(0,4)} \rput{120}(0,5){\psline(0,0)(0,1)\psline(0,2)(0,4)}
 \rput{120}(0,6){\psline(0,0)(0,1)\psline(0,2)(0,3)\psline(0,4)(0,5)} \rput{120}(0,7){\psline(0,0)(0,1)\psline(0,2)(0,3)\psline(0,4)(0,5)} \rput{120}(0,8){\psline(0,0)(0,1)\psline(0,2)(0,4)\psline(0,5)(0,6)}
 \rput{120}(0,9){\psline(0,0)(0,1)\psline(0,2)(0,5)\psline(0,6)(0,7)} \rput{120}(0,10){\psline(0,0)(0,1)\psline(0,2)(0,5)\psline(0,6)(0,7)} \rput{120}(0,11){\psline(0,0)(0,1)\psline(0,2)(0,4)\psline(0,5)(0,7)}
 \rput{120}(0,12){\psline(0,0)(0,3)\psline(0,4)(0,7)} \rput{120}(0,13){\psline(0,0)(0,7)} \rput{0}(0,0){\psline(0,2)(0,3)\psline(0,4)(0,11)}\rput{0}(-\xI,0.5){\psline(0,2)(0,10)}
 \rput{0}(-\xII,1){\psline(0,1)(0,2)\psline(0,3)(0,10)} \rput{0}(-\xIII,1.5){\psline(0,1)(0,4)\psline(0,6)(0,9)} \rput{0}(-\xIV,2){\psline(0,0)(0,1)\psline(0,2)(0,4)\psline(0,5)(0,7)\psline(0,8)(0,9)}
 \rput{0}(-\xV,2.5){\psline(0,0)(0,2)\psline(0,3)(0,5)\psline(0,6)(0,8)} \rput{0}(-\xVI,3){\psline(0,0)(0,2)\psline(0,3)(0,7)} \rput{0}(-\xVII,3.5){\psline(0,0)(0,6)}
 \def\completion{\rput{-60}(0,0){\psline(0,0)(0,1)}\rput{-120}(0,1){\psline(0,1)(0,0)}}  \rput{0}(0,0){\completion} \rput{0}(0,1){\completion} \rput{0}(0,3){\completion}  \rput{0}(0,11){\completion}
 \rput{0}(0,12){\completion} \psline[linestyle=dotted](0,-0.5)(0,14) \uput[0](0,13.5){$\ell$} \pspolygon[fillstyle=solid,fillcolor=lightgray](-\xIII,5.5)(-\xIII,7.5)(-\xV, 6.5)}

 \def\semipav{ \rput{60}(0,0){\psline(0,0)(0,7)} \rput{60}(0,1){\psline(0,0)(0,3)\psline(0,4)(0,7)} \rput{60}(0,2){\psline(0,0)(0,1)\psline(0,2)(0,4)\psline(0,5)(0,7)} \rput{60}(0,3){\psline(0,0)(0,2)\psline(0,3)(0,6)}
 \rput{60}(0,4){\psline(0,0)(0,1)\psline(0,3)(0,6)} \rput{60}(0,5){\psline(0,0)(0,1)\psline(0,2)(0,4)\psline(0,5)(0,6)} \rput{60}(0,6){\psline(0,0)(0,3)\psline(0,4)(0,5)} \rput{60}(0,7){\psline(0,1)(0,3)\psline(0,4)(0,5)}
 \rput{60}(0,8){\psline(0,1)(0,4)} \rput{60}(0,9){\psline(0,0)(0,2)\psline(0,3)(0,4)} \rput{60}(0,10){\psline(0,0)(0,3)} \rput{60}(0,11){\psline(0,0)(0,2)} \rput{60}(0,12){\psline(0,0)(0,1)} \rput{120}(0,1){\psline(0,0)(0,1)}
 \rput{120}(0,2){\psline(0,0)(0,2)} \rput{120}(0,3){\psline(0,1)(0,3)} \rput{120}(0,4){\psline(0,0)(0,2)\psline(0,3)(0,4)} \rput{120}(0,5){\psline(0,1)(0,4)} \rput{120}(0,6){\psline(0,1)(0,3)\psline(0,4)(0,5)}
 \rput{120}(0,7){\psline(0,0)(0,3)\psline(0,4)(0,5)} \rput{120}(0,8){\psline(0,0)(0,1)\psline(0,2)(0,4)\psline(0,5)(0,6)} \rput{120}(0,9){\psline(0,0)(0,1)\psline(0,3)(0,5)\psline(0,6)(0,7)}
 \rput{120}(0,10){\psline(0,1)(0,2)\psline(0,3)(0,5)\psline(0,6)(0,7)} \rput{120}(0,11){\psline(0,0)(0,1)\psline(0,2)(0,4)\psline(0,5)(0,7)} \rput{120}(0,12){\psline(0,0)(0,3)\psline(0,4)(0,7)}
 \rput{120}(0,13){\psline(0,0)(0,7)} \rput{0}(0,0){\psline(0,2)(0,3)\psline(0,4)(0,6)\psline(0,7)(0,10)} \rput{0}(-\xI,0.5){\psline(0,2)(0,3)\psline(0,4)(0,8)\psline(0,9)(0,10)}
 \rput{0}(-\xII,1){\psline(0,1)(0,2)\psline(0,3)(0,5)\psline(0,6)(0,8)\psline(0,9)(0,10)} \rput{0}(-\xIII,1.5){\psline(0,1)(0,4)\psline(0,6)(0,9)}
 \rput{0}(-\xIV,2){\psline(0,0)(0,1)\psline(0,2)(0,4)\psline(0,5)(0,7)\psline(0,8)(0,9)} \rput{0}(-\xV,2.5){\psline(0,0)(0,2)\psline(0,3)(0,5)\psline(0,6)(0,8)}
 \rput{0}(-\xVI,3){\psline(0,0)(0,2)\psline(0,3)(0,7)} \rput{0}(-\xVII,3.5){\psline(0,0)(0,6)}\def\completion{\rput{-60}(0,0){\psline(0,0)(0,1)}\rput{-120}(0,1){\psline(0,1)(0,0)}}
 \rput{0}(0,0){\completion} \rput{0}(0,1){\completion} \rput{0}(0,3){\completion} \rput{0}(0,10){\completion} \rput{0}(0,11){\completion} \rput{0}(0,12){\completion}}

 \def\pathsemipav{ \semipav\pspolygon[fillstyle=solid,fillcolor=lightgray](0,6)(-\xI,6.5)(0,7)
 \def\mI{-0.43301} \def\mII{-1.29903} \def\mIII{-2.165063} \def\mIV{-3.031088} \def\mV{-3.8971143} \def\mVI{-4.7631397} \def\mVII{-5.629165126}
\psdots[dotsize=7pt](\mI,0.25)(\mII,0.75)(\mIII,1.25)(\mIV,1.75)(\mV,2.25)(\mVI,2.75)(\mVII,3.25)\psdots[dotsize=7pt](\mI,1.25)(\mI,3.25)(\mI,6.25)(\mI,10.25)(\mI,11.25)(\mI,12.25)
\psset{linestyle=dotted,linewidth=2.5pt}\rput{-60}(\mII,0.75){\psline(0,0)(0,1)}\psline(\mIII,1.25)(\mI,2.25)(\mI,3.25)\psline(\mIV,1.75)(\mIII,2.25)(\mIII,3.25)(\mI,4.25)(\mI,6.25)
\psline(\mV,2.25)(\mV,3.25)(\mIV,3.75)(\mIV,5.75)(\mII,6.75)(\mII,8.75)(\mI,9.25)(\mI,10.25)\psline(\mVI,2.75)(\mVI,4.75)(\mV,5.25)(\mV,6.25)(\mIII,7.25)(\mIII,9.25)(\mII,9.75)(\mII,10.75)(\mI,11.25)
\psline(\mVII,3.25)(\mVII,5.25)(\mVI,5.75)(\mVI,7.75)(\mV,8.25)(\mV,9.25)(\mIV,9.75)(\mIV,10.75)(\mI,12.25)}

\def\path{\psgrid[subgriddiv=1,griddots=5,gridwidth=0.5pt,gridlabels=0pt,dotsize=5pt,showpoints=true,ticks=none](-8,-1)(1,14)\psaxes[showpoints=true,ticks=none,labels=none]{->}(0,0)(-8,-1)(1,14)
 \psdots[dotsize=7pt](-1,1)(-2,2)(-3,3)(-4,4)(-5,5)(-6,6)(-7,7) \psdots[dotsize=7pt](-1,13)(-1,12)(-1,11)(-1,7)(-1,4)(-1,2)  \psline[linewidth=2pt](-7,7)(-7,9)(-6,9)(-6,11)(-5,11)(-5,12)(-4,12)(-4,13)(-1,13)
  \psline[linewidth=2pt](-6,6)(-6,8)(-5,8)(-5,9)(-3,9)(-3,11)(-2,11)(-2,12)(-1,12)  \psline[linewidth=2pt](-5,5)(-5,6)(-4,6)(-4,8)(-2,8)(-2,10)(-1,10)(-1,11)  \psline[linewidth=2pt](-4,4)(-3,4)(-3,5)(-1,5)(-1,7)
  \psline[linewidth=2pt](-3,3)(-1,3)(-1,4) \psline[linewidth=2pt](-2,2)(-1,2)}

\def\Pf{\operatorname{Pf}}
\def\s{\sigma}
\def\l{\lambda}
\def\Acal{\mathcal{A}}
\def\Scal{\mathcal{S}}
\def\Ical{\mathcal{I}}
\def\T{\mathcal{T}}

\def\inv{\rm inv}
\def\sgn{\rm sgn}
\def\Mtil{\widetilde M}

\def\Btil{\widetilde B}

\begin{document}

\title[centered symmetric  rhombus tilings of a hexagon]{Enumeration of symmetric centered\\ rhombus tilings of a hexagon}

\author[A. Kasraoui]{Anisse Kasraoui}
\author[C. Krattenthaler]{Christian Krattenthaler}
\address{Fakult\"at f\"ur Mathematik, Universit\"at Wien,
Nordbergstrasse 15,
A-1090 Vienna,
Austria}
\email{anisse.kasraoui@univie.ac.at, christian.krattenthaler@univie.ac.at}
\thanks{Research supported by the grant S9607-N13 from Austrian Science Foundation FWF
 in the framework of the National Research Network  ``Analytic Combinatorics and Probabilistic Number theory".}

\date{}

\begin{abstract}
A  rhombus tiling of a hexagon is said to be centered if it contains the central lozenge. We compute the number of vertically symmetric 
rhombus tilings of a hexagon with side lengths $a, b, a, a, b, a$ which are centered.  When $a$ is odd and $b$ is even, this shows that the 
probability  that a random vertically symmetric rhombus tiling of a $a, b, a, a, b, a$ hexagon is centered is exactly 
the same as the probability that a random rhombus tiling of a $a, b, a, a, b, a$ hexagon is centered. 
This also leads to a factorization theorem for the number of all rhombus tilings of a hexagon which are centered.
\end{abstract}

\maketitle

%%%%%%%%%%%%%%%%%%%%%%%%%%%%%%%%%%%%%%%%%%%%%%%%%%%%%%%%%%%%%%%%%%%%%%%%%%%%%%%%%%%%%%%%%%%%%%%%%%%%%%%%%%%%%%
%%%%%%%%%%%%%%%%%%%%%%%%%%%%%%%%%%%%%%%%%%%%%%%%%%%%%%%%%%%%%%%%%%%%%%%%%%%%%%%%%%%%%%%%%%%%%%%%%%%%%%%%%%%%%
%%%%%%%%%%%%%%%%%%%%%%%%%%%%%%%%%%%%%%%%%%%%%%%%%%%%%%%%%%%%%%%%%%%%%%%%%%%%%%%%%%%%%%%%%%%%%%%%%%%%%%%%%%%%
%
%                              INTRODUCTION
%
%%%%%%%%%%%%%%%%%%%%%%%%%%%%%%%%%%%%%%%%%%%%%%%%%%%%%%%%%%%%%%%%%%%%%%%%%%%%%%%%%%%%%%%%%%%%%%%%%%%%%%%%%%%%%
%%%%%%%%%%%%%%%%%%%%%%%%%%%%%%%%%%%%%%%%%%%%%%%%%%%%%%%%%%%%%%%%%%%%%%%%%%%%%%%%%%%%%%%%%%%%%%%%%%%%%%%%%%%%%
%%%%%%%%%%%%%%%%%%%%%%%%%%%%%%%%%%%%%%%%%%%%%%%%%%%%%%%%%%%%%%%%%%%%%%%%%%%%%%%%%%%%%%%%%%%%%%%%%%%%%%%%%%%%
\section{Introduction}

 The enumeration of plane partitions, equivalently of rhombus tilings of a hexagon,
was initiated by MacMahon in the early twentieth century. Let $a$, $b$ and $c$ be positive integers.
By an \textit{$(a,b,c)$ hexagon} we mean an equi-angular hexagon with side-lengths $a,b,c,a,b,c$. We
always draw such a hexagon with  the sides of lengths $a,b,c,a,b,c$ in clockwise order starting from
the southwestern side, so that the sides of length $b$ are vertical.
  From a classical result of MacMahon~\cite[Sect. 429, $q\to1$]{MacM}, we know that
the number of tilings of an $(a,b,c)$ hexagon by rhombi whose sides have length 1 and whose angles
measure~60 and~120 degrees (equivalently, of plane partitions contained in an $a\times b \times c$ box)
is given by the product
\begin{align}\label{eq:def P(a,b)}
 T(a,b,c)=\prod_{i=1}^a\prod_{j=1}^b\prod_{k=1}^c \frac{i+j+k-1}{i+j+k-2}.
\end{align}
We call such tilings \textit{rhombus tilings}. The picture on the left in Figure~\ref{fig:Til} shows a rhombus tiling of a $(3,5,4)$ hexagon.
MacMahon also conjectured that the number of \textit{vertically symmetric} rhombus tilings of an $(a,b,a)$ hexagon (i.e, that are invariant under reflection
across the vertical symmetry axis of the hexagon; e.g., see the picture on the right in Figure~\ref{fig:H-F}) is given by the product
\begin{align}\label{eq:def N(a,b)}
ST(a,b,a)=\prod_{i=1}^a \frac{2i+b-1}{2i-1} \prod_{1\leq i<j\leq a}
\frac{i+j+b-1}{i+j-1}.
\end{align}
 This was first proved by Andrews~\cite{Andrews}. Other proofs, and refinements, were later found by e.g. Gordon~\cite{Gordon},
Macdonald~\cite[pp. 83--85]{McDon}, Proctor~\cite[Prop. 7.3]{Proct}, Fischer~\cite{Ilse2}, and the second author of the present
paper~\cite[Theorem 13]{Kratt}.
 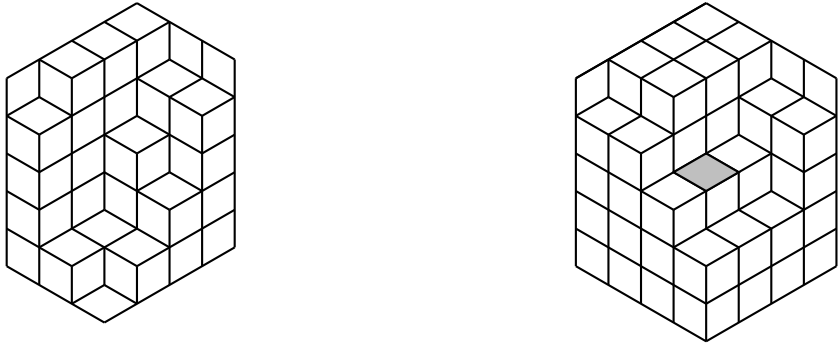
\begin{figure}[h!]\label{fig:Til}
\psset{xunit=0.5,yunit=0.5}
\begin{pspicture}(0,0.5)(18,9)
\rput{0}(0,1){\pav}
\rput{0}(16,0.5){\centpav}
\end{pspicture}
\caption{\textit{Left}: a rhombus tiling of a $(3,5,4)$ hexagon.
\textit{Right}: a centered rhombus tiling of a $(4,5,4)$ hexagon where the central lozenge is marked.}
\end{figure}

During the last two decades, there has been an increasing interest
in enumerating rhombus tilings of planar regions with holes. One of
the first result in this direction was the counting of rhombus
tilings of an $(a,b,a)$ hexagon that contain the central rhombus
(these tilings can be seen as the rhombus tilings of the region
obtained from an $(a,b,a)$ hexagon by removing its central rhombus).
We call such tilings \textit{centered}. Note that an $(a,b,a)$ hexagon
has a central rhombus only if~$a$ and~$b$ have opposite parity. The picture on the right in
Figure~\ref{fig:Til} shows a centered tiling of a $(4,5,4)$ hexagon.
The corresponding enumeration result (see Theorem~\ref{thm:Pcentered(2n+1,2x)} below)
is independently due to Ciucu and the second author~\cite[Theorems 1
and 2 and Corollaries 3 and 4]{CiuKrat-centered} and to Gessel and
Helfgott~\cite[Theorems 15 and 17]{GeHel}.
 For nonnegative integers~$n$ and~$x$, define
\begin{align}\label{eq:def Q(x,n)}
Q(n,x)&=\frac{1}{2}\frac{(2n)!^2(2x)!(x+2n-1)!}{(n)!^2(x)!(2x+4n-2)!}
       \left(\sum_{i=0}^{n-1} \frac{(-1)^{n-i-1}}{2n-2i-1} \frac{(x+n-i)_{2i}}{i!^2}\right),
\end{align}
where $(a)_k$ is for the Pochhammer symbol, defined by $(a)_k:=a(a+1)\cdots (a+k-1)$ for $k\geq 1$,
and $(a)_0:=1$.
\begin{thm}\label{thm:Pcentered(2n+1,2x)}
Let $n$ and~$x$ be two nonnegative integers.

\textit{(i)} For $x\geq1$, the number of centered rhombus tilings of
a $(2n+1,2x,2n+1)$ hexagon is \linebreak $Q(n+1,x)\cdot
T(2n+1,2x,2n+1),$ where $Q$ and $T$ are defined by~\eqref{eq:def
Q(x,n)} and~\eqref{eq:def P(a,b)}. Similarly, for $n\geq1$, the
number of centered rhombus tilings of a $(2n,2x+1,2n)$ hexagon is
$Q(n,x+1)\cdot T(2n,2x+1,2n)$.

\textit{(ii)} For $n\geq1$, exactly one third of the rhombus tilings of a $(2n+1,2n,2n+1)$
hexagon are centered. The same is true for a $(2n,2n+1,2n)$ hexagon.

\textit{(iii)} Let $a$ be a nonnegative real number. For $x\sim an$, the
probability that a random rhombus tiling of a $(2n+1,2x,2n+1)$
hexagon is centered is $\sim (2/\pi) \arcsin(1/(a+1))$ as $n$ tends
to infinity. The same is true for a $(2n,2x+1,2n)$ hexagon.
\end{thm}

 Generalizations of the preceding result were later obtained by Fulmek and the second author~\cite{FulKrat-I,FulKrat-II}, 
and Fisher~\cite{Ilse}. For other results on the enumeration of rhombus tilings of hexagons of which central triangles 
are removed, see e.g.~\cite{Ciu,Ciu-Zare,CiuKrat-Dual,Ther,KrattOkada}. 
Another result which is particularly relevant to our paper is the
one by Ciucu and the second author in~\cite{CiuKrat-interaction}
where for the first time the number of rhombus tilings of a half
$(a,b,a)$ hexagon with a triangular hole of size two and a free
boundary was computed. By a \textit{half $(a,b,a)$ hexagon} with a
\textit{free boundary} we mean the region, denoted in the rest of
the paper by $F_{a,b,a}$, obtained from the left  half of an
$(a,b,a)$ hexagon by regarding its boundary along~$\ell$, the
vertical symmetry axis of the hexagon, as free; i.e, rhombi in a
tiling of $F_{a,b,a}$ are allowed to protrude outward across~$\ell$
to its right. In Figure~\ref{fig:HalfTil}, the picture on the left 
shows a rhombus tiling of the half hexagon $F_{7,6,7}$, the other two show tilings of $F_{7,6,7}$ with 
triangular gaps of size one and two.
%%%%%%%%%%%%FIG 2
 \begin{figure}[h!]
\psset{xunit=0.5,yunit=0.5}
\begin{pspicture}(0,0)(22,15)
\rput{0}(3,0){\semipavf}
\rput{0}(13,0){\semipavb}
\rput{0}(23,0){\semipava }
\end{pspicture}
\caption{\textit{Left}: a rhombus tiling of the half hexagon $F_{7,6,7}$.
\textit{Middle}: a tiling of $F_{7,6,7}$ with a triangular hole of size two.
\textit{Right}: a tiling of $F_{7,6,7}$ with a triangular hole of size one.}\label{fig:HalfTil}
\end{figure}
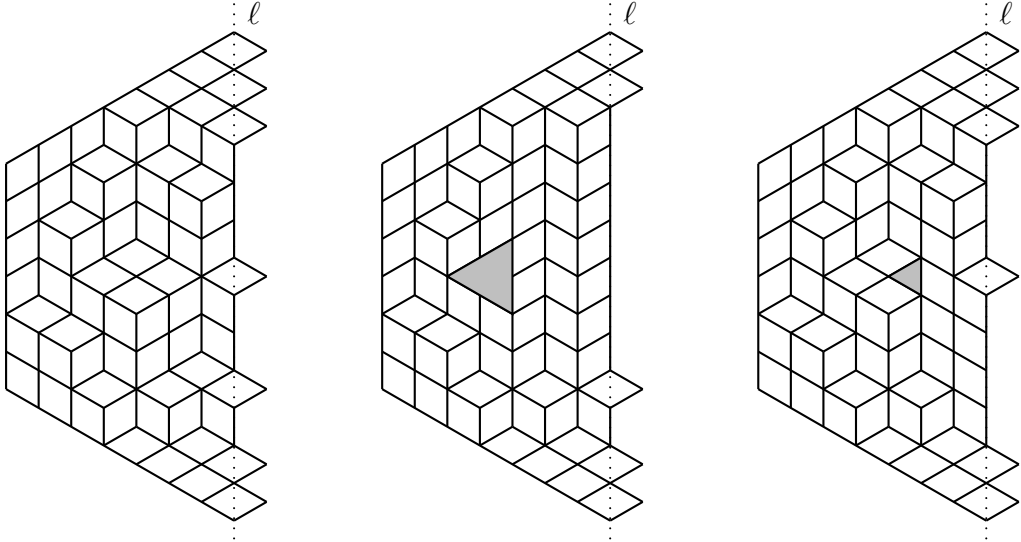

The present paper was motivated by an attempt to find similar
results to those obtained in~\cite{CiuKrat-interaction} for a
triangular gap of size one. While we didn't succeed for a general
position of the hole, we have been able to obtain a counting formula
for the number of rhombus tilings of the region $F_{a,b,a}^*$
obtained from the half hexagon $F_{a,b,a}$ by removing a triangular
hole of size one pointing to the left such that the center of its
right-side coincides with the center of the free boundary (this region is defined only if~$a$ and~$b$ have opposite parity).
The picture on the left in Figure~\ref{fig:H-F} shows a tiling of
$F_{7,6,7}^*$. As illustrated in Figure~\ref{fig:H-F}, by reflecting
the tilings of $F_{a,b,a}^*$ across the free boundary, it is easily seen that
these are equinumerous with the centered vertically symmetric
rhombus tilings of an $(a,b,a)$ hexagon.
%%%%%%%%%%%%FIG 2
 \begin{figure}[h!]
\begin{center}
\psset{xunit=0.5,yunit=0.5}
\begin{pspicture}(0,0)(24,15)
\rput{0}(5,0){\semipav  \psline[linestyle=dotted](0,-0.5)(0,14) \pspolygon[fillstyle=solid, fillcolor=lightgray] (0,6)(0,7)(-0.866, 6.5)}
\def\pavcomplet{ \semipav
\pnode(0,-1){V1} \pnode(0,10){V2} \symPlan(V1)(V2){\semipav}}
\rput{0}(20,0){\pavcomplet \pspolygon[fillstyle=solid, fillcolor=lightgray] (0,6)(0.866, 6.5)(0,7)(-0.866, 6.5)} \psline{<->}(7.5,6.5)(12,6.5)
\end{pspicture}
\caption{The correspondence between tilings of $F^*_{2n+1,2x,2n+1}$ and centered vertically symmetric tilings of a $(2n+1,2x,2n+1)$ hexagon.}\label{fig:H-F}
\end{center}
\end{figure}
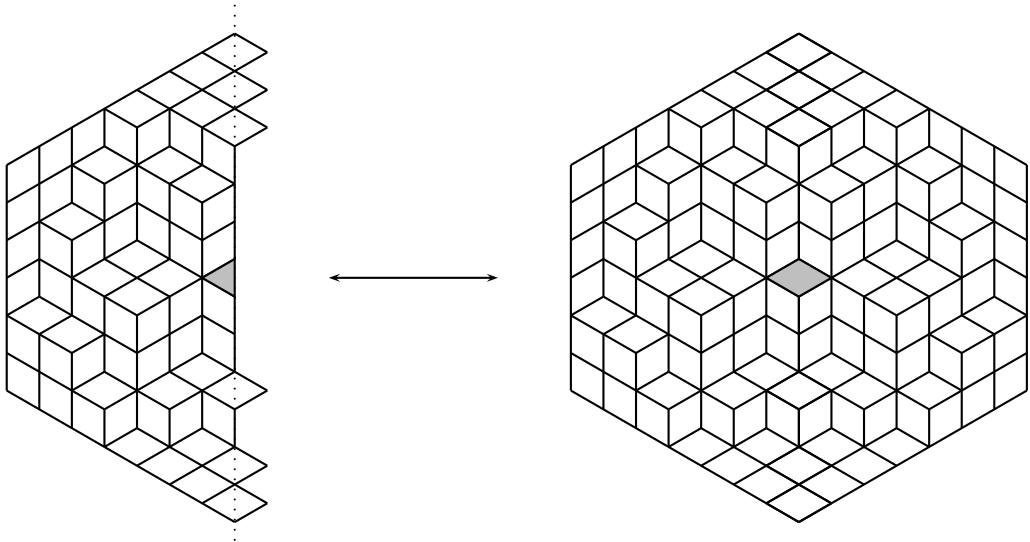
 The next two theorems, which can be seen as a ``symmetrization'' of
Theorem~\ref{thm:Pcentered(2n+1,2x)}(i), are our main results.

 \begin{thm}\label{thm:Ncentered(2n+1,2x)}
Let $n$ and $x$ be nonnegative integers. For $x\geq 1$, the number of centered
vertically symmetric rhombus tilings of a $(2n+1,2x,2n+1)$ hexagon is $Q(n+1,x)\cdot ST(2n+1,2x,2n+1)$,
where $Q$ and $ST$ are defined by~\eqref{eq:def Q(x,n)} and~\eqref{eq:def N(a,b)}.
\end{thm}

For positive integer $n$ and nonnegative integer $x$, define
\begin{align}
\begin{split}\label{eq:def-Ux}
U_n(x)&=\sum_{i=1}^n
\big((2n-1)!!+(-1)^{i+1}(2n)!!\big)\frac{\left(\frac{3}{2}-i\right)_{2n-1}}{(i-1)!(2n-i)!}\\
&\qquad\qquad \times
\big((x+1)_{i-1}(x+i+1)_{2n-i}-(x+1)_{2n-i}(x+2n+2-i)_{i-1}\big),
\end{split}
\end{align}
and
\begin{align}\label{eq:def R(x,n)}
R(n,x)&=2^{3n-2}\frac{(2x+2)!(x+2n)!}{(n)!(x+1)!(2x+4n)!}\,U_n(x).
\end{align}
Here, as usual, $a!!$ stands for the double factorial.

\begin{thm}\label{thm:Ncentered(2n,2x+1)}
Let $n$ and $x$ be nonnegative integers. For $n\geq 1$, the number of centered
vertically symmetric rhombus tilings of a $(2n,2x+1,2n)$ hexagon is  $R(n,x)\cdot ST(2n,2x+1,2n)$,
where $R$ and $ST$ are defined by~\eqref{eq:def R(x,n)} and~\eqref{eq:def N(a,b)}.
\end{thm}

 The next result, which is rather striking and deserves further  investigation, is immediate 
 from Theorems~\ref{thm:Ncentered(2n+1,2x)} and~\ref{thm:Pcentered(2n+1,2x)}.

\begin{cor}\label{cor:Pcentered(2n+1,2x)-Ncentered(2n+1,2x)}
Let $n$ and $x$ be  nonnegative integers. The probability that a
random vertically symmetric rhombus tiling of a $(2n+1,2x,2n+1)$
hexagon is centered is exactly the same as the probability that a
random rhombus tiling of a $(2n+1,2x,2n+1)$ hexagon is centered.
\end{cor}

 We should note here that the preceding result leads to an interesting (and intriguing) factorization
 for the number of all centered tilings of a hexagon. Given a planar region $R$ symmetric with respect to a
vertical axis and to a horizontal axis, let $\T(R)$ be the set of
all rhombus tilings of $R$. We also let $\T^{(|)}(R)$ (resp., $\T^{(-)}(R)$)
be the set of the tilings in $\T(R)$ that are vertically symmetric
(resp., horizontally symmetric). Let $H_{a,2b,a}$ be an $(a,2b,a)$
hexagon. Ciucu and the second author~\cite{CiuKrat-factor} observed the factorization 
\begin{align}\label{eq:FactorizationTil}
\# \T(H_{a,2b,a})&= \# \T^{(|)}(H_{a,2b,a}) \cdot \# \T^{(-)}(H_{a,2b,a}).
\end{align}
This can be proved by combining~\eqref{eq:def P(a,b)} and~\eqref{eq:def N(a,b)} with a formula of Proctor~\cite{Proct2} for $\# \T^{(-)}(H_{a,2b,a})$ 
(equivalently, the number of transpose complementary plane partitions in a $2b\times a\times a$ box); see also~\cite{CiuKrat-factor} where the above relation 
was put in a more general context. 
 Now, suppose that $a$ is odd and let $H^*_{a,2b,a}$ denote the region obtained  by
removing the central rhombus in $H_{a,2b,a}$. Recall that a centered tiling of $H_{a,2b,a}$ is obviously 
equivalent to a tiling of $H^*_{a,2b,a}$. Then, combining Corollary~\ref{cor:Pcentered(2n+1,2x)-Ncentered(2n+1,2x)} with~\eqref{eq:FactorizationTil} 
(and noting that any horizontally symmetric tiling of $H_{a,2b,a}$ is centered), we arrive at
the following factorization for the number of centered tilings.

\begin{cor}\label{cor:FactorizationCentTil}
For any nonnegative integers $n$ and $x$, we have
\begin{align}\label{eq:FactorizationCentTil}
\# \T(H^*_{2n+1,2x,2n+1})&= \# \T^{(|)}(H^*_{2n+1,2x,2n+1}) \cdot \#
\T^{(-)}(H^*_{2n+1,2x,2n+1}).
\end{align}
\end{cor}
%This is in line with a more general conjecture of Ciucu which, roughly speaking, asserts that the above factorization
%is valid  when $H^*_{2n+1,2x,2n+1}$ is replaced by any region obtained from a $(2n+1,2x,2n+1)$ hexagon by removing two unit triangular
%situated on the horizontal axis.
Note that the above factorization is very similar to the one in~\eqref{eq:FactorizationTil}.
Another result which is immediate from Corollary~\ref{cor:Pcentered(2n+1,2x)-Ncentered(2n+1,2x)}
and Theorem~\ref{thm:Pcentered(2n+1,2x)} is the following.

\begin{cor}\label{cor:Ncentered(2n+1,2x)}
$(i)$ For $n\geq 1$, the probability that a random vertically symmetric rhombus tiling of
a $(2n+1,2n,2n+1)$ hexagon is centered is $1/3$.

$(ii)$ Let $a$ be a nonnegative real number. For $x\sim an$, the
probability that a random vertically symmetric rhombus tiling of
a $(2n+1,2x,2n+1)$ hexagon is centered  is $\sim (2/\pi) \arcsin(1/(a+1))$, as $n$ tends to infinity.
\end{cor}

As one can expect, the second part of the preceding proposition is still valid
for a $(2n,2x+1,2n)$ hexagon.
\begin{cor}\label{cor:Ncentered(2n,2x+1)}
 Let $a$ be a nonnegative real number.  Then, for $x\sim an$, the probability
that a random vertically symmetric rhombus tiling of a $(2n,2x+1,2n)$ hexagon is centered is $\sim (2/\pi)
\arcsin(1/(a+1))$, as $n$ tends to infinity.
\end{cor}

The rest of this paper is devoted to the proof of the above corollary and Theorems~\ref{thm:Ncentered(2n+1,2x)}
 and~\ref{thm:Ncentered(2n,2x+1)}. As in many previous papers, our approach to proving
 Theorems~\ref{thm:Ncentered(2n+1,2x)}  and~\ref{thm:Ncentered(2n,2x+1)} is to first translate the centered vertically symmetric rhombus
 tilings into  families of non-intersecting lattice paths. Then, to enumerate these non-intersecting
 lattice paths,  we use a slight extension of a theorem of Stembridge~\cite{Stem} to obtain Pfaffians for the numbers
 we are interested in.  This is the subject of Section 2. The evaluation of these Pfaffians are presented in
 Section~3 and~4 with some auxiliary results proved in Section~5. It is based on the ``exhaustion/identification
 of factors'' method (e.g., see~\cite[Sect.~2.4]{Krattt}) and turns out to be particularly demanding. In particular, we need a
 Pfaffian factorization due to Ciucu and the second  author, an evaluation of a perturbed Mehta-Wang Pfaffian, and evaluations of
 very intricate combinatorial sums.   In the final section, Section 6, we perform the asymptotic calculation needed
  to derive Corollary~\ref{cor:Ncentered(2n,2x+1)} from Theorem~\ref{thm:Ncentered(2n,2x+1)}.

%%%%%%%%%%%%%%%%%%%%%%%%%%%%%%%%%%%%%%%%%%%%%%%%%%%%%%%%%%%%%%%%%%%%%%%%%%%%%%%%%%%%%%%%%%%%%%%%%%%%%%%%%%%%%%
%%%%%%%%%%%%%%%%%%%%%%%%%%%%%%%%%%%%%%%%%%%%%%%%%%%%%%%%%%%%%%%%%%%%%%%%%%%%%%%%%%%%%%%%%%%%%%%%%%%%%%%%%%%%%
%%%%%%%%%%%%%%%%%%%%%%%%%%%%%%%%%%%%%%%%%%%%%%%%%%%%%%%%%%%%%%%%%%%%%%%%%%%%%%%%%%%%%%%%%%%%%%%%%%%%%%%%%%%%
%
%                              rhombus tilings, nonintersecting lattice paths and Pfaffians
%
%%%%%%%%%%%%%%%%%%%%%%%%%%%%%%%%%%%%%%%%%%%%%%%%%%%%%%%%%%%%%%%%%%%%%%%%%%%%%%%%%%%%%%%%%%%%%%%%%%%%%%%%%%%%%
%%%%%%%%%%%%%%%%%%%%%%%%%%%%%%%%%%%%%%%%%%%%%%%%%%%%%%%%%%%%%%%%%%%%%%%%%%%%%%%%%%%%%%%%%%%%%%%%%%%%%%%%%%%%%
%%%%%%%%%%%%%%%%%%%%%%%%%%%%%%%%%%%%%%%%%%%%%%%%%%%%%%%%%%%%%%%%%%%%%%%%%%%%%%%%%%%%%%%%%%%%%%%%%%%%%%%%%%%%
%\newpage
\section{Centered vertically symmetric tilings, nonintersecting lattice paths and Pfaffians}

\subsection{Centered vertically symmetric tilings and nonintersecting lattice paths}\label{sect:tilings-paths}
As explained in the introduction (see Figure~\ref{fig:H-F}), centered vertically symmetric
tilings of an $(a,b,a)$ hexagon can be seen as tilings of the region $F^*_{a,b,a}$. Throughout this section, the term \textit{lattice path} will always
refer to a path in the lattice $\mathbb{Z}^2$ consisting of unit horizontal and vertical steps in the positive direction.
 There is a well known bijection between rhombus tilings of lattice regions
and families of non-intersecting lattice paths. An illustration in
our situation is given in Figure~\ref{fig:tilings-paths}.
%%%%%%%%%%%%FIG 3
 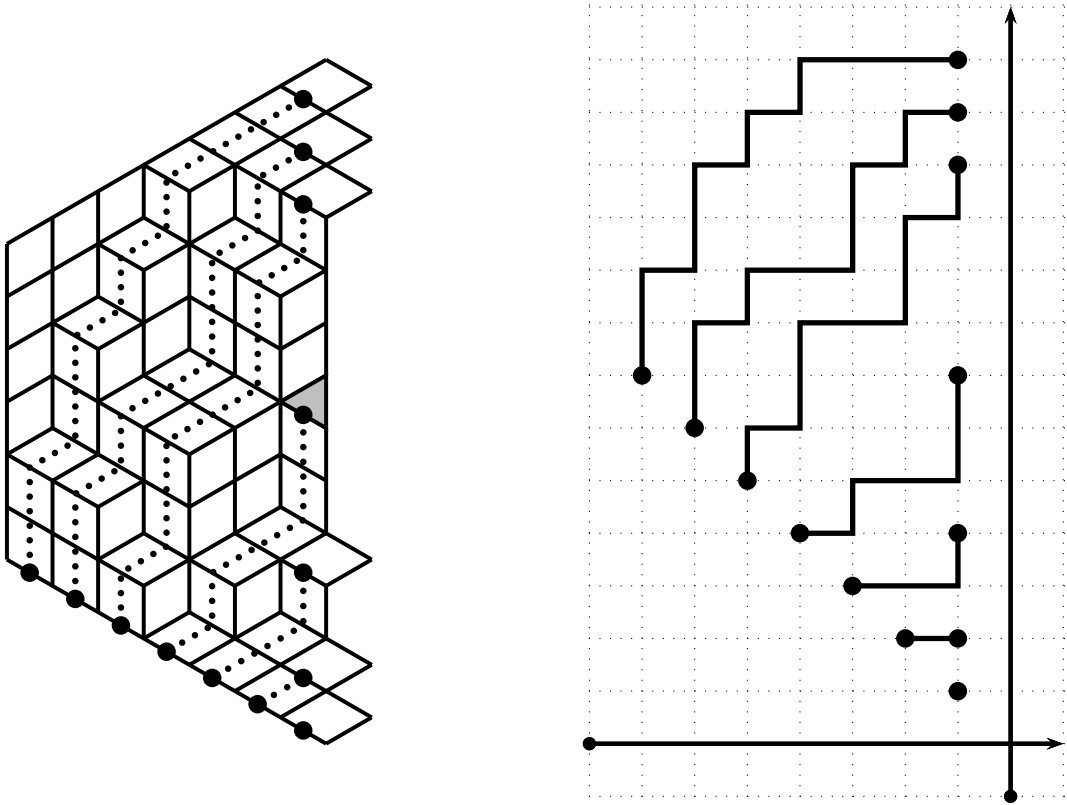
\begin{figure}[h!]
\begin{center}
\psset{xunit=0.7,yunit=0.7}
\begin{pspicture}(0,0)(20,17)
\psset{linewidth=1.5pt,griddots=10}
 \rput{0}(6,3){\pathsemipav}
\rput{0}(19,3){\path}
\end{pspicture}
\end{center}
\caption{The correspondence between tilings of $F^*_{n,x,n}$ and
 non-intersecting lattice paths.}\label{fig:tilings-paths}
\end{figure}
 By this bijection, tilings of $F^*_{2n+1,2x,2n+1}$ are seen to be equinumerous with families $(P_1, P_2, \ldots, P_{2n+1})$ of
non-intersecting lattice paths, where for $i = 1, 2,\ldots , 2n+1$, $P_i$ runs from~$(-i,i)$ to some point from the set $I=\{(-1,j):\,
1\leq j\leq 2x+2n+1\}$ with the additional condition that $(-1,x+n+1)$ must be an ending point of some path (we should note that empty path,
i.e., with no steps, can occur as in Figure~\ref{fig:tilings-paths} where the bottommost path is the empty path from $(-1,1)$ to $(-1,1)$).
Similarly, tilings of $F^*_{2n,2x+1,2n}$ are equinumerous with families $(P_1, P_2, \ldots, P_{2n})$ of non-intersecting lattice
paths, where $P_i$ runs from $(-i, i)$ to some point from the set $I$ with the same additional condition.
To enumerate these non-intersecting lattice paths, we shall use a slight extension of
a theorem of Stembridge~\cite[Theorem 3.2]{Stem}.

\subsection{Nonintersecting lattice paths and Pfaffians}\label{sect:paths-Pfaffian}
We first set up some terminology.  The \textit{Pfaffian} of a
skew-symmetric matrix $A$ will be denoted $\Pf A $. It is well-known
(see e.g.~\cite[Proposition~2.2]{Stem}) that
\begin{align}\label{eq:pf-det}
 (\Pf A)^2=\det A.
\end{align}
The \textit{signature} of a permutation $\pi\in\mathfrak{S}_n$,
where $\mathfrak{S}_n$ is the symmetric group  of degree $n$, will
be denoted $\sgn\,\pi$. Recall that
\begin{align}\label{eq:sign}
\sgn\,\pi=(-1)^{\inv\, \pi},
\end{align}
where $\inv\, \pi$ is the number of \textit{inversions} of $\pi$; i.e.,
the number of pairs $1\leq i<j\leq n$ such that $\pi(i)>\pi(j)$.

Given a weight-function $w$ that assigns values in the field of
complex numbers $\mathbb{C}$ to each edge of the integer lattice
$\mathbb{Z}^2$, we extend $w$ multiplicatively to multisets of
edges, so that $w(M)=\prod_{e\in M} w(e)$ for any such multiset $M$.
 The weight of a lattice path or $r$-tuple of lattice paths
is defined as the weight of the underlying multiset of edges. Given
any family $\mathcal{F}$ of edge multisets, we will write
$GF[\mathcal{F}]$ for the generating function according to the
weight $w$;
 i.e., $GF[\mathcal{F}]=\sum_{M\in \mathcal{F}} w(M)$. In particular, let us define
\begin{align}\label{eq:def h(A,B)}
h(A,B)=GF[\mathcal{P}(A\rightarrow B)]=\sum_{P\in
\mathcal{P}(A\rightarrow B)} w(P),
\end{align}
where $\mathcal{P}(A\rightarrow B)$ stands for the set of lattice paths
from $A$ to $B$.

\begin{thm}\label{thm:Paths-Pfaffians}
Let $\Acal=(A_1,A_2,\ldots ,A_p)$, $\Scal=(S_1, S_2, \ldots , S_q)$ and $\Ical=
(I_1, I_2,\ldots)$ be finite lists of points in the lattice $\mathbb{Z}^2$, with $p+q$ even.
For $\pi\in \mathfrak{S}_p$, let $\mathcal{P}^{nonint} (\Acal_{\pi}\rightarrow (\Scal,\Ical))$ denote the set of
families $(P_1, P_2,\ldots, P_p)$ of non-intersecting lattice paths with $P_k$ running from
$A_{\pi(k)}$ to $S_k$, for $k = 1, 2, \ldots, q$, and to $I_{j_k}$,
for $k = q + 1, q + 2,\ldots, p$, the indices being required to
satisfy $j_{q+1} < j_{q+2} <\ldots < j_p$. Then
\begin{align}\label{eq:Pfaffian-nonintersecting}
\Pf\left(
    \begin{array}{cc}
      Q & H \\
      -H^t & 0 \\
    \end{array}
  \right)&=(-1)^{{q\choose 2}} \sum_{\pi\in \mathfrak{S}_p} ({\rm sgn}\,
\pi)\; GF[\mathcal{P}^{nonint} (\Acal_{\pi}\rightarrow (\Scal,\Ical))],
\end{align}
where the matrix $Q = (Q_{i,j})_{1\leq i,j\leq p}$ is defined by
\begin{align}\label{eq:defQij}
Q_{i,j}=\sum_{1\leq s<t} h(A_i,I_s)\, h(A_j,I_t) - h(A_j,I_s)\,
h(A_i,I_t),
\end{align}
and the matrix $H=(H_{i,j})_{1\leq i\leq p, 1\leq j\leq q}$ by
\begin{align}\label{eq:defHij}
H_{i,j}& = h(A_i,S_j),
\end{align}
with $h(A,B)$ defined by~\eqref{eq:def h(A,B)}.
\end{thm}

 The preceding theorem is a slight extension of a result of Stembridge~\cite[Theorem 3.2]{Stem} and is just
a weighted version of~Theorem~5 in~\cite{CiuKrat-interaction}. Since
the proof of Stembridge's Theorem (alternatively, we can use the
minor summation formula of Ishikawa and
Wakayama~\cite[Theorem~2]{Ishikawa}) is easily adapted with only a
little extra effort to the preceding result, the proof details are
omitted.

%\newpage
%%%%%%%%%%%%%%%%%%%%%%%%%%%%%%%%%%%%%%%%%%%%%%%%%%%%%%%%%%%%%%%%%%%%%%%
%%%%%%%%%%%%%%%%%%%%%%%%%%%%%%%%%%%%%%%%%%%%%%%%%%%%%%%%%%%%%%%%%%%%%%%
\subsection{Centered vertically symmetric tilings and Pfaffians}
%%%%%%%%%%%%%%%%%%%%%%%%%%%%%%%%%%%%%%%%%%%%%%%%%%%%%%%%%%%%%%%%%%%%%%%
%%%%%%%%%%%%%%%%%%%%%%%%%%%%%%%%%%%%%%%%%%%%%%%%%%%%%%%%%%%%%%%%%%%%%%%
Combining the non-intersecting lattice paths interpretation of centered vertically symmetric
rhombus tilings with Theorem~\ref{thm:Paths-Pfaffians}, we obtain the following counting
Pfaffian formulas.

\begin{lem}\label{lem:tilings-Pfaffian_casA}
Let $n$ be a fixed nonnegative integer. Then, for any positive integer $x$,
the number of centered vertically symmetric rhombus tilings of a $(2n+1,2x,2n+1)$ hexagon
is equal to the Pfaffian of the skew-symmetric
matrix $M(x)= (M_{i,j}(x))_{1\leq i,j\leq 2n+2}$ defined by
\begin{align}\label{eq:def-Mij}
M_{i,j}(x)&=R_{i,j}(x)+T_{i,j}(x)-T_{j,i}(x),\quad 1\leq i,j\leq 2n+1,\\
M_{i,2n+2}(x)&=\binom {x+n}{i-1},\quad 1\leq i\leq 2n+1,\label{eq:def-Mij-Lr}
\end{align}
where $R_{i,j}(x)$ and $T_{i,j}(x)$ are defined by
\begin{align}
R_{i,j}(x)&:=\sum_{t=1}^{2x+2n+1}\frac{j-i}{t}{t\choose i}{t\choose
j}=\sum_{t=1}^{2x+2n+1}{t\choose i}{t-1\choose j-1}-{t-1\choose
i-1}{t\choose j},\label{eq:def-Rij}\\
T_{i,j}(x)&:={2x+2n+1\choose i}\left({x+n\choose j}+{x+n+1\choose j}
\right).\label{eq:def-Tij}
\end{align}
\end{lem}

\begin{lem}\label{lem:tilings-Pfaffian_casB}
Let $n$ be a fixed positive integer. Then, for any nonnegative
integer $x$,  the number of centered vertically symmetric rhombus
tilings of a $(2n,2x+1,2n)$ hexagon  is equal to the Pfaffian of the
skew-symmetric matrix $N(x)= (N_{i,j}(x))_{1\leq i,j\leq 2n+2}$
defined by
\begin{align}
N_{i,j}(x)&=M_{i,j}(x)=R_{i,j}(x)+T_{i,j}(x)-T_{j,i}(x),\quad 1\leq i,j\leq 2n+1,\label{eq:def-Nij}\\
N_{i,2n+1}(x)&={2n+2x+1 \choose i}-{n+x \choose i}-{n+x+1 \choose i},\quad 1\leq i\leq 2n,\label{eq:Qij_B_lastcolumn}\\
N_{i,2n+2}(x)&=M_{i,2n+2}(x)=\binom {x+n}{i-1},\quad 1\leq i\leq 2n,\quad N_{2n+1,2n+2}(x)=0,\label{eq:def-Nij-Lr}
\end{align}
where $M_{i,j}(x)$, $R_{i,j}(x)$ and $T_{i,j}(x)$ are defined as in Lemma~\ref{lem:tilings-Pfaffian_casA}.
\end{lem}

\begin{rmk}\label{rmk:Mn vs Nn}
 The matrices $M(x)$ and $N(x)$ differ only in the $(2n+1)$-th row and $(2n+1)$-th column.
\end{rmk}

 Before we turn to the proof of these two lemmas, we
provide an alternative expression for $R_{i,j}(x)$ which has the
advantage to be polynomial in $x$:
\begin{align}\label{eq:def-Rij-bis}
R_{i,j}(x)&=\sum_{\ell=0}^{i-1}\frac{j-i}{i}{j-1\choose
i-1-\ell}{\ell+j\choose \ell}{2x+2n+2\choose \ell+j+1}.
\end{align}
 The proof that the above relation is equivalent to~\eqref{eq:def-Rij} (for $x\geq 0$) amounts to a routine computation
 (see e.g. the proof of Equation~(4.13) in~\cite{CiuKrat-interaction}).
 We should also notice that the equivalence of these two relations
 also hold for $x<0$ once we interpret sums by
\begin{align}
 \sum_{k=m}^{n-1}\textrm{Expr}(k)=\left\{
           \begin{array}{ll}
             \sum_{k=m}^{n-1}\textrm{Expr}(k), & \hbox{if $n>m$,} \\
             0, & \hbox{if $n=m$,}\\
             -\sum_{k=n}^{m-1}\textrm{Expr}(k), & \hbox{if $n<m$,}
           \end{array}
         \right.
\end{align}
as in~\cite[Section~5]{CiuKrat-interaction}. With what we have said,
the expression~\eqref{eq:def-Rij} for $R_{i,j}(x)$ makes sense for
negative integers $x$ also and is equal to the expression
in~\eqref{eq:def-Rij-bis}. We will make use of this fact later in
this paper.\\

\noindent\textit{Proof of Lemma~\ref{lem:tilings-Pfaffian_casA}.}
 Let $T(F^*_{2n+1,2x,2n+1})$ denote the number of rhombus tilings of the region $F^*_{2n+1,2x,2n+1}$.
Let $\Acal$, $\Scal$ and $\Ical$ be the list of points in $\mathbb{Z}^2$ defined
by
\begin{align}\label{eq:def-SandI}
\begin{split}
 &\Acal=(A_1,A_2,\ldots,A_{2n+1})\quad\text{with $A_i=(-i, i)$},\\
 &\Scal=(I_{n+x+1})\quad\text{and}\quad\Ical=(I_1, I_2,\ldots,I_{2x+2n+1})\setminus(I_{n+x+1})\quad\text{with $I_j=(-1,j)$}.
\end{split}
\end{align}
It follows from what we have said in Section~\ref{sect:tilings-paths} that
\begin{align}\label{eq:T(Hnx)}
T(F^*_{2n+1,2x,2n+1})=\sum_{\pi\in \mathfrak{S}_{2n+1}}
\#\mathcal{P}^{nonint} (\Acal_{\pi}\rightarrow (\Scal,\Ical)),
\end{align}
where $\mathcal{P}^{nonint} (\Acal_{\pi}\rightarrow (\Scal,\Ical))$ is defined
as in Theorem~\ref{thm:Paths-Pfaffians}. Therefore,
by~\eqref{eq:Pfaffian-nonintersecting}, in order to express
$T(H_{2n+1,2x})$ in term of a Pfaffian, it suffices to find a weight-function that assigns ${\rm sgn}\,\pi$ to each family of paths in
$\mathcal{P}^{nonint} (\Acal_{\pi}\rightarrow (\Scal,\Ical))$ for
$\pi\in\mathfrak{S}_{2n+1}$. To this end, it shall be convenient to
slightly modify the list $\Acal$ (because of the possible occurrence of empty paths,
as in Figure~\ref{fig:tilings-paths}, which always have weight 1).
Let $\widetilde\Acal$ be defined by
\begin{align}\label{eq:def-Amod}
 \widetilde\Acal=(\widetilde A_1,\widetilde A_2,\ldots,\widetilde A_{2n+1})\quad
 \text{with $\widetilde A_1=(-1, 0)$ and $\widetilde A_i=A_i=(-i, i)$ for $i\geq 2$}.
\end{align}
 It is easily checked that
 $\#\mathcal{P}^{nonint} (\Acal_{\pi}\rightarrow (\Scal,\Ical))=\#\mathcal{P}^{nonint} (\widetilde\Acal_{\pi}\rightarrow (\Scal,\Ical))$
for any $\pi$ in~$\mathfrak{S}_{2n+1}$. So, by~\eqref{eq:T(Hnx)}, we have
\begin{align}\label{eq:T(Hnx)-mod}
T(F^*_{2n+1,2x,2n+1})=\sum_{\pi\in \mathfrak{S}_{2n+1}}
\#\mathcal{P}^{nonint} (\widetilde\Acal_{\pi}\rightarrow (\Scal,\Ical)).
\end{align}
Let $\pi=\pi_1\pi_2\ldots\pi_{2n+1}$  be a permutation in
$\mathfrak{S}_{2n+1}$, written in one-line notation, such that
$\mathcal{P}^{nonint}(\widetilde\Acal_{\pi}\rightarrow
(\Scal,\Ical))$ is non-empty. From the definition of the $\widetilde
A_i$'s and $I_j$'s, it is not difficult to see that the sequence
$\pi_2\pi_3\cdots\pi_{2n+1}$ must be increasing, and thus the number
of inversions of $\pi$, $\inv\,\pi$, is equal to $\pi_1-1$. This
implies that, for any family $\left(P_1, P_2,\ldots, P_{2n+1}\right)$ in
$\mathcal{P}^{nonint}(\widetilde\Acal_{\pi}\rightarrow
(\Scal,\Ical))$,  $\inv\,\pi$  is the number of paths $P_i$ that
ends at a point $I_j$ with $j<n+x+1$. So, if $w$ is a weight-function
on the edges of the lattice $\mathbb{Z}^2$ such that the
weight of a path $P$ from $\widetilde A_i$ to $I_j$ satisfies
\begin{align}
  w(P)=\left\{
  \begin{array}{ll}
    +1, & \hbox{$j\geq n+x+1$,} \\
    -1, & \hbox{$j<n+x+1$,}
  \end{array}
\right.\label{eq:def_weight1}
\end{align}
we would obtain from~\eqref{eq:T(Hnx)-mod},~\eqref{eq:sign} and what
we have said that
\begin{align}\label{eq:T(Hnx)-bis}
T(F^*_{2n+1,2x,2n+1})=\sum_{\pi\in \mathfrak{S}_{2n+1}} ({\rm sgn}\,\pi)\;
GF[\mathcal{P}^{nonint} (\widetilde\Acal_{\pi}\rightarrow (\Scal,\Ical)],
\end{align}
where $GF$ is the generating function according to the weight $w$,
as defined in Section~\ref{sect:paths-Pfaffian}. Note that the
construction of such a weight-function $w$ is easy: for instance, we
can assign to each edge of the lattice~$\mathbb{Z}^2$ the value $1$
except for the edges $\{(-1,0),(-1,1)\}$, $\{(-2,j),(-1,j)\}$ for
$2\leq j<n+x+1$, and $\{(-1,n+x),(-1,n+x+1)\}$, to which we assign
the value $-1$. So, we can assume that~\eqref{eq:T(Hnx)-bis} is true
for some weight-function $w$ satisfying~\eqref{eq:def_weight1}.

 Let $h(\widetilde A_i,I_j)$ be defined as in~\eqref{eq:def h(A,B)} and $P(\widetilde A_i, I_j)$
 denote the number of lattice paths from $\widetilde A_i$ to $I_j$. Then,
 by~\eqref{eq:def_weight1}, we have $h(\widetilde
A_i,I_j)=\epsilon P(\widetilde A_i, I_j)$ where $\epsilon=-1$ if $j<
n+x+1$ and $\epsilon=1$ otherwise. This, combined with the well
known counting of lattice paths, implies that
\begin{align}\label{eq:NumbLatticePaths}
h(\widetilde A_i, I_j)=(-1)^{\chi\left(j<  n+x+1\right)}\binom {j-1}{i-1},
\end{align}
where, as usual, for a claim $F$, $\chi(F)$ is 1 if $F$ is true and
0 otherwise. 
 Combining~\eqref{eq:T(Hnx)-bis} with Theorem~\ref{thm:Paths-Pfaffians}, we see that
\begin{align}\label{eq:def-MatTil1}
T(F^*_{2n+1,2x,2n+1})=\Pf\left(
    \begin{array}{cc}
      Q& H \\
      -H^t & 0 \\
    \end{array}
  \right),
\end{align}
where $Q= (Q_{i,j})_{1\leq i,j\leq 2n+1}$ is the matrix defined by
\begin{align}\label{eq:def_Qij}
Q_{i,j}=\sum_{1\leq s<t\leq 2x+2n+1\atop s,t\neq n+x+1}h(\widetilde
A_i,I_s)\, h(\widetilde A_j,I_t) - h(\widetilde A_j,I_s)\, h(\widetilde
A_i,I_t),
\end{align}
and $H=(H_{i,1})_{1\leq i\leq 2n+1}$ is the column vector defined by
\begin{align}\label{eq:def_Hi1}
H_{i,1}= h(\widetilde A_i,I_{x+n+1})=\binom
{x+n}{i-1},
\end{align}
where the last equality follows from~\eqref{eq:NumbLatticePaths}. So, to complete the proof of the lemma,
it suffices to show that the matrix entry $Q_{i,j}$ in~\eqref{eq:def_Qij} is equal
to the right-hand side of\eqref{eq:def-Mij}.

Let us define
\begin{align}\label{eq:def_Pij}
P_{i,j}(s,t) &={s-1\choose i-1}{t-1\choose j-1}-{s-1\choose
j-1}{t-1\choose i-1}.
\end{align}
Then, combining~\eqref{eq:def_Qij} with~\eqref{eq:NumbLatticePaths}, we obtain
\begin{align}\label{eq:decompo-Qij}
Q_{i,j}=\sum_{1\leq s<t\leq x+n}P_{i,j}(s,t)+\sum_{x+n+2\leq s<t\leq
2x+2n+1}P_{i,j}(s,t)-\sum_{s=1}^{n+x}\sum_{t=n+x+2}^{2n+2x+1}P_{i,j}(s,t).
\end{align}
By~\eqref{eq:def_Pij}, the last sum in the preceding relation can be evaluated as follows:
\begin{align}
\sum_{s=1}^{n+x}\sum_{t=n+x+2}^{2n+2x+1}P_{i,j}(s,t)
&=\sum_{s=1}^{n+x}{s-1\choose i-1}
\sum_{t=n+x+2}^{2n+2x+1}{t-1\choose j-1}-\sum_{s=1}^{n+x}
{s-1\choose j-1}\sum_{t=n+x+2}^{2n+2x+1} {t-1\choose i-1}\nonumber\\
\begin{split}\label{eq:SumPij eq1}
&={n+x\choose i}\left({2n+2x+1\choose j}-{n+x+1\choose j}\right)\\
&\quad- {n+x\choose j}\left({2n+2x+1\choose i}-{n+x+1\choose
i}\right)
\end{split}\\
\begin{split}\label{eq:SumPij eq2}
&={n+x\choose i}{2n+2x+1\choose j}-{n+x\choose j}{2n+2x+1\choose i}\\
&\quad+\frac{i-j}{n+x+1}{n+x+1\choose i}{n+x+1\choose j},
\end{split}
\end{align}
where we used the well-known identity
\begin{align}\label{eq:BinomialSum}
\sum_{\ell=m}^{n}{\ell\choose i}={n+1\choose i+1}-{m\choose i+1}
\end{align}
to obtain~\eqref{eq:SumPij eq1}. Similarly,
using~\eqref{eq:BinomialSum}, we can prove (for details, see
e.g.~\cite[Equation~(4.10)]{CiuKrat-interaction} where the case
$M=1$ is treated) that
\begin{align}
\sum_{M\leq s<t\leq N}P_{i,j}(s,t)&=\sum_{M\leq s<t\leq N}\left(
{s-1\choose i-1}{t-1\choose j-1}- {s-1\choose j-1}{t-1\choose
i-1}\right)\nonumber\\
&=\sum_{t=M}^N \frac{j-i}{t}{t\choose i}{t\choose j}\;+{N\choose
i}{M-1\choose j}- {N\choose j}{M-1\choose i}.\label{eq:Pij from M to
N}
\end{align}
Finally, after plugging~\eqref{eq:SumPij eq2} and~\eqref{eq:Pij from M to N} into~\eqref{eq:decompo-Qij},
it is easily checked that the matrix entry $Q_{i,j}$ is equal to the right-hand side of
the relation~\eqref{eq:def-Mij}. This ends the proof. \qed\\

%%%%%%%%%%%%%%%%%%%%%%%%%%%%%%%%%%%%%%%%%%%%%%%%%%%%%%%%%%%%%%%%%%%%%%%
%%%%%%%%%%%%%%%%%%%%%%%%%%%%%%%%%%%%%%%%%%%%%%%%%%%%%%%%%%%%%%%%%%%%%%%
%\subsection{Centered rhombus tilings of $F_{2n,2x+1}$ and Pfaffians}
%%%%%%%%%%%%%%%%%%%%%%%%%%%%%%%%%%%%%%%%%%%%%%%%%%%%%%%%%%%%%%%%%%%%%%%
%%%%%%%%%%%%%%%%%%%%%%%%%%%%%%%%%%%%%%%%%%%%%%%%%%%%%%%%%%%%%%%%%%%%%%%

\noindent\textit{Proof of Lemma~\ref{lem:tilings-Pfaffian_casB}.}
Using the correspondence illustrated in
Figure~\ref{fig:tilings-paths} and the same reasoning as in the
proof of~\eqref{eq:T(Hnx)-mod}, we see that $T(F^*_{2n,2x+1,2n})$,  the
number of tilings of the region $F^*_{2n,2x+1,2n}$, is equal to the number of families $(P_1,
P_2, \ldots, P_{2n})$ of non-intersecting lattice paths, where for
$i= 1, 2,\ldots,2n$, $P_i$ runs from $\widetilde A_i$ to some $I_j$,
with the additional condition that $I_{n+x+1}$ must be an ending
point of some path. Here, the $\widetilde A_i$'s and $I_j$'s have
the same meaning as in~\eqref{eq:def-Amod} and~\eqref{eq:def-SandI}.
To be able to use Theorem~\ref{thm:Paths-Pfaffians}, we shall need
to introduce a phantom vertex (see the remark after Theorem~3.1
in~\cite{Stem}) to fulfill the parity condition on $p+q$ in
Theorem~\ref{thm:Paths-Pfaffians}.

Let $\Acal'=(A'_1,A'_2,\ldots,A'_{2n+1})$, $\Scal'=(I_{n+x+1}')$ and $\Ical'=(I_1', I_2',\ldots,I_{2x+2n+2}')\setminus(I_{n+x+1}')$
be the list of points in $\mathbb{Z}^2$ defined by
\begin{align}\label{eq:def-S'andI'}
\begin{split}
 A'_1&=\widetilde A_1=(-1, 0),\quad   A'_i=\widetilde A_i=(-i, i)\quad\text{for $1\leq i\leq 2n$},\quad A'_{2n+1}=(0,-1),\\
 I'_j&=I_j=(-1,j)\quad\text{for $1\leq j\leq 2x+2n+1$},\quad I'_{2x+2n+2}=A'_{2n+1}=(0,-1).
\end{split}
\end{align}
Then, from what we have said in the
preceding paragraph it is not hard to show that
\begin{align}\label{eq:T(Hnx)'}
T(F^*_{2n,2x+1,2n})=\sum_{\pi\in \mathfrak{S}_{2n+1}}
\#\mathcal{P}^{nonint} (\Acal'_{\pi}\rightarrow (\Scal',\Ical')).
\end{align}
Let $\pi=\pi_1\pi_2\ldots\pi_{2n+1}$ be a permutation in
$\mathfrak{S}_{2n+1}$  such that $\mathcal{P}^{nonint}(\Acal'_{\pi}\rightarrow (\Scal',\Ical'))$
is non-empty. Then, from the definition of the $A_i'$'s and $I_j'$'s, it is not difficult to
see that the sequence $\pi_2\pi_3\cdots\pi_{2n+1}$ must be increasing and $\pi_{2n+1}=2n+1$ (for, if $\left(P_1,\ldots,
P_{2n+1}\right)\in P^{nonint}(\Acal'_{\pi}\rightarrow (\Scal',\Ical'))$, then
$P_{2n+1}$ must be the empty path from $A'_{2n+1}$ to
$I'_{2x+2n+2}$). Consequently, $\inv\,\pi$ is equal to $\pi_1-1$, and thus is also, for any family $\left(P_1,
P_2,\ldots, P_{2n+1}\right)$ in $P^{nonint}(\Acal'_{\pi}\rightarrow (\Scal',\Ical'))$, the
number of paths $P_i$ such that $P_i$ ends at a point $I'_j$ with
$1\leq j<n+x+1$. As it is done in the proof of Lemma~\ref{lem:tilings-Pfaffian_casA}, it is easy to find a weight-function
$w'$ on the edges of $\mathbb{Z}^2$ such that the weight of a path $P$ from $A'_i$ to $I'_j$ satisfies
\begin{align}
 w'(P)=\left\{
  \begin{array}{ll}
    +1, & \hbox{if $j\geq n+x+1$,} \\
    -1, & \hbox{if $1\leq j<n+x+1$.}
  \end{array}
\right.\label{eq:def_weight2}
\end{align}
We omit the details. It follows from~\eqref{eq:T(Hnx)'},~\eqref{eq:sign} and 
what we have said that
\begin{align}\label{eq:T(Hnx)-bis'}
T(F^*_{2n,2x+1,2n})=\sum_{\pi\in \mathfrak{S}_{2n+1}} ({\rm sgn}\,\pi)\;
GF[\mathcal{P}^{nonint} (\Acal'_{\pi}\rightarrow (\Scal',\Ical'))].
\end{align}

 Let $P(A_i', I_j')$ denote the number of lattice paths from $A_i'$ to $I_j'$.
 Similarly to~\eqref{eq:NumbLatticePaths},
we have $h(A_i',I_j')=\epsilon P(A_i', I_j')$ where $\epsilon=-1$ if
$j< n+x+1$ and $\epsilon=1$ otherwise, whence
\begin{align}\label{eq:NumbLatticePaths-b}
h(A_i', I_j')=\left\{\begin{array}{ll}
     (-1)^{\chi\left(j<  n+x+1\right)}\binom {j-1}{i-1}, & \hbox{if $i\neq 2n+1$ and $j\neq 2x+2n+2$,} \\
     1, & \hbox{if $(i,j)=(2n+1,2x+2n+2)$,} \\
     0, & \hbox{otherwise.}
  \end{array}\right.
\end{align}
Combining~\eqref{eq:T(Hnx)-bis'} with Theorem~\ref{thm:Paths-Pfaffians} and~\eqref{eq:NumbLatticePaths-b},
we obtain that
\begin{align}\label{eq:defMat1}
T(F^*_{2n,2x+1,2n})=\Pf\left(
    \begin{array}{cc}
      Q& H \\
      -H^t & 0 \\
    \end{array}
  \right),
\end{align}
where $Q= (Q_{i,j})_{1\leq i,j\leq 2n+1}$ is the matrix defined by
\begin{align}\label{eq:def_Qij-b}
Q_{i,j}=\sum_{1\leq s<t\leq 2x+2n+2\atop s,t\neq
n+x+1}h(A_i',I_s')\, h(A_j',I_t') - h(A_j',I_s')\, h(A_i',I_t'),
\end{align}
and $H=(H_{i,1})_{1\leq i\leq 2n+1}$ is the column vector defined by
\begin{align}\label{eq:def_Hi1-b}
H_{i,1}= h(A_i',I_{x+n+1}')=\chi(i\neq 2n+1)\binom {x+n}{i-1},
\end{align}
where the last equality follows from~\eqref{eq:NumbLatticePaths-b}.

Comparing~\eqref{eq:NumbLatticePaths-b}
with~\eqref{eq:NumbLatticePaths}, it is not hard to see that, for
$1\leq i,j\leq 2n$, the right-hand member of~\eqref{eq:def_Qij-b} is
equal to the right-hand member of~\eqref{eq:def_Qij} (which was shown to be the right-hand member of $\eqref{eq:def-Mij}$).
Note that this can also be directly proved from the combinatorial interpretation of~\eqref{eq:def_Qij} and~\eqref{eq:def_Qij-b}
(see e.g. Equation~(3.1) in~\cite{Stem}).
On the other hand, when $1\leq i\leq 2n$ and $j=2n+1$, from~\eqref{eq:def_Qij-b}
and~\eqref{eq:NumbLatticePaths-b} we have
\begin{align}
Q_{i,2n+1}&=\sum_{t=1\atop t\neq n+x+1}^{2x+2n+1}h(A'_i,I'_t)
=-\sum_{t=1}^{n+x}{t-1 \choose
i-1}+\sum_{t=n+x+2}^{2n+2x+1}{t-1\choose i-1},
\end{align}
which, by~\eqref{eq:BinomialSum}, simplifies to the right-hand
member of~\eqref{eq:Qij_B_lastcolumn}. Summarizing, we have proved that the matrix
in~\eqref{eq:defMat1}  is equal to the matrix $N(x)$ described in the lemma.
This ends the proof.
\qed

%%%%%%%%%%%%%%%%%%%%%%%%%%%%%%%%%%%%%%%%%%%%%%%%%%%%%%%%%%%%%%%%%%%%%%%%%%%%%%%%%%%%%%%%%%%%%%%%%%%%%%%%%%%%%%%%%%%%%%%%%%
%%%%%%%%%%%%%%%%%%%%%%%%%%%%%%%%%%%%%%%%%%%%%%%%%%%%%%%%%%%%%%%%%%%%%%%%%%%%%%%%%%%%%%%%%%%%%%%%%%%%%%%%%%%%%%%%%%%%%%%%%%
%%%%%%%%%%%%%%%%%%%%%%%%%%%%%%%%%%%%%%%%%%%%%%%%%%%%%%%%%%%%%%%%%%%%%%%%%%%%%%%%%%%%%%%%%%%%%%%%%%%%%%%%%%%%%%%%%%%%%%%%%%
%
%                         PREUVE THEOREME  1
%
%%%%%%%%%%%%%%%%%%%%%%%%%%%%%%%%%%%%%%%%%%%%%%%%%%%%%%%%%%%%%%%%%%%%%%%%%%%%%%%%%%%%%%%%%%%%%%%%%%%%%%%%%%%%%%%%%%%%%%%%%%
%%%%%%%%%%%%%%%%%%%%%%%%%%%%%%%%%%%%%%%%%%%%%%%%%%%%%%%%%%%%%%%%%%%%%%%%%%%%%%%%%%%%%%%%%%%%%%%%%%%%%%%%%%%%%%%%%%%%%%%%%%
%%%%%%%%%%%%%%%%%%%%%%%%%%%%%%%%%%%%%%%%%%%%%%%%%%%%%%%%%%%%%%%%%%%%%%%%%%%%%%%%%%%%%%%%%%%%%%%%%%%%%%%%%%%%%%%%%%%%%%%%%%
%\newpage
\section{Proof of Theorem~\ref{thm:Ncentered(2n+1,2x)}}

%%%%%%%%%%%%%%%%%%%%%%%%%%%%%%%%%%%%%%%%%%%%%%%%%%%%%%%%%%%%%%%%%%%%%%%%%%%%%%%%%%%%%%%%%%%%%%%%%%%%%%%%%%%%%%%%%%%%%%%%%%
%%%%%%%%%%%%%%%%%%%%%%%%%%%%%                 Plan de la preuve                          %%%%%%%%%%%%%%%%%%%%%%%%%%%%%%%%%
%%%%%%%%%%%%%%%%%%%%%%%%%%%%%%%%%%%%%%%%%%%%%%%%%%%%%%%%%%%%%%%%%%%%%%%%%%%%%%%%%%%%%%%%%%%%%%%%%%%%%%%%%%%%%%%%%%%%%%%%%%
 Throughout this section, we assume that $n$ is a (fixed) nonnegative integer.
From~\eqref{eq:def N(a,b)}, it is easily checked that $ST(2n+1,2x,2n+1)$ can be written in the form
\begin{align*}%\label{eq:N_2(2n+1,2x)}
ST(2n+1,2x,2n+1)=\frac{(x+\frac{1}{2})_{2n+1}}{(\frac{1}{2})_{2n+1}}
               \prod_{s=1}^{n}\frac{(2x+2s)_{4n-4s+3}}{(2s)_{4n-4s+3}}.
\end{align*}
This, combined with~Lemma~\ref{lem:tilings-Pfaffian_casA}, implies
that Theorem~\ref{thm:Ncentered(2n+1,2x)} is equivalent to the
equation
\begin{align}
\begin{split}\label{eq:Theorem1_bis}
 \Pf M(x)=
\frac{1}{4}\frac{(2n+2)!^2(2n)!}{(n+1)!^2(4n+1)!}\,
\prod_{s=1}^{n}\frac{(2x+2s)_{4n-4s+3}}{(2s)_{4n-4s+3}}\\
\times
\left(\sum_{i=0}^n\frac{(-1)^{n-i}}{2n+1-2i}\frac{(x+n+1-i)_{2i}}{(i!)^2}\right),
\end{split}
\end{align}
where $M(x)$ is the matrix defined in Lemma~\ref{lem:tilings-Pfaffian_casA}. We shall prove the
latter formula.

 In Sections~\ref{sec:PfM-1} and~\ref{sec:PfM-2}, we prove that $\Pf\,M(x)$ is a polynomial in $x$ of degree at most $2n^2+3n$
 and that
\begin{align}\label{eq:inv_shift-Pfaffian}
\Pf\, M(x)=(-1)^n\Pf\,M(-2n-1-x).
\end{align}
 In Sections~\ref{sec:PfM-3} and~\ref{sec:PfM-4}, we show that $\prod_{s=1}^n(x+s)^{s}$ and
$\prod_{s=1}^n(x+s+\frac{1}{2})^{s}$ divide $\Pf\, M(x)$ as a polynomial in $x$.
By~\eqref{eq:inv_shift-Pfaffian}, this implies that
\begin{align*}
\left(x+n+\frac{1}{2}\right)^{n}
\prod_{s=1}^{n-1}\left(x+s+\frac{1}{2}\right)^{s}
\left(x+2n+1-s-\frac{1}{2}\right)^{s}
\prod_{s=1}^n(x+s)^{s}(x+2n+1-s)^{s},
\end{align*}
 which simplifies to
\begin{align}
2^{-2n^2-n}\,\prod_{s=1}^{n}\left(2x+2s\right)_{4n-4s+3}
\end{align}
 and is a polynomial of degree $2n^2+n$, divides $\Pf M(x)$ as a polynomial in $x$.
 Altogether, this implies that
\begin{align}\label{eq:def-Px}
\Pf M(x)=P(x)\,\prod_{s=1}^{n}\left(2x+2s\right)_{4n-4s+3}
\end{align}
for some  polynomial $P(x)$ in $x$ of degree at most $2n$. Therefore, in order to prove~\eqref{eq:Theorem1_bis},
it remains to show that $P(x)$ is equal to the polynomial $K(x)$ defined by
\begin{align}\label{eq:def-Kx}
K(x)&=
\frac{1}{4}\frac{(2n+2)!^2(2n)!}{(n+1)!^2(4n+1)!}\prod_{s=1}^{n}\frac{1}{(2s)_{4n-4s+3}}\,
\left(\sum_{i=0}^n\frac{(-1)^{n-i}}{2n+1-2i}\frac{(x+n+1-i)_{2i}}{(i!)^2}\right).
\end{align}
 Since $P(x)$ and $K(x)$  are polynomials in $x$ of degree at most $2n$, it suffices to show that~$P(x)=K(x)$ holds for
$2n+1$ distinct values of $x$. In Section~\ref{sec:PfM-5}, we determine the values of
$P(x)$ and $K(x)$ at $x=0,-1,\dots,-n$ and consequently show that $P(x)=K(x)$ at $x=0,-1,\ldots,-n$.
Since, by~\eqref{eq:def-Px} and~\eqref{eq:inv_shift-Pfaffian},  $P(x)=P(-2n-1-x)$ and, by~\eqref{eq:def-Kx}, $K(x)=K(-2n-1-x )$  for any $x$,
this shows at the same time that $P(x)=K(x)$ at $x=-2n-1,-2n,\dots,-n-1$. In total, there
holds that $P(x)=K(x)$ at $2n+2$ values of $x$. This would complete the proof of~\eqref{eq:Theorem1_bis}.

%%%%%%%%%%%%%%%%%%%%%%%%%%%%%%%%%%%%%%%%%%%%%%%%%%%%%%%%%%%%%%%%%%%%%%%%%%%%%%%%%%%%%%%%%%%%%%%%%%%%%%%%%%%%%%%%%%%%%%%%%%
%%%%%%%%%%%%%%%%%%%%%%%%%%%%%                 Step 1                                     %%%%%%%%%%%%%%%%%%%%%%%%%%%%%%%%%
%%%%%%%%%%%%%%%%%%%%%%%%%%%%%%%%%%%%%%%%%%%%%%%%%%%%%%%%%%%%%%%%%%%%%%%%%%%%%%%%%%%%%%%%%%%%%%%%%%%%%%%%%%%%%%%%%%%%%%%%%%

\subsection{ $\Pf\,M(x)$ is a polynomial in $x$ of degree at most $2n^2+3n$}\label{sec:PfM-1}

   By Lemma~\ref{lem:tilings-Pfaffian_casA} and~\eqref{eq:def-Rij-bis}, the $(i,j)$-entry of
  $M(x)$ is a polynomial in $x$ of degree $i+j$ if $1\leq i, j\leq 2n+1$ and $i\neq j$,
  of degree $i-1$ if $1\leq i\leq 2n+1$ and $j=2n+2$. Moreover, $M_{i,j}(x)=0$ if $i=j$.
  By the formal definition of a Pfaffian (see e.g. \cite[Section~2]{Stem}), this immediately shows that $\Pf\,M(x)$ is
  a polynomial in $x$. This also implies that in the defining expansion of the determinant $\det M(x)$ each nonzero term
  is a polynomial of degree
 \begin{align*}
\left(\sum_{i=1}^{2n+1}i\right)-1+\left(\sum_{j=1}^{2n+1}j\right)-1=4n^2+6n.
 \end{align*}
  Consequently, $\det M(x)$ is a polynomial of degree at most $4n^2+6n$.
  The claim then follows from~\eqref{eq:pf-det}.

%%%%%%%%%%%%%%%%%%%%%%%%%%%%%%%%%%%%%%%%%%%%%%%%%%%%%%%%%%%%%%%%%%%%%%%%%%%%%%%%%%%%%%%%%%%%%%%%%%%%%%%%%%%%%%%%%%%%%%%%%%
%%%%%%%%%%%%%%%%%%%%%%%%%%%%%                 Step 2                                    %%%%%%%%%%%%%%%%%%%%%%%%%%%%%%%%%
%%%%%%%%%%%%%%%%%%%%%%%%%%%%%%%%%%%%%%%%%%%%%%%%%%%%%%%%%%%%%%%%%%%%%%%%%%%%%%%%%%%%%%%%%%%%%%%%%%%%%%%%%%%%%%%%%%%%%%%%%%
 \subsection{ $\Pf M(x)=(-1)^n \Pf M(-2n-1-x)$}\label{sec:PfM-2}
 We shall transform, up to sign, $M(x)$ into
$M(-2n-1-x)$ by a sequence of elementary row and column operations.
More precisely, let $B=\left(B_{i,j}\right)_{1\leq i,j\leq 2n+2}$ be
the lower triangular matrix of size $2n+2$ defined by
\begin{align}
B=\left(
    \begin{array}{ccc|c}
       & & & 0 \\
      &\left(B_{i,j}\right)_{\substack{1\leq i\leq 2n+1 \\ 1\leq j\leq 2n+1}} && \vdots \\
      &&& 0 \\\hline
      0&\cdots&0&-1
    \end{array}
  \right),
  \quad\text{with $B_{i,j}=\binom{i-1}{j-1}$ for $1\leq i,j\leq 2n+1$,}
\end{align}
 and let $M^{(1)}=\left(M^{(1)}_{i,j}\right)_{1\leq i,j\leq 2n+2}$ be the skew-symmetric matrix (of size $2n+2$) defined by $M^{(1)}=B\, M(x)\,B^t$.
 We claim that the $(i,j)$-entry in $M^{(1)}$ is, up to the sign $(-1)^{i+j}$,  the $(i,j)$-entry in $M(-2n-1-x)$.
 Since $\det B=-1$, this would yield
 $$\det M(x)=\det M^{(1)}=\det M(-2n-1-x),$$
 which, combined with~\eqref{eq:pf-det} and the degree of $\Pf M(x)$, would lead to~\eqref{eq:inv_shift-Pfaffian}, as desired.

We now turn to the proof of the claim that $M^{(1)}_{i,j}=(-1)^{i+j}M_{i,j}(-2n-1-x)$.
By definition of the matrices $M^{(1)}$ and $B$, we have
\begin{align}
M^{(1)}_{i,j}&=\sum_{a = 1}^{i} \sum_{b = 1}^{j}{{i-1}\choose{a-1}}
{{j-1}\choose{b-1}}  M_{a,b}(x),\quad\text{$1\leq i,j\leq 2n+1$,}\label{eq:def-Aij}\\
M^{(1)}_{i,2n+2}&=-\sum_{a = 1}^{i} {i-1 \choose{a-1}} M_{a,2n+2}(x),
\quad\text{$1\leq i\leq 2n+1$}.\label{eq:def-Aij-Lr}
\end{align}

Recall that, for $1\leq a,b\leq 2n+1$, $M_{a,b}(x)=R_{a,b}(x)
+T_{a,b}(x)-T_{b,a}(x)$, where the $R_{i,j}(x)$'s and $T_{i,j}(x)$'s
are given by~\eqref{eq:def-Rij} and~\eqref{eq:def-Tij}. It was
already shown in~\cite[Section~5, Step~1, Equations~(5.5)--(5.10)]{CiuKrat-interaction} that
\begin{align}
\sum_{a = 1}^{i} \sum_{b = 1}^{j}{{i-1}\choose{a-1}} {{j-1}\choose
{b-1}}R_{a,b}(x)&=
(-1)^{i+j}R_{i,j}(-2n-1-x).\label{eq:invarianceRij}
\end{align}
On the other hand, using the expression~\eqref{eq:def-Tij}, after a
routine calculation, we obtain
\begin{align}
&\sum_{a = 1}^{i} \sum_{b = 1}^{j}{{i-1}\choose{a-1}} {{j-1}\choose
{b-1}}\,T_{a,b}(x)\nonumber\\
&\quad=\sum_{a = 1}^{i} {{i-1}\choose{a-1}}{2x+2n+1\choose
a}\sum_{b = 1}^{j}{{j-1}\choose{b-1}}\left({x+n\choose b}+{x+n+1\choose b}\right)\nonumber\\
 &\quad=(-1)^{i+j} {-2x-2n-1 \choose i}\left({-n-x\choose j}+{-n-x-1\choose
 j}\right)\label{eq:invarianceTij-bis}\\
 &\quad=(-1)^{i+j}\,T_{i,j}(-2n-1-x),\label{eq:invarianceTij}
\end{align}
where, to obtain~\eqref{eq:invarianceTij-bis}, we used the relation
\begin{align}\label{eq:Chu-Van}
\sum_{k\geq 0}\binom {L}{k-\epsilon}\binom {M}{k}= \binom {L+M}{L+\epsilon}
=(-1)^{L+\epsilon}\binom {-M-1+\epsilon}{L+\epsilon}
\end{align}
which follows from the Chu--Vandermonde summation. After plugging~\eqref{eq:def-Mij} into~\eqref{eq:def-Aij}, it is immediate
from~\eqref{eq:invarianceRij} and~\eqref{eq:invarianceTij}  that $M^{(1)}_{i,j}=(-1)^{i+j}M_{i,j}(-2n-1-x)$ for $1\leq i,j\leq 2n+1$.
 Similarly, plugging~\eqref{eq:def-Mij-Lr} into~\eqref{eq:def-Aij-Lr}, we obtain for $1\leq
i\leq 2n+1$
\begin{align*}
M^{(1)}_{i,2n+2}= -\sum_{a = 1}^{i} {i-1 \choose{a-1}} \binom
{x+n}{a-1}=(-1)^{i}\binom{-x-n-1}{i-1}=(-1)^{i} M_{i,2n+2}(-2n-1-x),
\end{align*}
where the second equality again follows from~\eqref{eq:Chu-Van}
(specialized to $\epsilon=0$).

Summarizing, we have shown that
$M^{(1)}=\left((-1)^{i+j}M_{i,j}(-2n-1-x)\right)_{1\leq i,j\leq 2n+2}$, as
desired. This ends the proof.

%%%%%%%%%%%%%%%%%%%%%%%%%%%%%%%%%%%%%%%%%%%%%%%%%%%%%%%%%%%%%%%%%%%%%%%%%%%%%%%%%%%%%%%%%%%%%%%%%%%%%%%%%%%%%%%%%%%%%%%%%%
%%%%%%%%%%%%%%%%%%%%%%%%%%%%%                 Step 3                                     %%%%%%%%%%%%%%%%%%%%%%%%%%%%%%%%%
%%%%%%%%%%%%%%%%%%%%%%%%%%%%%%%%%%%%%%%%%%%%%%%%%%%%%%%%%%%%%%%%%%%%%%%%%%%%%%%%%%%%%%%%%%%%%%%%%%%%%%%%%%%%%%%%%%%%%%%%%%
\subsection{ $\prod_{s=1}^n(x+s)^{s}$ divides $\Pf M(x)$}\label{sec:PfM-3}

Let $s$ be an integer with $1\leq s\leq n$. It is easily checked
that
\begin{align}\label{eq:comblineaire-s}
&{}\,\textrm{For $2n-2s+2\leq a\leq 2n+1$, the $a$-th row of the
matrix $M(-s)$ is null.}
\end{align}
 Indeed, by Lemma~\ref{lem:tilings-Pfaffian_casA}, the entries  of the $a$-th row of $M(-s)$
are, for $1\leq j\leq 2n+1$,
\begin{align*}
M_{a,j}(-s)&=\sum_{t=1}^{2n-2s+1}\frac{j-a}{t}{t\choose a}{t\choose
j}+{2n-2s+1\choose a}\left({n-s\choose j}+{n-s+1\choose j}\right)\\
&\hspace{5cm} -{2n-2s+1\choose j}\left({n-s\choose a}+{n-s+1\choose
a}\right),
\end{align*}
and $M_{a,2n+2}(-s)={n-s\choose a-1}$. These entries are clearly
null if $a> 2n-2s+1$ since the binomial coefficients ${n-s\choose
a-1}$ and ${t\choose a}$, for $t=1,\ldots,2n-2s+1$, vanish.

 By~\eqref{eq:comblineaire-s}, we have
 $2s$ linear combinations of the rows of the matrix~$M(x)$
 that are linearly independent and vanish at $x=-s$. This,
 combined with the next lemma, yields the divisibility of $\Pf M(x)$  by~$(x+s)^{s}$, as desired.

\begin{lem}\cite[Section~2]{Krat}\label{lem:pf-factor}
Let $A(x)$ be a skew-symmetric matrix with polynomial entries. A
sufficient condition for $(x-b)^m$ to divide $\Pf A(x)$ is that
the dimension of the kernel of the matrix $A(b)$ is at least $2m$.
\end{lem}

%%%%%%%%%%%%%%%%%%%%%%%%%%%%%%%%%%%%%%%%%%%%%%%%%%%%%%%%%%%%%%%%%%%%%%%%%%%%%%%%%%%%%%%%%%%%%%%%%%%%%%%%%%%%%%%%%%%%%%%%%%
%%%%%%%%%%%%%%%%%%%%%%%%%%%%%                 Step 3                                     %%%%%%%%%%%%%%%%%%%%%%%%%%%%%%%%%
%%%%%%%%%%%%%%%%%%%%%%%%%%%%%%%%%%%%%%%%%%%%%%%%%%%%%%%%%%%%%%%%%%%%%%%%%%%%%%%%%%%%%%%%%%%%%%%%%%%%%%%%%%%%%%%%%%%%%%%%%%
\subsection{$\prod_{s=1}^n(x+s+1/2)^{s}$ divides $\Pf M(x)$}\label{sec:PfM-4}

 We first notice that some coefficients of the matrix $M(x)$ when specialized to $x=-s-1/2$,
 $1\leq s\leq n$, have a relative simple form. Namely, for $2n-2s+1\leq i\leq 2n+1$ and $j\leq 2n+1$,
we have
\begin{align}
M_{i,j}(-s-1/2)&=-M_{j,i}(-s-1/2)=-{2n-2s\choose
j}\left({n-s-\frac{1}{2}\choose i}+{n-s+\frac{1}{2}\choose i}
\right),\label{eq:Qij-s-1/2_simplification_i>}
\end{align}
and thus,
\begin{align}\label{eq:Qij-s-1/2_annulation}
M_{i,j}(-s-1/2)&=0\quad \textrm{for  $1\leq s\leq n$ and
$2n-2s+1\leq i,j\leq 2n+1$}.
\end{align}
Indeed, by~\eqref{eq:def-Mij}, we have
\begin{align*}
M_{i,j}(-s-1/2)&=\sum_{t=1}^{2n-2s}\frac{j-i}{t}{t\choose
i}{t\choose j}+{2n-2s\choose i}\left({n-s-\frac{1}{2}\choose
j}+{n-s+\frac{1}{2}\choose
j}\right)\\
&\hspace{5cm} -{2n-2s\choose j}\left({n-s-\frac{1}{2}\choose
i}+{n-s+\frac{1}{2}\choose i} \right)
\end{align*}
and the binomial coefficients
${t\choose i}$ vanish for $t=1,\ldots,2n-2s$ and $i\geq 2n-2s+1$.

\subsubsection{The case $1\leq s\leq n-1$} Suppose $1\leq s\leq n-1$. We shall show that
\begin{align}
&(-1)^{a}{n+s-\frac{1}{2}\choose a+1}\cdot\left( \textrm{row $(2n-a)$ of $M(-s-1/2)$ } \right)\nonumber\\
&\hspace{1cm}-\left(n+s-\frac{1}{2}\right) {2n\choose a}\cdot \left(
\textrm{row $2n$ of $M(-s-1/2)$ } \right)\label{eq:comblineaire-s-1/2_a}\\
&\hspace{2cm}-a {2n+1\choose a+1}\cdot\left( \textrm{row $(2n+1)$ of
$M(-s-1/2)$ } \right)=0\nonumber
\end{align}
for $a=1,\ldots,2s-1$, and that
\begin{align}
&\sum_{i=1}^{2n-2s}(-2)^{i-1}\cdot\left( \textrm{row $i$ of $M(-s-1/2)$ } \right)\nonumber\\
&\hspace{1cm}-4^{n-s}
\frac{(2n-2s+2)_{2s}}{(n-s+\frac{1}{2})_{2s-1}}\cdot\left(
\textrm{row $2n$ of $M(-s-1/2)$ } \right)\label{eq:comblineaire-s-1/2_b}\\
&\hspace{2cm}-2n\cdot 4^{n-s}
\frac{(2n-2s+2)_{2s}}{(n-s+1/2)_{2s}}\cdot\left( \textrm{row
$(2n+1)$ of $M(-s-1/2)$ } \right)=0.\nonumber
\end{align}
As these are $2s$ linear combinations of the rows of the matrix
$M(x)$ that are linearly independent and vanish at $x=-s-1/2$, this
would lead, by Lemma~\ref{lem:pf-factor}, to the divisibility of $\Pf M(x)$ by~$(x+s+\frac{1}{2})^{s}$.\\

\noindent\emph{Proof of~\eqref{eq:comblineaire-s-1/2_a}.}
 Suppose $1\leq a\leq 2s-1$. We have to show that, for $1\leq j\leq
 2n+2$, we have
\begin{align}
\begin{split}\label{eq:comblineaire-s-1/2_a_firstcolumns}
(-1)^{a}{n+s-\frac{1}{2}\choose a+1} M_{2n-a,j}(-s-1/2) -
\left(n+s-\frac{1}{2}\right) {2n\choose a} M_{2n,j}(-s-1/2)&\\
- a {2n+1\choose a+1}  M_{2n+1,j}(-s-1/2)=0.&
\end{split}
\end{align}

(a) When $2n-2s+1\leq j\leq 2n+1$,  all the matrix entries in
\eqref{eq:comblineaire-s-1/2_a_firstcolumns} are null
by~\eqref{eq:Qij-s-1/2_annulation}.  So, the identity is clearly
true.\\

 (b) Suppose $1\leq j\leq 2n-2s$. By plugging~\eqref{eq:Qij-s-1/2_simplification_i>} into~\eqref{eq:comblineaire-s-1/2_a_firstcolumns},
 we obtain the relation
\begin{align}\label{eq:comblineaire-s-1/2_a_eq1}
&(-1)^{a}{n+s-\frac{1}{2}\choose a+1}\left({n-s-\frac{1}{2}\choose
2n-a}+{n-s+\frac{1}{2}\choose
2n-a}\right)\nonumber\\
&\qquad-\left(n+s-\frac{1}{2}\right) {2n\choose
a}\left({n-s-\frac{1}{2}\choose 2n}+{n-s+\frac{1}{2}\choose
2n}\right)&\\
&\qquad\qquad\qquad - a {2n+1\choose
a+1}\left({n-s-\frac{1}{2}\choose 2n+1}+{n-s+\frac{1}{2}\choose
2n+1}\right)=0,\nonumber
\end{align}
which amounts to a routine verification.\\

(c) Suppose $j=2n+2$. By~\eqref{eq:def-Mij-Lr},
\eqref{eq:comblineaire-s-1/2_a_firstcolumns} reduces  to the relation
\begin{align}
&(-1)^{a}{n+s-\frac{1}{2}\choose a+1}{n-s-\frac{1}{2}\choose 2n-a-1}
-\left(n+s-\frac{1}{2}\right) {2n\choose a}{n-s-\frac{1}{2}\choose 2n-1}\nonumber\\
&- a {2n+1\choose a+1}{n-s-\frac{1}{2}\choose
2n}=0,\label{eq:comblineaire-s-1/2_a_eq2}
\end{align}
which again amounts to a routine verification. This ends the proof
of~\eqref{eq:comblineaire-s-1/2_a}.
\qed\\

\noindent\emph{Proof of~\eqref{eq:comblineaire-s-1/2_b}.} We have to
show that, for $1\leq j\leq  2n+2$, we have
\begin{align}
\begin{split}\label{eq:comblineaire-s-1/2_b_eq}
\sum_{i=1}^{2n-2s}(-2)^{i-1} M_{i,j}(-s-1/2)&=4^{n-s}
\frac{(2n-2s+2)_{2s}}{(n-s+\frac{1}{2})_{2s-1}}M_{2n,j}(-s-1/2)\\
&\quad+2n\cdot 4^{n-s}
\frac{(2n-2s+2)_{2s}}{(n-s+\frac{1}{2})_{2s}}M_{2n+1,j}(-s-1/2).
\end{split}
\end{align}

 (a) Suppose $j=2n+2$. Then, using the expression~\eqref{eq:def-Mij-Lr}, it is easily checked that the right-hand
 side of~\eqref{eq:comblineaire-s-1/2_b_eq} is zero (when specialized to $j=2n+2$), and thus~\eqref{eq:comblineaire-s-1/2_b_eq}
 reduces to the relation
\begin{align}\label{eq:comblineaire-s-1/2_b_eq1_hypidentity}
 \sum_{i=0}^{2n-2s-1}(-2)^{i} {n-s-\frac{1}{2}\choose
i}&=0.
\end{align}
Reversing the order of summation over $i$ (that is we replace $i$ by
$2n-2s-1-i$), the left-hand side
of~\eqref{eq:comblineaire-s-1/2_b_eq1_hypidentity} can be written
using standard hypergeometric notation
\begin{align*}
{}_pF_q\left[{{a_1,\ldots,a_p}\atop
{b_1,\ldots,b_q}};z\right]=\sum_{n\geq0}\frac{(a_1)_n\cdots(a_p)_n}{(b_1)_n\cdots(b_q)_n}\frac{z^n}{n!}
\end{align*}
as
\begin{align*}
(-2)^{2n-2s-1} {n-s-\frac{1}{2}\choose 2n-2s-1}\,
{}_2F_1\left[{{1,-2n+2s+1}\atop
{-n+s+\frac{3}{2}}};\frac{1}{2}\right]
\end{align*}
which is indeed null (recall that $1\leq s\leq n-1$) by means of
Gauss' second ${}_2F_1$-summation
\begin{align*}
{}_2F_1\left[{{a,-N}\atop
{\frac{1}{2}+\frac{a}{2}-\frac{N}{2}}};\frac{1}{2}\right] =\left\{
  \begin{array}{ll}
      0, & \hbox{if $N$ is an odd nonnegative integer,}\\[0.25cm]
     \frac{\left(\frac{1}{2}\right)_{N/2}}{\left(\frac{1-a}{2}\right)_{N/2}},
& \hbox{if $N$ is an even nonnegative integer.}
  \end{array}
 \right.
\end{align*}

(b) Suppose $2n-2s+1\leq j\leq 2n+1$. Then, using the
expressions~\eqref{eq:Qij-s-1/2_simplification_i>} and~\eqref{eq:Qij-s-1/2_annulation} for the matrix
entries in the sum in~\eqref{eq:comblineaire-s-1/2_b_eq},
it is easily checked that the right-hand side of~\eqref{eq:comblineaire-s-1/2_b_eq} vanishes so that~~\eqref{eq:comblineaire-s-1/2_b_eq}
reduces to the identity
\begin{align*}
 \sum_{i=1}^{2n-2s}(-2)^{i-1}\left({n-s-\frac{1}{2}\choose
j}+{n-s+\frac{1}{2}\choose j}\right){2n-2s\choose i}=0,
\end{align*}
which is an immediate consequence of the binomial theorem.\\

(c) Suppose $1\leq j\leq 2n-2s$. Then, by~\eqref{eq:def-Mij}, the
left-hand side of~\eqref{eq:comblineaire-s-1/2_b_eq} is
\begin{align}
\sum_{i=1}^{2n-2s}(-2)^{i-1}
M_{i,j}(-s-1/2)&=\sum_{i=1}^{2n-2s}(-2)^{i-1}\sum_{t=1}^{2n-2s}\frac{j-i}{t}{t\choose
i}{t\choose j}\label{eq:sumQi}\\
&\;\;+\left({n-s-\frac{1}{2}\choose
j}+{n-s+\frac{1}{2}\choose j}\right)\sum_{i=1}^{2n-2s}(-2)^{i-1}{2n-2s\choose i}\nonumber\\
&\;-{2n-2s\choose j}\sum_{i=1}^{2n-2s}
(-2)^{i-1}\left({n-s-\frac{1}{2}\choose i}+{n-s+\frac{1}{2}\choose
i}\right).\nonumber
\end{align}
The first sum in~\eqref{eq:sumQi} is equal to ${2n-2s\choose j}$.
This can be deduced from the formula
\begin{align}\label{eq:sumbinom}
\sum_{i=1}^{N}(-2)^{i-1}\sum_{t=1}^{N}\frac{j-i}{t}{t\choose
i}{t\choose j} &=\frac{1+(-1)^N}{2}\binom{N}{2},
\end{align}
which is valid for any nonnegative integers $N$ and $j$. For $N$
even, \eqref{eq:sumbinom} was proved in~\cite[proof of
equation~5.21]{CiuKrat-interaction}. The proof for $N$ odd is
similar as the case $N$ even and involves
only elementary manipulations. It is thus left to the reader.\\
By the binomial theorem the second sum in~\eqref{eq:sumQi} vanishes.
Moreover, the last sum in~\eqref{eq:sumQi} simplifies to
\begin{align}\label{eq:sumQi_last1}
&\sum_{i=1}^{2n-2s} (-2)^{i-1}\left({n-s-\frac{1}{2}\choose
i}+{n-s+\frac{1}{2}\choose
i}\right)=1-4^{n-s}{n-s-\frac{1}{2}\choose 2n-2s}.
\end{align}
This immediately follows from the telescoping equation
\begin{align*}
(-2)^{i-1}\left({n-s-\frac{1}{2}\choose i}+{n-s+\frac{1}{2}\choose
i}\right)=G(i+1)-G(i)
\end{align*}
with $G(i)=-(-2)^{i-1}{n-s-\frac{1}{2}\choose i-1}$.
 Finally, replacing  the sums in~\eqref{eq:sumQi} by their evaluation,
we arrive at
\begin{align}\label{eq:sumQi_simplification}
\sum_{i=1}^{2n-2s}(-2)^{i-1} M_{i,j}(-s-1/2)&=4^{n-s}{2n-2s\choose
j}{n-s-\frac{1}{2}\choose 2n-2s}.
\end{align}
On the other hand, by~\eqref{eq:Qij-s-1/2_simplification_i>}, the
right-hand side of~\eqref{eq:comblineaire-s-1/2_b_eq} is equal to
\begin{align*}
&-4^{n-s}\frac{(2n-2s+2)_{2s}}{(n-s+\frac{1}{2})_{2s-1}}\binom{2n-2s}{j}\left({n-s-\frac{1}{2}\choose
2n}+{n-s+\frac{1}{2}\choose 2n}\right)\\
&\qquad\qquad\qquad -2n 4^{n-s}
\frac{(2n-2s+2)_{2s}}{(n-s+\frac{1}{2})_{2s}}\binom{2n-2s}{j}\left({n-s-\frac{1}{2}\choose
2n+1}+{n-s+\frac{1}{2}\choose 2n+1}\right).
\end{align*}
Therefore, to prove that~\eqref{eq:comblineaire-s-1/2_b_eq} holds
for $1\leq j\leq 2n-2s$, it suffices to show that the above
expression is equal to the right-hand side
of~\eqref{eq:sumQi_simplification}. This amounts to a routine
verification. This concludes the proof
of~\eqref{eq:comblineaire-s-1/2_b}.
 \qed

\subsubsection{The case $s=n$} From~\eqref{eq:Qij-s-1/2_annulation}
we see that all the coefficients of the $i$-th row, $1\leq i\leq
2n+1$, of the matrix $M(-n-1/2)$ are null except its last
coefficient $M_{i,2n+2}(-n-1/2)$ (which is,
by~\eqref{eq:def-Mij-Lr}, equal to ${-1/2\choose i-1}$). It is then
immediate to find $2n$ linear combinations of the rows of the matrix
$M(-n-1/2)$ that are linearly independent and vanish. For instance,
for $a=1,\ldots,2n$, we have
\begin{align}
&{}\left( \textrm{row $a$ of $M(-n-1/2)$ } \right)-
\frac{M_{a,2n+2}(-n-1/2)}{M_{a+1,2n+2}(-n-\frac{1}{2})}\cdot\left(
\textrm{row $(a+1)$ of $M(-n-1/2)$} \right)=0.\nonumber
\end{align}
 This implies divisibility of $\Pf M(x)$ by~$(x+n+\frac{1}{2})^{n}$.

%%%%%%%%%%%%%%%%%%%%%%%%%%%%%%%%%%%%%%%%%%%%%%%%%%%%%%%%%%%%%%%%%%%%%%%%%%%%%%%%%%%%%%%%%%%%%%%%%%%%%%%%%%%%%%%%%%%%%%%%%%
%%%%%%%%%%%%%%%%%%%%%%%%%%%%%                 Step 5                                     %%%%%%%%%%%%%%%%%%%%%%%%%%%%%%%%%
%%%%%%%%%%%%%%%%%%%%%%%%%%%%%%%%%%%%%%%%%%%%%%%%%%%%%%%%%%%%%%%%%%%%%%%%%%%%%%%%%%%%%%%%%%%%%%%%%%%%%%%%%%%%%%%%%%%%%%%%%%
\subsection{$P(x)=K(x)$ at $x=0,-1,\ldots,-n$}\label{sec:PfM-5}

Let $\s$ be a given integer with $0\leq \s \leq n$. It is not too
hard to evaluate the polynomial $K(x)$, defined by~\eqref{eq:def-Kx}, at~$x=-\s$. Indeed, using the
Pfaff-Saalschutz summation, it is easily checked (see
~\cite[Equation~(5.26)]{CiuKrat-centered}) that
\begin{align}
\sum_{i=0}^n\frac{(-1)^{n-i}}{2n+1-2i}\frac{(n-\s+1-i)_{2i}}{(i!)^2}
=(-1)^{n+1}\frac{\left(n+\frac{3}{2}\right)_{n-\s}}{2\left(-n-\frac{1}{2}\right)_{n+1-\s}},
\quad\text{$0\leq \sigma\leq n$}.
\end{align}
 By inserting in~\eqref{eq:def-Kx} the latter identity, after
some simplification, we arrive  at
\begin{align}\label{eq:Kx at -s}
K(-\s)&=(-1)^{\s}\frac{(4n-2\s+1)!\,(2\s)!\,(2n)!}{(2n-\s)!\,\s!(4n+1)!}\prod_{s=1}^n\frac{(2s-1)!}{(4n+2-2s)!},
\quad\text{$0\leq \sigma\leq n$}.
\end{align}

The evaluation of $P(x)$ at $x=-\s$ is much more delicate. The
polynomial $P(x)$ is defined by means of \eqref{eq:def-Px}. Since the product on
the right-hand side of~\eqref{eq:def-Px}  is zero for $x=-\s$,
$1\le \s\le n$, we should write~\eqref{eq:def-Px} in the form
\begin{align*}
 P(x)=\frac{2^{-\sigma}}{(x+\sigma)^{\sigma}} \Pf M(x)
\;&\prod_{s=1\atop s\neq \sigma}^n(2x+2s)^{-s}\;\prod_{s=1}^n(2x+4n+2-2s)^{-s} \\
&\times\prod_{s=1}^{n}\left(2x+2s+1\right)^{-s}\;\prod_{s=1}^{n-1}\left(2x+4n+1-2s\right)^{-s}
\end{align*}
and subsequently specialize $x=-\sigma$. After some manipulation,
this gives
\begin{align}\label{eq:Px at -s}
\begin{split}
 P(-\s)=\frac{2^{-\sigma}\prod _{s=1}^{n-\s} (2s-1)!}{\prod _{s=1}^{\s-1}
(2s)!\;\prod_{s=n-\s+1}^{2n-\s} (2s)!}
\left(\frac{1}{(x+\sigma)^{\sigma}} \Pf
M(x)\right)\bigg\vert_{x=-\sigma}.
\end{split}
\end{align}
To evaluate $(x+\s)^{-\s} \Pf M(x)$ at $x=-\s$, we shall use the
following lemma due to Ciucu and the second author and used in a similar
context.
\begin{lem}\label{lem:FactorizationPfaffian}{(\cite[Lemma 11]{CiuKrat-interaction})}
Let $N,a,b$ be positive integers with $a< b\le N$, where $N$ and
$b-a$ are even. Let $A=(A_{i,j})_{1\le i,j\le N}$ be a
skew-symmetric matrix with the following properties:
\begin{enumerate}
\item The entries of $A$ are polynomials in $x$.
\item The entries in rows $a+1,a+2,\dots,b$ {\rm(}and, hence, also in
the corresponding columns{\rm)} are divisible by $x+s$.
\end{enumerate}
Then
$$\left(\frac {1} {(x+s)^{(b-a)/2}}\Pf A\right)\bigg\vert_{x=-s}=\Pf \widetilde
A\,\cdot\,\Pf S,$$
 where $\widetilde A$ is the matrix which arises from
$A$ by deleting rows and columns $a+1,a+2,\dots,b$ and subsequently
specializing $x=-s$, and
$$S=\left(
\left(\frac {1} {x+s}A_{i,j}\right)\bigg\vert_{x=-s}\right) _{a+1\le
i,j\le b}.$$
\end{lem}

  Let $\s$ be a given integer with $1\leq \s \leq n$. It follows  from~\eqref{eq:comblineaire-s} that the coefficients of the $i$-th row,
  $2n-2\s+2 \leq i \leq  2n+1$, of the matrix $M(x)$ are divisible by $(x+\s)$. Applying Lemma~\ref{lem:FactorizationPfaffian},
 we obtain
\begin{align}\label{eq:factorization_Pfaffian}
\left(\frac {1} {(x+\s)^{\s}}\Pf M(x)\right)\bigg\vert_{x=-\s}
=\Pf \widetilde M\, \cdot\, \Pf S,\end{align}
 where $\Mtil= (\Mtil_{i,j})_{1\leq i,j\leq 2n-2\s+2}$ is the skew-symmetric matrix  of size $2n-2\s+2$ defined by
\begin{align}
\begin{split}\label{eq:def-Mtil}
 \Mtil_{i,j}&=M_{i,j}(-\s),\quad 1\leq i,j\leq 2n-2\s+1,\\
 \Mtil_{i,2n-2\s+2}&=M_{i,2n+2}(-\s),\quad 1\leq i\leq 2n-2\s+1,
 \end{split}
\end{align}
 and $S= (S_{i,j})_{1\leq i,j\leq 2\s}$ is the skew-symmetric matrix  of size $2\s$ defined by
\begin{align}\label{eq:def-MatS}
S=\left( \left(\,\frac {1}
{x+\s}M_{i,j}(x)\right)\bigg\vert_{x=-\s}\,\right)_{2n-2\s+2\le
i,j\le 2n+1}.
\end{align}
  We point out that~\eqref{eq:factorization_Pfaffian} also holds for $\s=0$ once we interpret the Pfaffian of an empty matrix (i.e., the
 Pfaffian of $S$) as $1$. In particular, under that convention, the arguments below can be used for $0\le \s\le n$, that is,
 {\it including} $\s=0$.

 We shall prove that $\Pf \Mtil=1$ and
\begin{align}
&\Pf S=(-1)^{\s} 2^{\s} \left(\prod_{i=1}^{2\s}(2n-2\s+i)! \right)
\left( \prod _{i=1} ^{\s}\frac {(2i-1)!}
{(4n-2\s+2i+1)!}\right).\label{eq:evaluation_PfS}
\end{align}
If we substitute in~\eqref{eq:factorization_Pfaffian} these values
for $\Pf \widetilde M$ and $\Pf S$, and then insert the
 resulted equation in~\eqref{eq:Px at -s}, after
 some simplification, we would arrive at
\begin{align}\label{eq:Pn at -s}
 P(-\s)=(-1)^{\s}\frac{ \prod_{s=2n-2\s+1}^{2n}s!\; \prod _{s=1} ^{\s} (2s-1)!\;\prod _{s=1} ^{n-\s} (2s-1)!}
{ \prod _{s=1} ^{\s-1} (2s)!\;\prod _{s=n+1} ^{2n-\s} (2s)!\;\prod
_{s=0} ^{\s-1}(4n+1-2s)!}, \quad\text{$0\leq \sigma\leq n$}.
\end{align}
Then, a routine comparison of~\eqref{eq:Pn at -s} with~\eqref{eq:Kx at -s}
would show that $P(x)=K(x)$ at $x=0,-1,\ldots,-n$, as desired. So, to
complete our proof, it remains to establish the evaluations
of $\Pf \Mtil$ and $\Pf S$.\\

\noindent\emph{Evaluation of $\Pf \Mtil$.} If we
compare the matrix $\Mtil$ defined by~\eqref{eq:def-Mtil} with the matrix
$M(x)$ in Lemma~\ref{lem:tilings-Pfaffian_casA}, then we see that
$\widetilde M$ is equal to the matrix $M(x)$ in
Lemma~\ref{lem:tilings-Pfaffian_casA}  with $n$ replaced by $n-\s$
and with $x=0$. By Lemma~\ref{lem:tilings-Pfaffian_casA},
  this implies that $\Pf \Mtil$ is equal to the number  of centered vertically symmetric tilings of a $(2(n-\s)+1,0,2(n-\s)+1)$ hexagon.
 In a tiling of such a hexagon, all rhombi are forced, and trivially the unique tiling of such a hexagon is centered and
 symmetric. Consequently, we have $\Pf \Mtil=1$.\\

\noindent\emph{Evaluation of $\Pf S$.} By~\eqref{eq:def-MatS}
and~\eqref{eq:def-Mij}, the $(i,j)$-entry of the matrix $S$
satisfies
 \begin{align*}
S_{i,j}&=\frac{R_{2N+1+i,2N+1+j}(x)}{x+\s}+\frac{T_{2N+1+i,2N+1+j}(x)}{x+\s}
-\frac{T_{2N+1+j,2N+1+i}(x)}{x+\s}\bigg\vert_{x=-\s},
\end{align*}
where we have set $N=n-\s$. Using expression~\eqref{eq:def-Rij-bis}
for the $R_{i,j}$'s and~\eqref{eq:def-Tij} for the $T_{i,j}$'s,
after a routine calculation, we arrive at
\begin{align}
\begin{split}\label{eq:ValSij}
S_{i,j}= 2\cdot \sum _{\ell=0} ^{i+2n-2\s}(-1)^{\ell+j+1} &\frac
{j-i} {i+2n-2\s+1}\binom {j+2n-2\s}{i+2n-2\s-\ell}\\
&\cdot\binom{\ell+j+2n-2\s+1}{\ell} \frac {(2n-2\s+2)!\,(\ell+j-1)!}
{(\ell+j+2n-2\s+2)!}.
\end{split}
\end{align}
To derive the above equation, we just used the following
relations whose proof is left to the reader: for $r,t$ with $r\geq1,
t\geq 0$, and $\epsilon=1,2$ we have
\begin{align}\label{eq:limit1-binom at -s}
\begin{split}
&\lim_{x\to -\s} \frac{1}{x+\s}\binom
{2x+2n+t}{r+2n-2\s+t}=(-1)^{r-1}2\,\frac{(2n-2\s+t)!\,(r-1)!}{(r+2n-2\s+t)!},\\
&\lim_{x\to -\s} \binom {x+n+\epsilon}{r+2n-2\s+1}=0.
\end{split}
\end{align}

\begin{lem}\label{lem:evaluationPfS} %For nonnegative integers $a$, we have
 \begin{multline*}
\sum _{\ell=0} ^{i+a}(-1)^{\ell+j+1} \frac {j-i} {i+a+1}\binom
{j+a}{i+a-\ell}\binom {\ell+j+a+1}{\ell} \frac
{(a+2)!\,(\ell+j-1)!} {(\ell+j+a+2)!}\\
={{\left( -1 \right) }^{i + j }}\,{\frac {
      ( j-i ) \,(a+i)!\,(a+j)! }
      {(2a+i+j+2)!}}.
\end{multline*}
\end{lem}
 The above result was proved in the particular case $a=2n-2\s-1$ in~\cite[page~277]{CiuKrat-interaction}.
 Since the arguments in~\cite[page~277]{CiuKrat-interaction} can be used in the same way to prove Lemma~\ref{lem:evaluationPfS},
we omit the proof details. Combining the preceding lemma with~\eqref{eq:ValSij}, we arrive at
\begin{align}\label{eq:ValSij-simplification}
S_{i,j}= {\left( -1 \right) }^{i + j }\,{\frac {2\,
      ( j-i ) \,(2n-2\s+i)!\,(2n-2\s+j)! }
      {(4n-4\s+i+j+2)!}}.
\end{align}
Consequently, we have
\begin{align*}
\underset {1\le i,j\le 2\s}\Pf(S_{i,j})&=(-1)^{\s} 2^{\s}
\left(\prod_{i=1}^{2\s}(2n-2\s+i)! \right) \underset {1\le i,j\le
2\s}\Pf\left(\frac{j-i}{(4n-4\s+i+j+2)!}\right),
\end{align*}
 which simplifies to~\eqref{eq:evaluation_PfS}, as desired, by
the Pfaffian evaluation
\begin{align}\label{eq:Mehta-Wang}
\underset {0\leq i,j\leq 2k-1}\Pf\left(\frac {(j-i )}
{(b+i+j)!}\right) =\prod_{i=0}^{k-1}
\frac{(2i+1)!}{(b+2k+2i-1)!},\quad k\geq 1.
\end{align}
Note that the above equation is a slight variation of a Pfaffian
evaluation due to Mehta and Wang (see~\cite[Corollary~10]{CiuKrat-interaction}).
This completes the proof of~\eqref{eq:evaluation_PfS}.  \qed

%%%%%%%%%%%%%%%%%%%%%%%%%%%%%%%%%%%%%%%%%%%%%%%%%%%%%%%%%%%%%%%%%%%%%%%%%%%%%%%%%%%%%%%%%%%%%%%%%%%%%%%%%%%%%%%%%%%%%%%%%%
%%%%%%%%%%%%%%%%%%%%%%%%%%%%%%%%%%%%%%%%%%%%%%%%%%%%%%%%%%%%%%%%%%%%%%%%%%%%%%%%%%%%%%%%%%%%%%%%%%%%%%%%%%%%%%%%%%%%%%%%%%
%%%%%%%%%%%%%%%%%%%%%%%%%%%%%%%%%%%%%%%%%%%%%%%%%%%%%%%%%%%%%%%%%%%%%%%%%%%%%%%%%%%%%%%%%%%%%%%%%%%%%%%%%%%%%%%%%%%%%%%%%%
%
%                      PROOF OF  THEOREM 2
%
%%%%%%%%%%%%%%%%%%%%%%%%%%%%%%%%%%%%%%%%%%%%%%%%%%%%%%%%%%%%%%%%%%%%%%%%%%%%%%%%%%%%%%%%%%%%%%%%%%%%%%%%%%%%%%%%%%%%%%%%%%
%%%%%%%%%%%%%%%%%%%%%%%%%%%%%%%%%%%%%%%%%%%%%%%%%%%%%%%%%%%%%%%%%%%%%%%%%%%%%%%%%%%%%%%%%%%%%%%%%%%%%%%%%%%%%%%%%%%%%%%%%%
%%%%%%%%%%%%%%%%%%%%%%%%%%%%%%%%%%%%%%%%%%%%%%%%%%%%%%%%%%%%%%%%%%%%%%%%%%%%%%%%%%%%%%%%%%%%%%%%%%%%%%%%%%%%%%%%%%%%%%%%%%
%\newpage
\section{Proof of Theorem~\ref{thm:Ncentered(2n,2x+1)}}

%%%%%%%%%%%%%%%%%%%%%%%%%%%%%%%%%%%%%%%%%%%%%%%%%%%%%%%%%%%%%%%%%%%%%%%%%%%%%%%%%%%%%%%%%%%%%%%%%%%%%%%%%%%%%%%%%%%%%%%%%%
%%%%%%%%%%%%%%%%%%%%%%%%%%%%%                 Plan de la preuve                          %%%%%%%%%%%%%%%%%%%%%%%%%%%%%%%%%
%%%%%%%%%%%%%%%%%%%%%%%%%%%%%%%%%%%%%%%%%%%%%%%%%%%%%%%%%%%%%%%%%%%%%%%%%%%%%%%%%%%%%%%%%%%%%%%%%%%%%%%%%%%%%%%%%%%%%%%%%%
 Throughout this section, we assume that $n$ is a (fixed) positive integer.
From~\eqref{eq:def N(a,b)}, it is easily checked that $ST(2n,2x+1,2n)$
can be written in the form
\begin{align*}%\label{eq:N_2(2n+1,2x)}
ST(2n,2x+1,2n)&=\frac{(x+1)_{2n}}{(\frac{1}{2})_{2n}}
               \prod_{s=1}^{n}\frac{(2x+1+2s)_{4n-4s+1}}{(2s)_{4n-4s+1}}.
\end{align*}
This, combined with Lemma~\ref{lem:tilings-Pfaffian_casB}, leads to
the following reformulation of Theorem~\ref{thm:Ncentered(2n,2x+1)}:
\begin{align}\label{eq:Theorem2_bis}
\Pf N(x)&= \frac{2^{5n-1}}{n!(4n)!}\, (x+1)_{2n} \,
\prod_{s=2}^{n}\frac{(2x+2s)_{4n-4s+3}}{(2s)_{4n-4s+3}}\;U_{n}(x),
\end{align}
 where $N(x)$ is the matrix defined in Lemma~\ref{lem:tilings-Pfaffian_casB} and
 $U_n(x)$ is defined by~\eqref{eq:def-Ux}.
 Our proof of~\eqref{eq:Theorem2_bis} is, in spirit, quite similar to the proof of~\eqref{eq:Theorem1_bis}.

  First, with exactly the same kind of reasoning used in Section~\ref{sec:PfM-1},
it is easily seen that $\Pf N(x)$ is a polynomial in $x$ of degree at
most $2n^2+n-1$. We omit the details. In Section~\ref{sec:PfN-2}, we prove that
\begin{align}\label{eq:inv_shift-Pfaffian_N}
\Pf N(x)=(-1)^{n+1}\Pf N(-2n-1-x).
\end{align}
In Sections~\ref{sec:PfN-3} and~\ref{sec:PfN-4}, we show that $\prod_{s=1}^n(x+s)^{s}$ and
$\prod_{s=1}^{n-1}(x+s+\frac{3}{2})^{s}$ divide $\Pf N(x)$ as a
polynomial in $x$. This, combined with~\eqref{eq:inv_shift-Pfaffian_N}, implies that
\begin{align*}
\left(x+n+\frac{1}{2}\right)^{n-1}
\prod_{s=1}^{n-2}\left(x+s+\frac{3}{2}\right)^{s}
\left(x+2n+1-s-\frac{3}{2}\right)^{s}
\prod_{s=1}^n(x+s)^{s}(x+2n+1-s)^{s},
\end{align*}
 which is equal to
\begin{align}
2^{-2n^2+3n-1}\,(x+1)_{2n}\,\prod_{s=2}^{n}\left(2x+2s\right)_{4n-4s+3}
\end{align}
 and is a polynomial of degree $2n^2-n+1$, divides $\Pf N(x)$ as a polynomial in $x$.
 Altogether, this implies that
\begin{align}\label{eq:def-Tx}
\Pf N(x)&=T(x)\,(x+1)_{2n}\,\prod_{s=2}^{n}\left(2x+2s\right)_{4n-4s+3},
\end{align}
for some  polynomial $T(x)$ in $x$ of degree at most $2n-2$. Therefore, in order to prove~\eqref{eq:Theorem2_bis},
it remains to show that $T(x)$ is equal to the polynomial $L(x)$ defined by
\begin{align}\label{eq:def-Lx}
L(x)&=\frac{2^{5n-1}}{n!(4n)!}\,\prod_{s=2}^{n}\frac{1}{(2s)_{4n-4s+3}}\;U_{n}(x),
\end{align}
with $U_n(x)$ given by~\eqref{eq:def-Ux}. Since $T(x)$ and $L(x)$  are polynomials in $x$ of degree at most $2n-2$,
it suffices to show that~$T(x)=L(x)$ holds for $2n-1$ distinct values of $x$. In Section~\ref{sec:PfN-5},
we determine the values of $T(x)$ and $L(x)$ at $x=-1,\dots,-n$ and consequently show that $T(x)=L(x)$ at $x=-1,\ldots,-n$.
Since, by~\eqref{eq:def-Tx} and~\eqref{eq:inv_shift-Pfaffian_N},  $T(x)=T(-2n-1-x)$ and, by~\eqref{eq:def-Lx}, $L(x)=L(-2n-1-x )$ for any $x$,
this shows at the same time that $T(x)=L(x)$ at $x=-2n,-2n+1,\dots,-n-1$. In total, there
holds that $T(x)=L(x)$ at $2n$ values of~$x$. This would complete the proof of~\eqref{eq:Theorem2_bis}.

%%%%%%%%%%%%%%%%%%%%%%%%%%%%%%%%%%%%%%%%%%%%%%%%%%%%%%%%%%%%%%%%%%%%%%%%%%%%%%%%%%%%%%%%%%%%%%%%%%%%%%%%%%%%%%%%%%%%%%%%%%
%%%%%%%%%%%%%%%%%%%%%%%%%%%%%                 Step 1                                     %%%%%%%%%%%%%%%%%%%%%%%%%%%%%%%%%
%%%%%%%%%%%%%%%%%%%%%%%%%%%%%%%%%%%%%%%%%%%%%%%%%%%%%%%%%%%%%%%%%%%%%%%%%%%%%%%%%%%%%%%%%%%%%%%%%%%%%%%%%%%%%%%%%%%%%%%%%%

%\subsection{$\Pf\,N(x)$ is a polynomial in $x$ of degree at most $2n^2+n-1$}\label{sec:PfN-1}
%    By Lemma~\ref{lem:tilings-Pfaffian_casB} and~\eqref{eq:def-Rij-bis}, the $(i,j)$-entry of $N(x)$ is a polynomial in $x$ of degree $i+j$
% if $1\leq i, j\leq 2n$ and $i\neq j$, of degree~$i$ if $1\leq i\leq 2n$ and $j=2n+2$, of degree $i-1$ if $1\leq i\leq 2n+1$ and $j=2n+2$. Moreover,
% $N_{i,j}(x)=0$ if $i=j$. This, combined with the formal definition of a Pfaffian, implies that $\Pf\,M(x)$ is a polynomial in $x$. Furthermore,
% in the defining expansion of the determinant $\det N(x)$, each nonzero term has  degree
%\begin{align*}
%\left(\sum_{i=1}^{2n}i\right)-1+\left(\sum_{j=1}^{2n}j\right)-1=4n^2+2n-2.
%\end{align*}
% This implies that $\det N(x)$ is of degree at most $4n^2+2n-2$. The claim then follows from~\eqref{eq:pf-det}.

%%%%%%%%%%%%%%%%%%%%%%%%%%%%%%%%%%%%%%%%%%%%%%%%%%%%%%%%%%%%%%%%%%%%%%%%%%%%%%%%%%%%%%%%%%%%%%%%%%%%%%%%%%%%%%%%%%%%%%%%%%
%%%%%%%%%%%%%%%%%%%%%%%%%%%%%                 Step 2                                   %%%%%%%%%%%%%%%%%%%%%%%%%%%%%%%%%
%%%%%%%%%%%%%%%%%%%%%%%%%%%%%%%%%%%%%%%%%%%%%%%%%%%%%%%%%%%%%%%%%%%%%%%%%%%%%%%%%%%%%%%%%%%%%%%%%%%%%%%%%%%%%%%%%%%%%%%%%%

\subsection{ $\Pf N(x)=(-1)^{n+1} \Pf N(-2n-1-x)$}\label{sec:PfN-2}
 The proof is quite similar to the proof of~\eqref{eq:inv_shift-Pfaffian} and requires only slight changes.
 Let $\Btil$ be the lower triangular matrix of size $2n+2$ defined by
\begin{align}
\Btil=\left(
    \begin{array}{ccc|rr}
       & & & & \\
      &\left(B_{i,j}\right)_{\substack{1\leq i\leq 2n \\ 1\leq j\leq 2n}} &&0& 0 \\
      &&&& \\\hline
      &0&&-1&0\\
      &0&&0&-1
    \end{array}
  \right),
  \quad\text{with $B_{i,j}=\binom{i-1}{j-1}$ for $1\leq i,j\leq 2n$,}
\end{align}
 and let $N^{(1)}$ be the skew-symmetric matrix (of size $2n+2$) defined by
 $N^{(1)}=\Btil\, N(x)\,{\Btil}^t$.

 We shall prove that the $(i,j)$-entry in $N^{(1)}$ is, up to the sign $(-1)^{i+j}$,
 the $(i,j)$-entry in $N(-2n-1-x)$.
First, from~Remark~\ref{rmk:Mn vs Nn}, the definition of $N^{(1)}$ and
what we have proved in Section~\ref{sec:PfM-2}, we have
\begin{align*}
N^{(1)}_{i,j}=M^{(1)}_{i,j}=(-1)^{i+j}M_{i,j}(-2n-1-x)=(-1)^{i+j}N_{i,j}(-2n-1-x)\quad\text{if
$i,j\neq 2n+1$},
\end{align*}
where $M^{(1)}$ is the matrix defined in Section~\ref{sec:PfM-2}. Furthermore, by definition of the
matrices $N^{(1)}$ and $\Btil$ and~\eqref{eq:Qij_B_lastcolumn}, there
holds for $1\leq i\leq 2n$
\begin{align*}
N^{(1)}_{i,2n+1}&=-\sum_{a = 1}^{i} {i-1 \choose{a-1}} N_{a,2n+2}(x)\\
&=\sum_{a = 1}^{i} {i-1 \choose{a-1}} \left( \binom {x+n}{a}+ \binom
{x+n+1}{a}-\binom {2x+2n+1}{a} \right),
\end{align*}
 which simplifies, by Chu--Vandermonde summation (see e.g.~\eqref{eq:Chu-Van}) and~\eqref{eq:Qij_B_lastcolumn}, to
\begin{align*}
N^{(1)}_{i,2n+1}&=(-1)^i\left(\binom {-x-n}{i}+
\binom{-x-n-1}{i}-\binom {-2x-2n-1}{i}\right)\\
&=(-1)^{i+1}N_{i,2n+1}(-2n-1-x).
\end{align*}

   Summarizing, we have shown that $N^{(1)}_{i,j}=(-1)^{i+j}N_{i,j}(-2n-1-x)$ for all $i,j$.
 With exactly the same reasoning used in Section~\ref{sec:PfM-2}, this leads to~\eqref{eq:inv_shift-Pfaffian_N}.

%%%%%%%%%%%%%%%%%%%%%%%%%%%%%%%%%%%%%%%%%%%%%%%%%%%%%%%%%%%%%%%%%%%%%%%%%%%%%%%%%%%%%%%%%%%%%%%%%%%%%%%%%%%%%%%%%%%%%%%%%%
%%%%%%%%%%%%%%%%%%%%%%%%%%%%%                 Step 3                                     %%%%%%%%%%%%%%%%%%%%%%%%%%%%%%%%%
%%%%%%%%%%%%%%%%%%%%%%%%%%%%%%%%%%%%%%%%%%%%%%%%%%%%%%%%%%%%%%%%%%%%%%%%%%%%%%%%%%%%%%%%%%%%%%%%%%%%%%%%%%%%%%%%%%%%%%%%%%

\subsection{ $\prod_{s=1}^n(x+s)^{s}$ divides $\Pf N(x)$}\label{sec:PfN-3}
Let $s$ be an integer with $1\leq s\leq n$. We claim that
\begin{align}\label{eq:comblineaire-s_N}
&{}\,\textrm{For $2n-2s+2\leq a\leq 2n$, the $a$-th row of the
matrix $N(-s)$ is null,}
\end{align}
 and that
\begin{align}\label{eq:comblineaire-s_2_N}
&\sum_{i=1}^{2n-2s+1}(-1)^{i-1}\left(2^{i-1}-\binom{i-1}{n-s}\right)
\cdot\left( \textrm{row $i$ of $N(-s)$} \right) -\left( \textrm{row
$(2n+1)$ of $N(-s)$ } \right)=0.
\end{align}

 As these are $2s$ linear combinations of the rows of the matrix
$N(x)$ that are linearly independent and vanish at $x=-s$,
this implies divisibility of $\Pf N(x)$ by~$(x+s)^{s}$.\\

\noindent\emph{Proof of~\eqref{eq:comblineaire-s_N}.} Suppose
$2n-2s+2\leq a\leq 2n$. It follows from~Remark~\ref{rmk:Mn vs Nn}
and~\eqref{eq:comblineaire-s} that $N_{a,j}(-s)=0$ if $1\leq j\leq
2n+2$ and $j\neq 2n+1$. Moreover, by~\eqref{eq:Qij_B_lastcolumn}, we
have
$$
N_{a,2n+1}(-s)=\binom {2n-2s+1}{a}-\binom {n-s}{a}-\binom
{n-s+1}{a},
$$
which is clearly zero since $a> 2n-2s+1\geq n-s+1$. \qed\\

\noindent\emph{Proof of~\eqref{eq:comblineaire-s_2_N}.} We have to
show that, for $1\leq j\leq 2n+2$, there holds
\begin{align}\label{eq:comblineaire-s_2_N-bis}
&\sum_{i=1}^{2n-2s+1}(-2)^{i-1}N_{i,j}(-s)
-\sum_{i=1}^{2n-2s+1}(-1)^{i-1}\binom{i-1}{n-s}N_{i,j}(-s)-N_{2n+1,j}(-s)=0.
\end{align}

(a) Suppose $j=2n+2$. By~\eqref{eq:def-Nij-Lr}, we have
$N_{i,2n+2}(-s)=\binom{n-s}{i-1}$ if $i\leq n-s+1$ and~0 otherwise.
Therefore, when $j=2n+2$, the left-hand side
of~\eqref{eq:comblineaire-s_2_N-bis} simplifies to
\begin{align*}
\sum_{i=1}^{n-s+1}
(-2)^{i-1}\binom{n-s}{i-1}-\sum_{i=1}^{n-s+1}(-1)^{i-1}\binom{i-1}{n-s}\binom{n-s}{i-1}=(-1)^{n-s}-(-1)^{n-s}=0,
\end{align*}
as desired. That the first (resp., second) sum in the above relation
is equal to $(-1)^{n-s}$ is immediate from the binomial theorem
(resp., from the fact that the term with index $i=n-s+1$ is the only
one term which is nonzero).\\

(b) Suppose $j=2n+1$. Using the expression~\eqref{eq:Qij_B_lastcolumn} for the matrix entry
$N_{i,2n+1}(-s)$ and noticing that $N_{2n+1,2n+1}(-s)=0$, we see
that~\eqref{eq:comblineaire-s_2_N-bis} reduces to
\begin{align*}
\sum_{i=1}^{2n-2s+1}(-1)^{i-1}\left(2^{i-1}-\binom{i-1}{n-s}\right)\left(\binom{2n-2s+1}{i}-{n-s\choose
i}-{n-s+1\choose i}\right) =0,
\end{align*}
which is immediate from the sum evaluations
\begin{align}
&\sum_{i=1}^{2n-2s+1} (-2)^{i-1}  {2n-2s+1\choose i}
=\sum_{i=1}^{2n-2s+1} (-2)^{i-1}\left({n-s\choose i}+{n-s+1\choose
i}\right) =1,\label{eq:comb-Sa}\\
 &\sum_{i=1}^{2n-2s+1} (-1)^{i-1}\binom{i-1}{n-s}\left({n-s\choose i}+{n-s+1\choose i}\right)
=(-1)^{n-s},\label{eq:comb-Sb}\\
&\sum_{i=1}^{2n-2s+1} (-1)^{i-1}\binom{i-1}{n-s} {2n-2s+1\choose i}
  =(-1)^{n-s}.\label{eq:comb-Sc}
\end{align}
 The two sums in~\eqref{eq:comb-Sa} can be easily evaluated by the binomial theorem.
 The sum in~\eqref{eq:comb-Sb} has only one term which is nonzero, the term with index $i=n-s+1$ which is equal to
 $(-1)^{n-s}$, whence~\eqref{eq:comb-Sb}. In the sum in~\eqref{eq:comb-Sc}, the terms with index~$i$ between~$1$ and
 $n-s$ vanish. By shifting the order of summation over $i$ by $n-s+1$ (that is we replace $i$ by $i+n-s+1$) and then using the relation
 $(-1)^{i}\binom{i+n-s}{n-s}=\binom{-n+s-1}{i}$, this sum becomes
 \begin{align*}
(-1)^{n-s}\sum_{i=0}^{n-s} \binom{-n+s-1}{i}\binom {2n-2s+1}{n-s-i},
\end{align*}
 which is equal, by Chu--Vandermonde summation,  to $(-1)^{n-s}\binom{n-s}{n-s}=(-1)^{n-s}$.\\

 (c) Suppose $2n-2s+2\leq j\leq 2n$. From~\eqref{eq:comblineaire-s_N} and the skew-symmetry of $N(-s)$
 the $j$-th column of $N(-s)$ is null, and thus the relation~\eqref{eq:comblineaire-s_2_N-bis} is clearly true.\\

 (d) Suppose $1\leq j\leq 2n-2s+1$. Using the expressions~\eqref{eq:def-Nij} and~\eqref{eq:Qij_B_lastcolumn}
 for the corresponding matrix entries, we see that~\eqref{eq:comblineaire-s_2_N-bis} reduces to
\begin{align}
\begin{split}\label{eq:reduction_comblineaire-s_2_N/j<=2n}
\sum_{i=1}^{2n-2s+1} (-1)^{i-1}&\left(2^{i-1}-\binom{i-1}{n-s}\right)\left(R_{i,j}(-s)+T_{i,j}(-s)-T_{j,i}(-s)\right)\\
 &\qquad+{2n-2s+1\choose j}-{n-s\choose j}-{n-s+1\choose j} =0,
\end{split}
\end{align}
 where $R_{i,j}(x)$ and $T_{i,j}(x)$ are defined as in~\eqref{eq:def-Rij} and~\eqref{eq:def-Tij}.

 From the expression~\eqref{eq:def-Tij} for $T_{i,j}(x)$ and~\eqref{eq:comb-Sa}--\eqref{eq:comb-Sc} we have
\begin{align}
 \begin{split}\label{eq:comb-Sd}
&\sum_{i=1}^{2n-2s+1}
(-1)^{i-1}\left(2^{i-1}-\binom{i-1}{n-s}\right)\left(T_{i,j}(-s)-T_{j,i}(-s)\right)\\
&\quad=\left(1-(-1)^{n-s}\right)\bigg{(} {n-s\choose
j}+{n-s+1\choose j}  -{2n-2s+1\choose j} \bigg{)}.
 \end{split}
\end{align}
From~\eqref{eq:sumbinom} we have
\begin{align}\label{eq:comb-Se}
\sum_{i=1}^{2n-2s+1} (-2)^{i-1}R_{i,j}(-s)&=\sum_{i=1}^{2n-2s+1}
(-2)^{i-1} \sum_{t=1}^{2n-2s+1}\frac{j-i}{t}{t\choose j}{t\choose
i}=0.
\end{align}
From the expression~\eqref{eq:def-Rij} for $R_{i,j}(x)$, we have
\begin{align*}
&\sum_{i=1}^{2n-2s+1}(-1)^{i-1} \binom{i-1}{n-s}R_{i,j}(-s)\\
&=\sum_{i=1}^{2n-2s+1}(-1)^{i-1}\binom{i-1}{n-s}\sum_{t=1}^{2n-2s+1}\binom{t-1}{j-1}\binom{t}{i}-\binom{t}{j}\binom{t-1}{i-1}\\
 &=\sum_{t=1}^{2n-2s+1}{t-1\choose j-1}\sum_{i=1}^{t}(-1)^{i-1} \binom{t} {i} \binom{i-1}{n-s}
 -\sum_{t=1}^{2n-2s+1}\binom{t}{j}\sum_{i=1}^{t} (-1)^{i-1}
 \binom{t-1}{i-1}\binom{i-1}{n-s},
\end{align*}
 where the last equality follows by interchanging the sum over $i$  with the sum over $t$.
 In the latter expression, the first inner sum is zero if $t\leq n-s$ (because of the binomial coefficient
 $\binom{i-1}{n-s}$) and is equal to $(-1)^{n-s}$ if $t>n-s$ (see~\eqref{eq:comb-Sc} and its proof). The second inner sum is
 $(-1)^{n-s}$ if $t=n-s+1$ and 0 otherwise, as it can be checked by using the relation
 $\binom{t-1}{i-1}\binom{i-1}{n-s}=\binom{t-1}{n-s}\binom{t-1-(n-s)}{i-1-(n-s)}$ and the binomial theorem.
 Altogether, this implies that
\begin{align*}
\sum_{i=1}^{2n-2s+1}(-1)^{i-1}
\binom{i-1}{n-s}R_{i,j}(-s)&=(-1)^{n-s}\sum_{t=n-s+1}^{2n-2s+1}{t-1\choose
j-1}-(-1)^{n-s} \binom{n-s+1}{j}\nonumber\\
&=(-1)^{n-s}\left({2n-2s+1\choose j}-{n-s\choose j}
-\binom{n-s+1}{j}\right),
\end{align*}
where the last equality follows from~\eqref{eq:BinomialSum}. This,
combined with~\eqref{eq:comb-Sd} and ~\eqref{eq:comb-Se}, leads
to~\eqref{eq:reduction_comblineaire-s_2_N/j<=2n}. This completes the
proof of~\eqref{eq:comblineaire-s_2_N-bis}.\qed

%%%%%%%%%%%%%%%%%%%%%%%%%%%%%%%%%%%%%%%%%%%%%%%%%%%%%%%%%%%%%%%%%%%%%%%%%%%%%%%%%%%%%%%%%%%%%%%%%%%%%%%%%%%%%%%%%%%%%%%%%%
%%%%%%%%%%%%%%%%%%%%%%%%%%%%%                 Step 4                                     %%%%%%%%%%%%%%%%%%%%%%%%%%%%%%%%%
%%%%%%%%%%%%%%%%%%%%%%%%%%%%%%%%%%%%%%%%%%%%%%%%%%%%%%%%%%%%%%%%%%%%%%%%%%%%%%%%%%%%%%%%%%%%%%%%%%%%%%%%%%%%%%%%%%%%%%%%%%

\subsection{ $\prod_{s=1}^{n-1}(x+s+3/2)^{s}$ divides $\Pf N(x)$}\label{sec:PfN-4}
Let $s$ be an integer with  $1\leq s\leq n-1$. We claim that for
$a=1,2,\ldots,2s$, we have
\begin{align}
\begin{split}\label{eq:comblineaire-s-3/2/N}
 &(-1)^{a}{n+s-\frac{1}{2}\choose a+1}\cdot\left( \textrm{row $(2n-1-a)$ of $N(-s-3/2)$ } \right)\\
 &\qquad-\left(n+s-\frac{1}{2}\right) {2n-1\choose a}\cdot \left(\textrm{row $2n-1$ of $N(-s-3/2)$ } \right)\\
 &\qquad\qquad-a {2n\choose a+1}\cdot\left( \textrm{row $(2n)$ of $N(-s-3/2)$ } \right)=0.
\end{split}
\end{align}
As these are $2s$ linear combinations of the rows of the matrix
$N(x)$ that are linearly independent and vanish at $x=-s-3/2$, this
implies divisibility of $\Pf N(x)$ by~$(x+s+\frac{3}{2})^{s}$.\\

\noindent\emph{Proof of~\eqref{eq:comblineaire-s-3/2/N}.} Suppose
$1\leq a\leq 2s$. We have to show that, for $1\leq j\leq 2n+2$,
\begin{align}
\begin{split}\label{eq:Comb-s-3/2-bis}
&(-1)^{a}{n+s-\frac{1}{2}\choose a+1}N_{2n-1-a,j}(-s-3/2)\\
 &\quad-\left(n+s-\frac{1}{2}\right) {2n-1\choose a}N_{2n-1,j}(-s-3/2)
 -a {2n\choose a+1}N_{2n,j}(-s-3/2)=0.
 \end{split}
\end{align}

 (a) Suppose $2n+1\leq j \leq 2n+2$. It follows from~\eqref{eq:Qij_B_lastcolumn} that
 $$N_{i,2n+1}(-s-3/2)=-\binom{n-s-3/2}{i}-\binom{n-s-1/2}{i}\quad\text{if $2n-2s-1\leq i\leq 2n$.}$$
 Furthermore, by~\eqref{eq:def-Nij-Lr}, we have $N_{i,2n+2}(-s-3/2)=\binom{n-s-3/2}{i-1}$.
 Therefore,~\eqref{eq:Comb-s-3/2-bis} when specialized to $j=2n+1$ and $j=2n+2$ reduces to, respectively, the identities
 \begin{align}
 \begin{split}\label{eq:comblineaire-s-3/2/N/j=2n+1}
&(-1)^{a+1}{n+s-\frac{1}{2}\choose a+1}\left(\binom{n-s-3/2}{2n-1-a}+\binom{n-s-1/2}{2n-1-a}\right)\\
&\qquad+\left(n+s-\frac{1}{2}\right) {2n-1\choose a}\left(\binom{n-s-3/2}{2n-1}+\binom{n-s-1/2}{2n-1}\right)\\
 &\qquad\qquad+a{2n\choose a+1}\left(\binom{n-s-3/2}{2n}+\binom{n-s-1/2}{2n}\right)=0,
\end{split}\\[0.5cm]
\begin{split}\label{eq:comblineaire-s-3/2/N/j=2n+2}
 &(-1)^{a}{n+s-\frac{1}{2}\choose a+1}{n-s-\frac{3}{2}\choose 2n-a-2}\\
 &\qquad-\left(n+s-\frac{1}{2}\right) {2n-1\choose a}{n-s-\frac{3}{2}\choose 2n-2}-a{2n\choose a+1}{n-s-\frac{3}{2}\choose 2n-1} =0.
\end{split}
\end{align}
The proof of these relations amounts to a routine verification and so is left to the reader.\\

 (b) Suppose $2n-2s-1\leq j\leq 2n$. Combining Remark~\ref{rmk:Mn vs Nn} with~\eqref{eq:Qij-s-1/2_annulation} and the skew-symmetry of
 $N(-s-3/2)$,
 we see that the $j$-th column of $N(-s-3/2)$ is null, and thus~\eqref{eq:Comb-s-3/2-bis} is clearly true.\\

 (c) Suppose $1\leq j\leq 2n-2s-2$. It follows from Remark~\ref{rmk:Mn vs Nn} and~\eqref{eq:Qij-s-1/2_simplification_i>} that
 $$
 N_{i,j}(-s-3/2)=-{2n-2s-2\choose j}\left({n-s-\frac{3}{2}\choose i}+{n-s-\frac{1}{2}\choose i} \right)
 $$
 for $2n-2s-1\leq i \leq 2n$. Inserting~this in~\eqref{eq:Comb-s-3/2-bis} and then
 dividing both sides of the resulted equality by ${2n-2s-2\choose j}$ gives
 the relation~\eqref{eq:comblineaire-s-3/2/N/j=2n+1}, which amounts to a routine verification.
 This completes the proof of~\eqref{eq:Comb-s-3/2-bis}.\qed

%%%%%%%%%%%%%%%%%%%%%%%%%%%%%%%%%%%%%%%%%%%%%%%%%%%%%%%%%%%%%%%%%%%%%%%%%%%%%%%%%%%%%%%%%%%%%%%%%%%%%%%%%%%%%%%%%%%%%%%%%%
%%%%%%%%%%%%%%%%%%%%%%%%%%%%%                 Step 5                                     %%%%%%%%%%%%%%%%%%%%%%%%%%%%%%%%%
%%%%%%%%%%%%%%%%%%%%%%%%%%%%%%%%%%%%%%%%%%%%%%%%%%%%%%%%%%%%%%%%%%%%%%%%%%%%%%%%%%%%%%%%%%%%%%%%%%%%%%%%%%%%%%%%%%%%%%%%%%
%\newpage
\subsection{ $T(x)=L(x)$ at $x=-1,\ldots,-n$}\label{sec:PfN-5}

Let $\s$ be a given integer with $1\leq \s \leq n$. It is easy to
evaluate $L(x)$ at~$x=-\s$. After a routine
calculation, we obtain
\begin{align}\label{eq:Ln at -s}
L(-\s)&=\frac{2^{5n-1}}{n!(4n)!}\,
\prod_{s=2}^{n}\frac{1}{(2s)_{4n-4s+3}}\left(\frac{3}{2}-\s\right)_{2n-1}
\big( (-1)^{\s-1}(2n-1)!!+(2n)!! \big).
\end{align}

The evaluation of $T(x)$ at $x=-\s$ is much more delicate. For the
same reason invoked in Section~\ref{sec:PfM-5}, we should
write~\eqref{eq:def-Tx} in the form
\begin{align*}
 T(x)=2^{2n-\sigma}\frac{1}{(x+\sigma)^{\sigma}} \Pf N(x)
\;&\prod_{s=1\atop s\neq \sigma}^n(2x+2s)^{-s}\;\prod_{s=1}^n(2x+4n+2-2s)^{-s}\\
&\times\prod_{s=1}^{n-1}\left(2x+3+2s\right)^{-s}\;\prod_{s=1}^{n-2}\left(2x+4n-1-2s\right)^{-s}.
\end{align*}
and subsequently specialize $x=-\sigma$. After some manipulation,
this leads to
\begin{align}
\begin{split}\label{eq:Tn at -s}
 T(-\sigma)=\left(\frac{1}{(x+\sigma)^{\sigma}} \Pf N(x)\right)\bigg\vert_{x=-\sigma}
\;&(-1)^{\s-1}2^{3-\s}\frac{(2\s-3)!\,(4n-2\s-1)!}{(\s-2)!\,(2n-\s-1)!}\\
&\times\frac{\prod_{s=1}^{n-\s}(2s-1)!}{\prod_{s=1}^{\s-1}(2s)!\;\prod_{s=n-\s+1}^{2n-\s}(2s)!}.
\end{split}
\end{align}

  Let $B(x)=\big(B_{i,j}(x)\big)_{1\leq i,j\leq 2n+2}$ denote the  matrix obtained from $N(x)$ by first adding
 $$\sum_{i=1}^{2n-2\s+1}(-1)^{i}\left(2^{i-1}-\binom{i-1}{n-\s}\right)\cdot\left( \textrm{row $i$ of $N(x)$} \right)$$
to row $2n+1$, and then, adding
$$\sum_{i=1}^{2n-2\s+1}(-1)^{i}\left(2^{i-1}-\binom{i-1}{n-\s}\right)\cdot\left( \textrm{column $i$ of $N(x)$} \right)$$
to column $2n+1$.
  Of course, we have  $\Pf B(x)=\Pf N(x)$, and it follows from~\eqref{eq:comblineaire-s_N} and~\eqref{eq:comblineaire-s_2_N} that
the $i$-th row, $2n-2 \sigma+2 \leq i \leq 2n+1$, of $B(-\sigma)$ is
null, or equivalently, $(x+\s)$ is a factor
 of each entry in the $i$-th row in matrix $B(x)$. Applying Lemma~\ref{lem:FactorizationPfaffian} to the matrix $B(x)$, we obtain
\begin{align}\label{eq:factorization_Pfaffian_N}
\left(\frac {1} {(x+\s)^{\s}}\Pf N(x)\right)\bigg\vert_{x=-\s}
=\left(\frac {1} {(x+\s)^{\s}}\Pf B(x)\right)\bigg\vert_{x=-\s}
=\Pf \widetilde B\,\cdot\,\Pf D,
\end{align}
 where $\widetilde B$ is the matrix which arises from $B(x)$ by deleting rows and columns $2n-2\s+2,2n-2\s+3,\ldots,2n+1$ and
subsequently specializing $x=-\s$, and
\begin{align}\label{eq:def-D}
D=\left( \left(\frac {1}
{x+\s}B_{i,j}(x)\right)\bigg\vert_{x=-\s}\right) _{2n-2\s+2\le
i,j\le 2n+1}.
\end{align}
  By Remark~\ref{rmk:Mn vs Nn} and the definition of the matrix $B(x)$, we have $\widetilde B=\widetilde M$ where $\widetilde M$ is
 defined by~\eqref{eq:def-Mtil}. We have shown in Section~\ref{sec:PfM-5} that $\Pf \widetilde M=1$, and thus $\Pf \widetilde B =1$.
 The evaluation of $\Pf D$ is much more complicated and so is postponed to the next section to simplify the readability of the paper.
\begin{lem}\label{lem:Pf(S)_N}
 \begin{align*}
\Pf D&= (-1)^{\s-1}
2^{5n-2}\frac{(\s-1)!\,(n-\s)!\,(2n-\s)!}{n!^2}\,
\big((2n)!!+(-1)^{\s+1}(2n-1)!!\big) \\
& \qquad\qquad\times\left(\frac{3}{2}-\s\right)_{2n-1}\,
\frac{\prod_{s=1}^{\s-2}(2s)! \big(\prod _{s=n-\s+1}
^{n}(2s)!\big)^2}{\prod_{s=2n-\s}^{2n}(2s)!}.
\end{align*}
\end{lem}

  If we substitute in~\eqref{eq:factorization_Pfaffian_N} the values obtained for $\Pf \widetilde B$ and $\Pf D$, and then insert the
 obtained result in~\eqref{eq:Tn at -s}, it is easy to check that $T(-\s)$ is equal to the right-hand member of~\eqref{eq:Ln at -s}.
 Consequently, we have $T(-\s)=L(-\s)$, as desired.

%%%%%%%%%%%%%%%%%%%%%%%%%%%%%%%%%%%%%%%%%%%%%%%%%%%%%%%%%%%%%%%%%%%%%%%%%%%%%%%%%%%%%%%%%%%%%%%%%%%%%%%%%%%%%%%%%%%%%%%%%%
%
%                         Evaluation of Pf D
%
%%%%%%%%%%%%%%%%%%%%%%%%%%%%%%%%%%%%%%%%%%%%%%%%%%%%%%%%%%%%%%%%%%%%%%%%%%%%%%%%%%%%%%%%%%%%%%%%%%%%%%%%%%%%%%%%%%%%%%%%%%
%\newpage
\section{Evaluation of $\Pf D$: Proof of Lemma~\ref{lem:Pf(S)_N}}
This section is dedicated to the evaluation of the Pfaffian of the
matrix $D$ defined by~\eqref{eq:def-D}.  We begin by describing more
explicitly the entries of the matrix $D$.  The next section provides
an efficient way to compute the Pfaffian of a skew-symmetric matrix
which differs from a Mehta-Wang matrix of even size only in its last
row and column. A more precise statement is given in
Proposition~\ref{prop:PerturbedMehtaWang}. This allows us to write
$\Pf D$ in the form of a multisum, which is evaluated in the last
subsection.

Throughout this section, $\s$ and $n$ are positive integers with $1\leq \s \leq n$. For the sake of simplicity,
we also set $N=n-\s$.

\subsection{The entries of the matrix $D$}
\begin{lem}\label{lem:matrixD}
  The matrix $D=(D_{i,j})_{1\leq i,j\leq 2\s}$ defined by~\eqref{eq:def-D} is a skew-symmetric matrix of size $2\s$.
 For $1\leq i,j\leq 2\s-1$, we have
 \begin{align}\label{eq:Dij}
D_{i,j} &={{\left( -1 \right) }^{i + j }}\,{\frac {2\, ( j-i )
\,(2N+i)!\,(2N+j)! }
           {(4N+i+j+2)!}}.
\end{align}
For $1\leq j\leq 2\s-1$, we have
\begin{multline}\label{eq:Dij-lastrow}
D_{2\s,j}=
\frac {(-1)^{j}2} {\binom {2N+j+1}{j-1}}\sum _{i=0} ^{2N}(-1)^i\left(2^i-\binom i{N}\right)\frac {2N+j-i} {i+1}\\
\cdot\sum _{\ell=0} ^{i}(-1)^\ell \binom {2N+j}{i-\ell}\binom {\ell+j-1}\ell\frac {1} {2N+\ell+j+2}\\
+\frac {(-1)^{N+j}} {(N+1)\binom
{2N+j+1}{j-1}}+\frac{(-1)^{j}(j-1)N!(N+j-1)!}{(2N+j+1)!}.
\end{multline}
\end{lem}

\begin{proof}
Combining the definitions of the matrices $D$
and $B(x)$ (see~\eqref{eq:def-D} and above it) with Remark~\ref{rmk:Mn vs Nn}, we see that the
$(i,j)$-entry of $D$ is for $1\leq i,j\leq 2\s-1$ equal
to the $(i,j)$-entry of the matrix $S$ defined in
\eqref{eq:def-MatS}. Equation~\eqref{eq:Dij} then follows
from~\eqref{eq:ValSij-simplification}.

Recall that $N=n-\s$ and suppose $1\leq j\leq 2\s-1$. Then, by~\eqref{eq:def-D} and the definition of the matrix $B(x)$,
we have
\begin{align}\label{eq:Dij-lastrow-pf}
&D_{2\s,j}=\left(\frac {N_{2n+1,2N+1+j}(x)} {x+\s}
         +\sum_{i=1}^{2N+1}
                (-1)^{i}\left(2^{i-1}-\binom{i-1}{N}\right)
                 \frac {N_{i,2N+1+j}(x)} {x+\s}\right)\bigg\vert_{x=-\s}.
\end{align}

 It is a routine matter to derive from~\eqref{eq:Qij_B_lastcolumn} that
\begin{align*}
 \frac{N_{2n+1,2N+1+j}(x)}{x+\s}\bigg\vert_{x=-\s}&=-\frac{\binom{2x+2N+2\s+1}{2N+1+j}}{x+\s}+\frac{\binom{x+N+\s}{2N+1+j}}{x+\s}
+\frac{\binom{x+N+\s+1}{2N+1+j}}{x+\s}\bigg\vert_{x=-\s}\\
 &=(-1)^{j}2\frac{(2N+1)!(j-1)!}{(2N+j+1)!}+(-1)^{N+j}\frac{(N+j)!N!}{(2N+j+1)!}\\
 &\quad +(-1)^{N+j-1}\frac{(N+1)!(N+j-1)!}{(2N+j+1)!}.
\end{align*}
 Similarly, combining~\eqref{eq:def-Nij} with~\eqref{eq:def-Tij} and~\eqref{eq:def-Rij-bis},
 after a straightforward calculation, we obtain for $1\leq i \leq 2N+1$
\begin{align*}
&\frac{N_{i,2N+1+j}(x)}{x+\s}\bigg\vert_{x=-\s}\\
&=\sum_{\ell=0}^{i-1}\frac{2N+1+j-i}{i}{2N+j\choose
i-1-\ell}
{\ell+2N+1+j\choose \ell}(-1)^{\ell+j-1}2\frac{(2N+2)!(\ell+j-1)!}{(2N+j+2+\ell)!}\\
 &\quad+{2N+1\choose i}
\bigg((-1)^{N+j}\frac{(N+j)!N!}{(2N+j+1)!}+(-1)^{N+j-1}\frac{(N+1)!(N+j-1)!}{(2N+j+1)!}
\bigg)\nonumber\\
 &\quad-(-1)^{j-1}2\frac{(2N+1)!(j-1)!}{(2N+j+1)!}
\left({N\choose i}+{N+1\choose i}\right).\nonumber
\end{align*}
Plugging the last two equalities into~\eqref{eq:Dij-lastrow-pf} and then using
~\eqref{eq:comb-Sa}--\eqref{eq:comb-Sc}, we obtain after some
simplification~\eqref{eq:Dij-lastrow}.
\end{proof}

\subsection{The Pfaffian of a perturbed Mehta-Wang matrix}

\begin{prop}\label{prop:PerturbedMehtaWang}
Let $s,R$ be positive integers and $A=\left(a_{i,j}\right)_{1\leq
i,j\leq 2s}$ be a skew-symmetric matrix such that, for $1\leq
i,j\leq 2 s-1$,
\begin{align*}\label{eq:defAij}
a_{i,j} &=\,\frac {( j-i )} {(R+i+j)!}.
\end{align*}
Then, the Pfaffian of $A$ satisfies the relation
\begin{align}
\Pf A&= -2^{s-1}(s-1)!\left(\sum_{j=1}^{2s-1} a_{2s,j}\cdot
\lambda_{j}\right) \prod_{i=0}^{s-2}
\frac{(2i+1)!}{(R+2s+1+2i)!},
\end{align}
where $\lambda_{j}$ is defined, for
$1\leq j\leq 2s-1$, by
\begin{align}\label{eq:defLambda}
\l_{j}=(-1)^{j+1}\sum_{k=0}^{s-1}
\frac{2^{-k}}{k!}\binom{2k}{j+k-s}(R+2s)_{j+k-s}.
\end{align}
\end{prop}

\begin{proof} Let $B=\left(B_{i,j}\right)$ be the upper triangular matrix of size $2s$ defined by
$$
B=\left(
\begin{array}{c|c}
  \l_1 & \l_2 \cdots \l_{2s-1}\,0 \\ \hline
  0 & \raisebox{-15pt}{{\large\mbox{{$I_{2s-1}$}}}} \\[-4ex]
  \vdots & \\[-0.5ex]
  0 &
\end{array}
\right),
$$
where $I_{2s-1}$ is for the identity matrix of size $2s-1$, and set $\widetilde A=B^t A B$.
 Clearly, $\tilde A$ is a skew-matrix of size $2s$ which differs from $A$ only in its first column and its first row, and
we have the relation
\begin{align}\label{eq:FirstColumn}
\tilde a_{i,1}=\sum_{j=1}^{2s-1} a_{i,j}\cdot \l_j.
\end{align}
  We claim that $\tilde a_{i,1}=0$ for $1\leq i\leq 2s-1$, so that
$$
\widetilde{A}= \left(\begin{array}{c|ccccc} 0&0&\cdots &0&-\tilde
a_{2s,1}\\\hline
0&a_{2,2}&\cdots &a_{2,2s-1}&a_{2,2s}\\
\vdots &\vdots & \ddots &\vdots &\vdots\\
0&a_{2s-1,2}&\cdots &a_{2s-1,2s-1}&a_{2s-1,2s}\\
\tilde a_{2s,1}&a_{2s,2}&\cdots &a_{2s,2s-1}&a_{2s,2s}
\end{array}\right).
$$
This, combined with the relation $\widetilde A=B^t A B$, would imply that
$$
\Pf \widetilde A=-\tilde a_{2s,1}\,\underset{2\leq i,j\leq 2s-1}\Pf\left(a_{i,j}\right)
\quad\text{and}\quad
\Pf A=\Pf \widetilde A/\det B=2^{s-1}(s-1)!\Pf \widetilde A.$$
Consequently, we would have
$$\Pf A=-2^{s-1}(s-1)!\tilde a_{2s,1}\,\underset{2\leq i,j\leq 2s-1}\Pf\left(\frac {j-i}
{(R+i+j)!}\right).
$$
 Proposition~\ref{prop:PerturbedMehtaWang} then would immediately follow
from~\eqref{eq:FirstColumn} and the Pfaffian evaluation~\eqref{eq:Mehta-Wang}.

  So, to complete the proof,
 it remains to check our claim that $\tilde a_{i,1}=0$ for $1\leq i\leq
 2s-1$.  By~\eqref{eq:FirstColumn} and \eqref{eq:defAij}--\eqref{eq:defLambda}, we have, for $1\leq i\leq
 2s-1$,
\begin{align}
\tilde a_{i,1}&=\sum_{j=1}^{2s-1}\frac {(j-i)} {(R+i+j)!}(-1)^{j+1}
\sum_{k=0}^{s-1} \frac{2^{-k}}{k!}\binom{2k}{j+k-s}(R+2s)_{j+k-s}\label{eq:comblineaireA_firstRows}\\
            &=\sum_{k=0}^{s-1}\sum_{j=0}^{2k} \frac{2^{-k}}{k!}
\frac {(-1)^{j+s-k} (j+s-k-i)} {(R+i+j+s-k)!}
\binom{2k}{j}(R+2s)_{j},\label{eq:comblineaireA_firstRows-bis}
\end{align}
where~\eqref{eq:comblineaireA_firstRows-bis} follows from~\eqref{eq:comblineaireA_firstRows} by first
interchanging the sums over $j$ and $k$ and then shifting the (now) inner sum over $j$ by $s-k$. The inner sum
in~\eqref{eq:comblineaireA_firstRows-bis} can be rewritten, by splitting the term $(j+s-k-i)$, as
 a sum of two ${}_2F_1$ series. After some manipulation, we obtain
\begin{align*}
\tilde a_{i,1}
&=\sum_{k=0}^{s-1}\frac{2^{-k}}{k!}\frac {(-1)^{s-k+1} (2k)(R+2s)}
{(R+i+s-k+1)!} \, {}_2F_1\left[{{R+2s+1,-2k+1}\atop
{R+i+s-k+2}};1\right]\\
&\qquad+\sum_{k=0}^{s-1}\frac{2^{-k}}{k!}\frac
{(-1)^{s-k} (s-k-i)} {(R+i+s-k)!} \,{}_2F_1\left[{{R+2s,-2k}\atop
{R+i+s-k+1}};1\right]\\
&=\sum_{k=1}^{s-1}\frac{2^{-k+1}}{k!}\frac {(-1)^{s-k+1} k(R+2s)(i-s-k+1)_{2k-1}} {(R+i+s+k)!}\\
&\qquad+\sum_{k=0}^{s-1}\frac{2^{-k}}{k!}\frac {(-1)^{s-k}
(s-k-i)(i-s-k+1)_{2k}} {(R+i+s+k)!},
\end{align*}
where the last equality follows from Chu-Vandermonde
summation formula. It is easily checked that each  summand in the two above sums
vanishes when $i=s$, whence~$\tilde a_{s,1}=0$. Suppose $i\neq s$.
Splitting the term $(s-k-i)$ in the second sum in the above
expression, and rewriting the (now) three sums in hypergeometric
notation, we arrive at
\begin{multline*}
\tilde a_{i,1}=\frac {(-1)^{s}(R+2s)(i-s)} {(R+i+s+1)!}\,
{}_2F_1\left[{{1+i-s,1+s-i}\atop
{R+i+s+2}};\frac{1}{2}\right]\\
+\frac{(-1)^{s}2^{-1}(i-s)_2} {(R+i+s+1)!}\,
{}_2F_1\left[{{2+i-s,1+s-i}\atop
{R+i+s+2}};\frac{1}{2}\right]\\
+\frac {(-1)^{s} (s-i)} {(R+i+s)!}\, {}_2F_1\left[{{1+i-s,s-i}\atop
{R+i+s+1}};\frac{1}{2}\right].
\end{multline*}
To see that the above expression vanishes, it suffices (after
multiplying the above expression by $(-1)^{s}(i-s)^{-1} (R+i+s+1)!$)
to prove that
\begin{align*}
  (c-a)\; {}_2F_1\left[{{a,b+1}\atop{c+1}};x\right]+(1-x)a\; {}_2F_1\left[{{a+1,b+1}\atop{c+1}};x\right]
  -c \; {}_2F_1\left[{{a,b}\atop{c}};x\right]=0,
\end{align*}
where $a=1+i-s$, $b=s-i$, $c=R+i+s+1$ and $x=1/2$. This identity can
be easily derived from Gauss contiguous relations for the ${}_2F_1$
series, or more directly, by extracting the coefficient of $x^n$ in
each side. In our case, we have to check that
\begin{align*}
  (c-a)\frac{(a)_n (b+1)_n}{(c+1)_n(1)_n}+a\left(\frac{(a+1)_n (b+1)_n}{(c+1)_n(1)_n}
  -\frac{(a+1)_{n-1} (b+1)_{n-1}}{(c+1)_{n-1}(1)_{n-1}}\right)
  -c \frac{(a)_n (b)_n}{(c)_n(1)_n}=0,
\end{align*}
which amounts to a routine computation. To summarize, we have proved
that $\tilde a_{i,1}=0$ for $1\leq i\leq 2s-1$. This completes the
proof of Proposition~\ref{prop:PerturbedMehtaWang}.
\end{proof}

%%%%%%%%%%%%%%%%%%%%%%%%%%%%%%%%%%%%%%%%%%%%%%%%%%%%%%%%%%%%%%%%%%%%%%%%%%%%%%%%%%%%%%%%%%%%%%%%%%%%%%%%%%%%%%%%%%%%%%%%%%%%%
%%%%%%%%%%%%%%%%%%%%%%%%%%%%%%%%%%%%%%%%%%%%%%%%%%%%%%%%%%%%%%%%%%%%%%%%%%%%%%%%%%%%%%%%%%%%%%%%%%%%%%%%%%%%%%%%%%%%%%%%%%%%%
%\newpage

\subsection{Proof of Lemma~\ref{lem:Pf(S)_N}}
 It follows from~\eqref{eq:Dij} that we can apply Proposition~\ref{prop:PerturbedMehtaWang}
to the matrix $$\left(\frac{(-1)^{i+j}D_{i,j}}{2(2N+i)!(2N+j)!}\right)_{1\leq i,j\leq 2\s}.$$
After some simplification, we obtain
\begin{align*}
\Pf D&= (-1)^{\s}2^{\s}\left(\prod_{i=1}^{2\s}(2N+i)!\right)\underset{1\leq i,j\leq 2\s}\Pf\left(\frac{(-1)^{i+j}D_{i,j}}{2(2N+i)!(2N+j)!}\right)\\
      &=(-1)^{\s}2^{2\s-2}(\s-1)!\left(\prod_{i=1}^{2\s-1}(2N+i)!\right) \left(\prod_{i=0}^{\s-2}
\frac{(2i+1)!}{(4N+2\s+3+2i)!}\right)\,V(N,\s),
\end{align*}
where
\begin{align}\label{eq:def-Vns}
 V(N,\s)=\sum_{j=1}^{2\s-1}\frac{D_{2\s,j}}{(2N+j)!}\sum_{k=0}^{\s-1}\frac{2^{-k}}{k!}\binom{2k}{j+k-\s}(4N+2+2\s)_{j+k-\s}.
\end{align}
This easily leads to the following reformulation of Lemma~\ref{lem:Pf(S)_N}.
\begin{prop}\label{prop:reformulation-lemma} Let $V(N,\s)$ be defined by~\eqref{eq:def-Vns}. Then,
\begin{align*}
 V(N,\s)=\frac{2^{N+2}(2N+\s+1)!}{(N+\s)!(4N+2\s+2)!}\left((-1)^{\s}(2N+2\s)!!-(2N+2\s-1)!!\right).
\end{align*}
\end{prop}
The first step towards a proof of the preceding result is to simplify expression~\eqref{eq:Dij-lastrow} for $D_{2\s,j}$.
We shall prove at the end of this section that
 \begin{align}\label{eq:simplification-Dij}
 \begin{split}
  \frac{(-1)^{j+1}D_{2\s,j}}{(2N+j)!}
&=\frac{(j-1)(2N+1)!}{(4N+j+3)!}\left(\frac{2(2N+1)!}{N!(N+1)!}-2^{2N+2}\right)\\
&\quad-\frac{(2N+2)!}{(4N+j+3)!}\left(\frac{4(2N+1)!}{N!(N+1)!}+2^{2N+3}\right)
-\frac{(j-1)(N+j-1)!N!}{(2N+j)!(2N+j+1)!}\\
&\quad+\frac{2(j-1)(2N+2)!}{(2N+j+1)!(N+1)!}\sum_{h=N}^{2N+1}\frac{(h+j-1)!(h+1)!}{(h-N)!(2N+j+h+2)!}.
 \end{split}
\end{align}
Therefore, by~\eqref{eq:def-Vns}, we have
\begin{align}
\begin{split}\label{eq:decompo-Vns}
 V(N,\s)&= (2N+1)!\left(\frac{2(2N+1)!}{N!(N+1)!}-2^{2N+2}\right)
 \sum_{j=1}^{2\s-1}\frac{(-1)^{j+1}(j-1)}{(4N+j+3)!}\l_j(N,\s)\\
 &-(2N+2)!\left(\frac{4(2N+1)!}{N!(N+1)!}+2^{2N+3}\right)
  \sum_{j=1}^{2\s-1}\frac{(-1)^{j+1}}{(4N+j+3)!}\l_j(N,\s)\\
 &-N!\sum_{j=1}^{2\s-1}\frac{(-1)^{j+1}(j-1)(N+j-1)!}{(2N+j)!(2N+j+1)!}\l_j(N,\s)\\
 &+\frac{2(2N+2)!}{(N+1)!}\sum_{j=1}^{2\s-1}\frac{(-1)^{j+1}(j-1)}{(2N+j+1)!}
  \sum_{h=N}^{2N+1}\frac{(h+j-1)!(h+1)!}{(h-N)!(2N+j+h+2)!}\l_j(N,\s),
 \end{split}
\end{align}
where we have set
\begin{align}\label{eq:def-lamdaNs}
\l_j(N,\s)=\sum_{k=0}^{\s-1}\frac{2^{-k}}{k!}\binom{2k}{j+k-\s}(4N+2+2\s)_{j+k-\s}.
\end{align}
 All the sums in~\eqref{eq:decompo-Vns} can be evaluated in closed-form expressions.
\begin{lem}\label{lem:multisum}
For all positive integers $N$ and $\s$ we have
 \begin{align}
 \begin{split}\label{eq:multisum1}
 & \sum_{j=1}^{2\s-1}\frac{(-1)^{j+1}(j-1) }{(4N+j+3)!}\l_j(N,\s)=0,
 \end{split}\\
\begin{split}\label{eq:multisum2}
 & \sum_{j=1}^{2\s-1}\frac{ (-1)^{j+1}}{(4N+j+3)!}\l_j(N,\s)
 =(-2)^{\s-1}\frac{(2N+\s+1)!}{(2N+2)!(4N+2\s+2)!},
 \end{split}\\
\begin{split}\label{eq:multisum3}
&\sum_{j=1}^{2\s-1}\frac{(-1)^{j+1}(j-1)(N+j-1)!}{(2N+j)!(2N+j+1)!}\l_j(N,\s)=0,
 \end{split}\\
 \begin{split}\label{eq:multisum4}
&  \sum _{j=1}
^{2\s-1}\sum_{h=N}^{2N+1}\frac{(-1)^{j+1}(j-1)\,(h+j-1)!\,(h+1)!}
     {(h-N)!\,(2N+j+h+2)!\,(2N+j+1)!}\l_j(N,\s)\\
&\qquad=\frac {(2N+\s+1)!} {(N+1)!\,(4N+2\s+2)!}
\left((-2)^{\s-1}-\frac {(N+1)!^2\,(2N+2\s)!}
{2^{\s-1}\,(N+\s)!^2\,(2N+2)!}\right).
 \end{split}
\end{align}
\end{lem}

At this point, we should notice that, after plugging these sum
evaluations into~\eqref{eq:decompo-Vns}, it becomes a routine matter
to verify Proposition~\ref{prop:reformulation-lemma}.  So, to complete the proof of Lemma~\ref{lem:Pf(S)_N},
it remains to prove the preceding lemma and~\eqref{eq:simplification-Dij}.\\

\noindent\textit{Proof of Lemma~\ref{lem:multisum}.} (1) The double
sum on the left-hand side of~\eqref{eq:multisum1}  is the $i=1$,
$s=\sigma$ and $R=4N+2$ specialization
to~\eqref{eq:comblineaireA_firstRows}, which was shown to be zero
in the proof of Proposition~\ref{prop:PerturbedMehtaWang}, whence~\eqref{eq:multisum1}.\\

(2) Let $S_1(N,\s)$ denote the double sum on the left-hand side of~\eqref{eq:multisum2}.
Interchanging the sums over $j$ and $k$ in $S_1(N,\s)$, we see that the
(now) inner sum over $j$ can be written as a ${}_2F_1$ series which
is summable by Chu-Vandermonde formula. To be precise, we obtain
\begin{align*}
S_1(N,\s)&=\sum_{j=1}^{2\s-1}\frac{
(-1)^{j+1}}{(4N+j+3)!}\sum_{k=0}^{\s-1}\frac{2^{-k}}{k!}\binom{2k}{j+k-\s}(4N+2+2\s)_{j+k-\s}\\
&=\sum_{k=0}^{\s-1}\frac{2^{-k}}{k!}\frac
{(-1)^{\s-k+1}}{(4N+\s-k+3)!}
\, {}_2F_1\left[{{2N+2\s+2,-2k}\atop{4N+4+\s-k}};1\right]\\
&=\sum_{k=0}^{\s-1}\frac{2^{-k}}{k!}\frac {(-1)^{\s-k+1}}{(4N+\s-k+3)!}
\frac{(2-\s-k)_{2k}}{(4N+4+\s-k)_{2k}}.
\end{align*}
Writing the above sum in hypergeometric notation, we arrive at
\begin{align*}
S_1(N,\s)&=\frac{(-1)^{\s+1}}{(4N+\s+3)!}\,{}_2F_1\left[{{\s-1,2-\s}\atop{4N+4+\s}};\frac{1}{2}\right]
=(-1)^{\s+1}2^{\s-1}\frac{(2N+\s+1)!}{(2N+2)!(4N+2\s+2)!},
\end{align*}
where the last equality  follows from Bailey's ${}_2F_1$ summation
formula. This ends the proof of~\eqref{eq:multisum2}.\\

(3) Let $S_2(N,\s)$ denote the double sum on the left-hand side of
\eqref{eq:multisum3}. We shall establish the recurrence
\begin{equation} \label{eq:rek1}
S_2(N,\s)+(4N+2\s+3)S_2(N,\s+1)=0,\quad \quad \s\ge1.
\end{equation}
Since it is easy to verify directly that $S_2(N,1)=0$, this would
immediately imply the claim.

In order to prove \eqref{eq:rek1}, we use the Gosper--Zeilberger
algorithm (cf.\ \cite{GospAB} and \cite[\S~II.5]{PeWZAA} --- the
particular implementation that we used is the {\sl Mathematica}
implementation by Paule and Schorn \cite{PaScAA}) to find that
\begin{equation} \label{eq:rek1f}
F_2(N,\s,j,k) +(4N+2\s+3)F_2(N,\s+1,j,k)= G_2(N,\s,j,k+1)-G_2(N,\s,j,k),
\end{equation}
where $F_2(N,\s,j,k)$ is the summand of the sum on the left-hand side
of \eqref{eq:multisum3}, that is,
$$
F_2(N,\s,j,k)=(-1)^{j+1}\frac {(j-1)\,(N+j-1)!} {(2N+j)!\,(2N+j+1)!}
\frac {1} {2^k\,k!}\binom {2k}{j+k-\s} (4N+2\s+2)_{j+k-\s},
$$
and
$$
G_2(n,\s,j,k)=\frac {j+k-\s} {4N+2\s+2}F_2(N,\s,j,k).
$$
We now sum both sides of \eqref{eq:rek1f} over $k$ between $0$ and
$\s-1$ and subsequently over $j$ between $1$ and $2\s-1$. Taking into
account the telescoping effect on the right-hand side when we
perform summation over $k$, we arrive at
\begin{multline*} %\label{}
S_2(N,\s)+(4N+2\s+3)S_2(N,\s+1)
-(4N+2\s+3)\sum _{j=1} ^{2\s-1}F_2(N,\s+1,j,\s)\\
-(4N+2\s+3)\big(F_2(N,\s+1,2\s,\s)+F_2(N,\s+1,2\s,\s-1)+F_2(N,\s+1,2\s+1,\s)\big)\\
=\sum _{j=1} ^{2\s-1}(G_2(N,\s,j,\s)-G_2(N,\s,j,0)).
\end{multline*}
After some simplification, this becomes
\begin{align*} \label{}
S_2(N,\s)&{}+(4N+2\s+3)S_2(N,\s+1)\\
&=
  \left( 4 N + 4 \s + 3 \right)
    \sum_{j = 1}^{2 \s-1}
      {\frac {{{\left( -1 \right) }^{j}}\,\left( j-1 \right) \,
          \left( N+j-1 \right) !\,\left( 2 \s \right) !\,
          ({ \textstyle 4 N + 2 \s+3}) _{j-1} }
        {{2^\s}\,\left( j-1 \right) !\,\left( 2 N+j \right) !\,
          \left(  2 N +j+1 \right) !\,\s!\,\left( 2 \s-j+1 \right)
!}}\\
&\kern2cm + {\frac {{2^{1 - \s}} \left( 4 N + 4 \s + 3 \right) \,
      \left( N + 2 \s -1\right) !\,
      ({ \textstyle 4 N + 2 \s+3}) _{2 \s-1} }
    {\left( N + \s+1 \right) \,\left( \s-2 \right) !\,
      {{\left( 2 N + 2 \s \right) !}^2}}} \\
&=  \left( 4 N + 4 \s + 3 \right)
    \sum_{j = 1}^{2 \s+1}
      {\frac {{{\left( -1 \right) }^{j}}\,\left( j-1 \right) \,
          \left( N+j-1 \right) !\,\left( 2 \s \right) !\,
          ({ \textstyle 4 N + 2 \s+3}) _{j-1} }
        {{2^\s}\,\left( j-1 \right) !\,\left( 2 N+j \right) !\,
          \left(  2 N +j+1 \right) !\,\s!\,\left( 2 \s-j+1 \right) !}}
.
\end{align*}
It remains to show that the right-hand side of this equation
vanishes. In order to see this, we write the sum over $j$ in
hypergeometric notation. Thereby we obtain
\begin{multline*} %\label{}
S_2(N,\s)+(4N+2\s+3)S_2(N,\s+1)\\= {\frac {2 \s\left( 4 N + 4 \s
+3\right) \,\left( 4 N + 2 \s +3\right) \,
     \left( N+1 \right) !}
   {{2^\s}\,\s!\,\left( 2 N+2 \right) !\,\left( 2 N+3 \right) !}}
     {} _{3} F _{2} \!\left [ \begin{matrix} { - 2 \s+1, 4 N + 2 \s+4,N+2}\\ {
      2 N+3, 2 N+4}\end{matrix} ; {\displaystyle 1}\right ].
\end{multline*}
This $_3F_2$-series can be evaluated by means of Watson's
$_3F_2$-summation (see \cite[(2.3.3.13); Appendix (III.23)]{SlatAC})
\begin{equation*} %\label{eq:S3233}
{} _{3} F _{2} \!\left [ \begin{matrix} { a, b, c}\\ { {\frac {1 + a
+ b} 2}, 2
   c}\end{matrix} ; {\displaystyle 1}\right ]
=
 \frac {\Gamma\left( {\frac 1 2}\right)\, \Gamma\left({\frac 1 2} + c\right)\,
   \Gamma\left(  {\frac 1 2} + {\frac a 2} +
   {\frac b 2}\right)\,
   \Gamma\left( {\frac 1 2} -{\frac a 2}-{\frac b 2}+c\right)}
 {\Gamma\left({\frac 1 2} + {\frac a
   2}\right)\, \Gamma\left( {\frac 1 2} + {\frac b 2}\right)\,
   \Gamma\left( {\frac 1 2} - {\frac a 2} + c\right)\,
   \Gamma\left( {\frac 1 2} -
   {\frac b 2} + c\right)}.
\end{equation*}
In fact, because of the term $\Gamma(\frac {1} {2}+\frac {a} {2})$
in the denominator on the right-hand side, the series vanishes
whenever $a$ is an odd negative integer. The recurrence
\eqref{eq:rek1} follows immediately. This completes the proof of~\eqref{eq:multisum3}.\\

(4) Again, we prove the claim by induction on $\s$. To start the
induction, we have to show that \eqref{eq:multisum4} holds for $\s=1$.
This is trivially the case, since in this case the sum over $j$
reduces to just one term, the term for $j=1$, so that both sides of
\eqref{eq:multisum4} vanish.

Let $S_3(N,\s)$ denote the triple sum on the left-hand side of
\eqref{eq:multisum4}. We now claim that
\begin{multline} \label{eq:rek2}
S_3(N,\s)+(4N+2\s+3)S_3(N,\s+1)\\= -\frac{ (4 N+4
   \s+3)\,(N+1)!\, (2 N+\s+1)!\, (2 N+2
   \s)!}
{2^\s\,(2 N+2)!\, (N+\s)!\,
   (N+\s+1)!\, (4 N+2 \s+2)!},
\quad \quad \s\ge1.
\end{multline}
Since the right-hand side of \eqref{eq:multisum4} satisfies the above
recurrence --- as is not difficult to check --- this would prove the
lemma.

In order to prove \eqref{eq:rek2}, we multiply both sides of
\eqref{eq:rek1} by
$$
\frac {(h+j-1)!\,(h+1)!} {(h-N)!\,(2N+j+h+2)!} \bigg/ \frac
{(N+j-1)!} {(2N+j)!}
$$
to find that
\begin{equation} \label{eq:rek2f}
F_3(N,\s,j,h,k) +(4N+2\s+3)F_3(N,\s+1,j,h,k)=
G_3(N,\s,j,h,k+1)-G_3(N,\s,j,h,k),
\end{equation}
where $F_3(N,\s,j,h,k)$ is the summand of the sum on the left-hand
side of \eqref{eq:multisum4}, that is,
\begin{multline*}
F_3(N,\s,j,h,k)=(-1)^{j+1}\frac {(j-1)\,(h+j-1)!\,(h+1)!}
{(h-N)!\,(2N+j+h+2)!\,(2N+j+1)!}\\
\times \frac {1} {2^k\,k!}\binom {2k}{j+k-\s} (4N+2\s+2)_{j+k-\s},
\end{multline*}
and
$$
G_3(n,\s,j,h,k)=\frac {j+k-\s} {4N+2\s+2}F_3(N,\s,j,h,k).
$$
We now sum both sides of \eqref{eq:rek2f} over $k$ between $0$ and
$\s-1$, and subsequently over $j$ between $1$ and $2\s-1$, and over
$h$ between $N$ and $2N+1$. Taking into account the telescoping
effect on the right-hand side when we perform summation over $k$, we
arrive at
\begin{multline*} %\label{}
S_3(N,\s)+(4N+2\s+3)S_3(N,\s+1)
-(4N+2\s+3)\sum _{j=1} ^{2\s-1}\sum_{h=N}^{2N+1}F_3(N,\s+1,j,h,\s)\\
-(4N+2\s+3)\sum_{h=N}^{2N+1}
\big(F_3(N,\s+1,2\s,h,\s)\kern6cm\\
\kern5cm
+F_3(N,\s+1,2\s,h,\s-1)+F_3(N,\s+1,2\s+1,h,\s)\big)\\
=\sum _{j=1}
^{2\s-1}\sum_{h=N}^{2N+1}\big(G_3(N,\s,j,h,\s)-G_3(N,\s,j,h,0)\big).
\end{multline*}
After some simplification, this becomes
\begin{align*} \label{}
S_3&(N,\s)+(4N+2\s+3)S_3(N,\s+1)\\
&=  -
    \sum _{h=N} ^{2N+1}\sum_{j = 2}^{2 \s-1}
      {\frac {{{\left( -1 \right) }^{j}}\,\left( 4 N + 4 \s + 3 \right) \,
          \left( h+j-1 \right) !\,\left( h+1 \right) !\,\left( 2 \s \right) !\,
          ({ \textstyle 4 N + 2 \s+3}) _{j-1} }
        {{2^\s}\,\left( j-2 \right) !\,
          \left( h-N\right) !\,\left( 2 N+j +h+2\right) !\,
          \left(  2 N +j+1 \right) !\,\s!\,\left( 2 \s-j+1 \right)!}}
\\
&\kern.5cm + \sum _{h=N} ^{2N+1}
  4\,\left( 4 N + 4 \s+3 \right)
   \big( 3 + 2 h + 5 N + 3 h N + 2 {N^2} + \s + h \s - 2 N \s \\
&\kern5cm -
       2 h N \s - 4 {N^2} \s - 4 {\s^2} - 2 h {\s^2} - 8 N {\s^2} -
       4 {\s^3} \big)\\
&\kern4cm \cdot {\frac { \left( h+1 \right) !\,
     \left( h + 2 \s -1\right) !\,
     ({ \textstyle 4 N + 2 \s+3}) _{2 \s-1} }
   {{2^\s}\,\left( h - N \right) !\,\left( \s-1 \right) !\,
     \left( 2 N + 2 \s+2 \right) !\,\left( 2 N + 2 \s+h+3 \right) !}}\\
&= -
    \sum _{h=N} ^{2N+1}\sum_{j = 2}^{2 \s+1}
      {\frac {{{\left( -1 \right) }^{j}}\,\left( 4 N + 4 \s + 3 \right) \,
          \left( h+j-1 \right) !\,\left( h+1 \right) !\,\left( 2 \s \right) !\,
          ({ \textstyle 4 N + 2 \s+3}) _{j-1} }
        {{2^\s}\,\left( j-2 \right) !\,
          \left( h-N\right) !\,\left( 2 N+j +h+2\right) !\,
          \left(  2 N +j+1 \right) !\,\s!\,\left( 2 \s-j+1 \right)!}}
.
\end{align*}
By writing the sum over $j$ in hypergeometric notation, this turns
into
\begin{multline*} \label{}
S_3(N,\s)+(4N+2\s+3)S_3(N,\s+1)\\
=-\sum_{h=N}^{2N+1} {\frac {\left( 4 N + 2 \s+3 \right) \,
      \left( 4 N + 4 \s +3\right) \,
      {{\left( h+1 \right) !}^2}}
    {2^{\s-1}\,\left( h - N \right) !\,\left( 2 N+3 \right) !\,
      \left( 2 N+h+4 \right) !\,\left( \s-1 \right) !}}\\
\cdot
      {} _{3} F _{2} \!\left [ \begin{matrix} { h+2, 4 N + 2 \s+4,- 2 \s+1}\\ {
       2 N+4, 2 N+h+5}\end{matrix} ; {\displaystyle 1}\right ]
.
\end{multline*}
Next we apply the contiguous relation
$$%C26
{} _{3} F _{2} \!\left [ \begin{matrix} { a, b, c}\\ {
d,e}\end{matrix} ;
   {\displaystyle z}\right ]  =
  {\frac{b  } {b - a}}
  {} _{3} F _{2} \!\left [ \begin{matrix} { a, b + 1, c}\\ {d,e}\end{matrix} ;
        {\displaystyle z}\right ] +
   {\frac {a  } {a - b}}
  {} _{3} F _{2} \!\left [ \begin{matrix} { a + 1, b, c}\\ { d,e}\end{matrix} ;
        {\displaystyle z}\right ].
$$
We obtain
\begin{multline*} \label{}
S_3(N,\s)+(4N+2\s+3)S_3(N,\s+1)\\
= -\sum_{h=N}^{2N+1} {\frac {\left( 4 N + 2 \s+3 \right) \,
       \left( 4 N + 2 \s+4 \right) \,\left( 4 N + 4 \s+3 \right) \,
       {{\left( h +1\right) !}^2}}
     {{2^{\s-1}}\,\left( 4 N + 2 \s-h+2 \right) \,\left( h - N \right) !\,
       \left( 2 N +3\right) !\,\left( 2 N+h+4 \right) !\,
       \left( \s -1\right) !}}\\
\cdot
       {} _{3} F _{2} \!\left [ \begin{matrix} { h+2, 4 N + 2 \s+5,  - 2 \s+1}\\
        { 2 N+4, 2 N+h+5}\end{matrix} ; {\displaystyle 1}\right ] \\
+\sum_{h=N}^{2N+1}
  {\frac {\left( 4 N + 2 \s+3 \right) \,
      \left( 4 N + 4 \s+3 \right) \,
      \left( h+1 \right) !\,\left( h+2 \right) !}
    {{2^{\s-1}}\,\left( 4 N + 2 \s-h+2 \right) \,\left( h - N \right) !\,
      \left( 2 N+3 \right) !\,\left( 2 N+h+4 \right) !\,
      \left( \s-1 \right) !}}\\
\cdot
      {} _{3} F _{2} \!\left [ \begin{matrix} { h+3, 4 N + 2 \s+4, 1 - 2 \s}\\ {
       2 N+4, 2 N+h+5}\end{matrix} ; {\displaystyle 1}\right ]
.
\end{multline*}
Both $_3F_2$-series can be evaluated by means of the
Pfaff--Saalsch\"utz summation (cf.\ \cite[(2.3.1.3); Appendix
(III.2)]{SlatAC})
$$
{}_{3} F_{2} \! \left[ \begin{matrix} { a, b, -n} \\ { c, 1 + a + b
- c - n} \end{matrix} ; {\displaystyle 1} \right ] =
{\frac{({\textstyle c-a})_{n} \, ({\textstyle c-b})
_{n}}{({\textstyle c})_{n} \, ({\textstyle c-a-b})_{n}}},
$$
where $n$ is a non-negative integer. If we apply the formula, then,
after some simplification, the above recurrence reduces to

\begin{multline*} \label{}
S_3(N,\s)+(4N+2\s+3)S_3(N,\s+1)\\
= 4\,\left( 4 N + 2 \s+3 \right) \,\left( 4 N + 4 \s+3 \right)
\sum_{h=N}^{2N+1}
\left( 1 + h - N + h N - 2 {N^2} - 3 \s - 4 N \s - 2 {\s^2} \right)\\
\cdot {\frac {
     {{\left( h+1 \right) !}^2}\,({ \textstyle 2 N+3}) _{2 \s-2} \,
     ({ \textstyle 2 - h + 2 N}) _{2 \s-2} }
   {{2^\s}\,\left( h - N \right) !\,\left( \s-1 \right) !\,
     \left( 2 N + 2 \s+2 \right) !\,\left( 2 N + 2 \s+h+3 \right) !}}
.
\end{multline*}
Let $S_4(N,\s)$ denote the right-hand sum. The Gosper--Zeilberger
algorithm then yields the recurrence
\begin{multline} \label{eq:rek3}
(2 N + 2 \s+1) (4 N + 4 \s+7)S_4(N,\s)\\-
  4 (2 N+3) (N + \s+2) (4 N + 2 \s+3) (4 N + 2 \s+5) (4 N +
     4 \s+3)S_4(N+1,\s)=0,
\end{multline}
with a half-page certificate, which we omit here for the sake of
brevity. Since it is straightforward to check that $S_4(1,\s)$ is
equal to the right-hand side of \eqref{eq:rek2} with $N=1$, and that
the right-hand side of \eqref{eq:rek2} satisfies the recurrence in
\eqref{eq:rek3}, the claimed recurrence \eqref{eq:rek2} follows.

This completes the proof of~\eqref{eq:multisum4}, and thus of the lemma.
\qed\\

We now turn our attention to the proof of~\eqref{eq:simplification-Dij}.\\

\noindent\textit{Proof of~\eqref{eq:simplification-Dij}.}
By~\eqref{eq:Dij-lastrow}, we have
\begin{align}\label{eq:decompo-Dij}
D_{2\s,j}=\frac {(-1)^{j}2} {\binom {2N+j+1}{j-1}}\cdot S+\frac {(-1)^{N+j}} {(N+1)\binom
{2N+j+1}{j-1}}+\frac{(-1)^{j}(j-1)N!(N+j-1)!}{(2N+j+1)!},
\end{align}
where $S$ stands for the double sum in~\eqref{eq:Dij-lastrow},
that is
\begin{multline*}
S=\sum _{i=0}
^{2N}(-1)^i\left(2^i-\binom i{N}\right)
\frac {2N+j-i} {i+1}\sum _{\ell=0} ^{i}(-1)^\ell \binom {2N+j} {i-\ell}\binom
{\ell+j-1}\ell\frac {1} {2N+\ell+j+2}.
\end{multline*}

By extending the sum over $i$, we rewrite $S$ as
\begin{multline*}
\sum _{i=0}
^{2N+1}(-1)^i\left(2^i-\binom i{N}\right)
\frac {2N+j-i} {i+1}\\
\cdot \sum _{\ell=0} ^{i}(-1)^\ell \binom {2N+j}
{i-\ell}\binom {\ell+j-1}\ell\frac {1} {2N+\ell+j+2}\\
+\left(2^{2N+1}-\binom
{2N+1}{N}\right)
\frac {j-1} {2N+2}\kern5cm\\
\cdot \sum _{\ell=0} ^{2N+1}(-1)^\ell \binom {2N+j}
{2N+1-\ell}\binom {\ell+j-1}\ell\frac {1} {2N+\ell+j+2}.
\end{multline*}

We concentrate now on the evaluation of the second sum over $\ell$.
It can be written as $_2F_1$-series which is summable by means of
the Chu--Vandermonde summation formula, so that we obtain
\begin{align*}
 & \left(2^{2N+1}-\binom{2N+1}{N}\right)\frac {j-1} {2N+2}\sum _{\ell=0} ^{2N+1}(-1)^\ell \binom {2N+j}
{2N+1-\ell}\binom {\ell+j-1}\ell\frac {1} {2N+\ell+j+2}\\
 &=\left(2^{2N+1}-\binom{2N+1}{N}\right)\frac {j-1} {2N+2}\binom {2N+j} {2N+1}\frac {1} {2N+j+2}
     {}_2 F_1\!\left[\begin{matrix} 2N+j+2,-2N-1\\2N+j+3\end{matrix}; 1\right]\\
 &= \left(2^{2N+1}-\binom{2N+1}{N}\right)\frac {j-1} {2N+2}\binom {2N+j} {2N+1}\frac {1} {2N+j+2}\frac {(1)_{2N+1}} {(2N+j+3)_{2N+1}}\\
 &= \frac {(j-1)\,(2N+j)!\,(2N+j+1)!} {(2N+2)\,(j-1)!\,(4N+j+3)!}\left(2^{2N+1}-\binom {2N+1}{N}\right).
\end{align*}
Consequently, our sum $S$ is equal to
\begin{align*}
&\sum _{i=0}
^{2N+1}(-1)^i\left(2^i-\binom i{N}\right)
\frac {2N+j-i} {i+1}\sum _{\ell=0} ^{i}(-1)^\ell \binom {2N+j}
{i-\ell}\binom {\ell+j-1}\ell\frac {1} {2N+\ell+j+2}\\
&+\frac {(j-1)\,(2N+j)!\,(2N+j+1)!} {(2N+2)\,(j-1)!\,(4N+j+3)!}\left(2^{2N+1}-\binom {2N+1}{N}\right).
\end{align*}

Next we apply the partial fraction expansion
$$
\frac {1} {(i+1)(2N+\ell+j+2)}=\frac {1} {2N+\ell+j-i+1}\left( \frac
{1} {i+1}-\frac {1} {2N+\ell+j+2}\right).
$$
Thus, we have
\begin{equation} \label{eq:1}
S=S_1-S_2 + \frac {(j-1)\,(2N+j)!\,(2N+j+1)!} {(2N+2)\,(j-1)!\,(4N+j+3)!}\left(2^{2N+1}-\binom {2N+1}{N}\right),
\end{equation}
where
\begin{multline*}
S_1=\frac {1} {(2N+j+1)} \sum _{i=0}
^{2N+1}(-1)^i\left(2^i-\binom i{N}\right)
\frac {2N+j-i} {i+1}\\
\cdot \sum _{\ell=0} ^{i}(-1)^\ell \binom {2N+j+1} {i-\ell}\binom
{\ell+j-1}\ell
\end{multline*}
and
\begin{multline*}
S_2=\frac {1} {(2N+j+1)} \sum _{i=0}
^{2N+1}(-1)^i\left(2^i-\binom i{N}\right)
(2N+j-i)\\
\cdot \sum _{\ell=0} ^{i}(-1)^\ell \binom {2N+j+1} {i-\ell}\binom
{\ell+j-1}\ell\frac {1} {2N+\ell+j+2}.
\end{multline*}

We start with the evaluation of $S_1$. We have
\begin{multline*}
\sum _{\ell=0} ^{i}(-1)^\ell \binom {2N+j+1} {i-\ell}\binom
{\ell+j-1}\ell =\binom {2N+j+1}i
{}_2 F_1\!\left[\begin{matrix} j,-i\\
2N+j-i+2\end{matrix}; 1\right]\\
=\binom {2N+j+1}i\frac {(2N-i+2)_i} {(2N+j-i+2)_i} =\binom {2N+1}i,
\end{multline*}
and therefore
\begin{align*} %\label{}
S_1&=\frac {1} {(2N+j+1)} \sum _{i=0}
^{2N+1}(-1)^i\left(2^i-\binom i{N}\right)
\frac {2N+j-i} {2N+2}\binom {2N+2}{i+1}\\
&=\frac {1} {(2N+j+1)}
\sum _{i=0} ^{2N+1}(-1)^i\left(2^i-\binom i{N}\right)\\
&\kern3cm \cdot \frac {1} {2N+2}\left( (2N+2)\binom
{2N+1}{i+1}+(j-1)\binom {2N+2}{i+1}
\right)\\
&=\frac {1} {(2N+j+1)} \bigg( -\frac {1}
{2}(-1)^{2N+1}+\frac {1} {2} +\frac {j-1} {2N+2}\left(-\frac {1}
{2}(-1)^{2N+2}+\frac {1} {2}\right)
\bigg)\\
&\kern1cm -\frac {1} {(2N+j+1)} \bigg(
(-1)^N\binom {2N+1}{N+1}
\,{}_2 F_1\!\left[\begin{matrix} N+1,-N\\
N+2\end{matrix}; 1\right]\\
&\kern3cm + \frac {j-1} {2N+2}(-1)^N\binom {2N+2}{N+1}
\,{}_2 F_1\!\left[\begin{matrix} N+1,-N-1\\
N+2\end{matrix}; 1\right]
\bigg)\\
&=\frac {1} {(2N+j+1)} \bigg( 1-
(-1)^N\binom {2N+1}{N+1}\frac {N!\,(N+1)!} {(2N+1)!}\\
&\kern5cm -(-1)^N \frac {j-1} {2N+2}\binom {2N+2}{N+1}\frac
{(N+1)!^2} {(2N+2)!}
\bigg)\\
&=\frac {1} {2N+j+1}-\frac{(-1)^N} {2N+2}.
\end{align*}

Next we consider the evaluation of $S_2$. We have
\begin{multline*}
S_2=\frac {1} {(2N+j+1)} \sum _{i=0}
^{2N+1}(-1)^i\left(2^i-\binom i{N}\right)
(2N+j-i)\\
\cdot \sum _{\ell=0} ^{i}(-1)^\ell \binom {\ell+j-1}\ell \frac {1}
{2\pi \sqrt{-1}} \bigg(\int _{C} ^{}\frac {(1+z)^{2N+j+1}}
{z^{i-\ell+1}}\,dz\bigg) \bigg(\int _{0}
^{1}x^{2N+\ell+j+1}\,dx\bigg),
\end{multline*}
where $C$ is a small contour in the complex plane encircling the
origin in positive orientation. The sum over $\ell$ can be extended
to a sum from $0$ to $\infty$ since the terms corresponding to
$\ell$'s which are larger than $i$ vanish. Hence, by evaluating the
(geometric) sum over $\ell$, we may rewrite this as
\begin{align*} %\label{}
S_2&= \frac {1} {(2N+j+1)2\pi \sqrt{-1}} \int
_{C} ^{}\int _{0} ^{1}
\sum _{i=0} ^{2N+1}(-1)^i\left(2^i-\binom i{N}\right)\\
&\kern7cm \cdot \frac {d} {dz}\big(z^{2N+j-i}\big) \frac
{(1+z)^{2N+j+1}} {z^{2N+j}}\frac {x^{2N+j+1}} {(1+xz)^j}\,dz
\,dx\\
&= -\frac {1} {(2N+j+1)2\pi \sqrt{-1}} \int _{C}
^{}\int _{0} ^{1}
\sum _{i=0} ^{2N+1}(-1)^i\left(2^{2N+1-i}-\binom {2N+1-i}{N}\right)\\
&\kern7cm \cdot \frac {d} {dz}\big(z^{i+j-1}\big) \frac
{(1+z)^{2N+j+1}} {z^{2N+j}}\frac {x^{2N+j+1}} {(1+xz)^j}\,dz \,dx.
\end{align*}
Consequently,
\begin{equation} \label{eq:2}
S_2=\frac {1} {(2N+j+1)}(-S_3+S_4),
\end{equation}
where
$$
S_3=\frac {1} {2\pi \sqrt{-1}} \int _{C} ^{}\int _{0} ^{1} \sum
_{i=0} ^{2N+1}(-1)^i2^{2N+1-i} \frac {d} {dz}\big(z^{i+j-1}\big)
\frac {(1+z)^{2N+j+1}} {z^{2N+j}}\frac {x^{2N+j+1}} {(1+xz)^j}\,dz
\,dx
$$
and
\begin{multline*}
S_4=\frac {1} {2\pi \sqrt{-1}} \int _{C} ^{}\int _{0} ^{1}
\sum _{i=0} ^{2N+1}(-1)^i\binom {2N+1-i}{N}\\
\cdot \frac {d} {dz}\big(z^{i+j-1}\big) \frac {(1+z)^{2N+j+1}}
{z^{2N+j}}\frac {x^{2N+j+1}} {(1+xz)^j}\,dz \,dx.
\end{multline*}

Now we evaluate $S_3$. Similarly to before, we may extend the sum
over $i$ to a sum from $0$ to $\infty$. The resulting sum is again a
geometric series, so that we obtain
\begin{align*}
S_3&=\frac {1} {2\pi \sqrt{-1}} \int _{C} ^{}\int _{0} ^{1} 2^{2N+1}
\frac {d} {dz}\left(\frac {z^{j-1}} {1+\frac {z} {2}}\right) \frac
{(1+z)^{2N+j+1}} {z^{2N+j}}\frac {x^{2N+j+1}} {(1+xz)^j}\,dz
\,dx\\
&=\frac {2^{2N+1}} {2\pi \sqrt{-1}} \int _{C} ^{}\int _{0} ^{1}
\frac {\left(j-1+(j-2)\frac {z} {2}\right)} {\left(1+\frac {z}
{2}\right)^2} \frac {(1+z)^{2N+j+1}} {z^{2N+2}}\frac {x^{2N+j+1}}
{(1+xz)^j}\,dz \,dx.
\end{align*}
Now we do the substitution $z\to z/(1-z)$. Thereby, we obtain
\begin{align*}
S_3&=\frac {2^{2N+1}} {2\pi \sqrt{-1}} \int _{C} ^{}\int _{0} ^{1}
\frac {\left(j-1-j\frac {z} {2}\right)} {\left(1-\frac {z}
{2}\right)^2} \frac {x^{2N+j+1}} {z^{2N+2}\,(1-z(1-x))^j}\,dz
\,dx\\
&=\frac {2^{2N+1}} {2\pi \sqrt{-1}} \int _{C} ^{}\int _{0} ^{1}
\frac {\left(j-1-j\frac {z} {2}\right)} {\left(1-\frac {z}
{2}\right)^2} \frac {x^{2N+j+1}} {z^{2N+2}} \sum_{h=0}^{2N+1}\binom
{h+j-1}h z^h(1-x)^h\,dz
\,dx\\
&=\frac {2^{2N+1}} {2\pi \sqrt{-1}} \int _{C} ^{} \frac
{\left(j-1-j\frac {z} {2}\right)} {\left(1-\frac {z} {2}\right)^2}
\sum_{h=0}^{2N+1}\binom {h+j-1}h \frac {(2N+j+1)!\,h!} {(2N+j+h+2)!}
\frac {1} {z^{2N+2-h}}
\,dz\\
&=2^{2N+1}
\sum_{h=0}^{2N+1}\binom {h+j-1}h \frac {(2N+j+1)!\,h!} {(2N+j+h+2)!}\\
&\kern5cm \cdot \left((j-1)(2N+2-h)-j(2N+1-h)
\right)2^{-2N-1+h}\\
&= \sum_{h=0}^{2N+1}\binom {h+j-1}h \frac {(2N+j+1)!\,h!}
{(2N+j+h+2)!}
\left(-2N-2+h+j\right)2^{h}\\
&= \sum_{h=0}^{2N+1} \bigg(\binom {h+j}{h+1} \frac
{(2N+j+1)!\,(h+1)!} {(2N+j+h+2)!}2^{h+1} - \binom {h+j-1}h \frac
{(2N+j+1)!\,h!} {(2N+j+h+1)!}2^h
\bigg)\\
&=\binom {2N+j+1}{2N+2} \frac {(2N+j+1)!\,(2N+2)!}
{(4N+j+3)!}2^{2N+2}
-1\\
&=2^{2N+2} \frac {(2N+j+1)!^2} {(4N+j+3)!\,(j-1)!} -1 .
\end{align*}

Finally, we compute $S_4$. In the earlier definition of $S_4$, we
may again extend the sum over $i$ to a sum from $0$ to $\infty$.
Using
$$
\binom {2N+1-i}{N}=\frac {1} {2\pi \sqrt{-1}} \int_{\tilde C}\frac
{(1+u)^{2N+1-i}} {u^{N+1}}\,du,
$$
where $\tilde C$ is a small contour in the complex plane encircling
the origin in positive orientation, we then obtain
\begin{align*}
S_4&=\frac {1} {(2\pi \sqrt{-1})^2} \int _{\tilde C} ^{} \int _{C}
^{}\int _{0} ^{1} \sum _{i=0} ^{\infty}(-1)^i \frac {d}
{dz}\big(z^{i+j-1}\big) \frac {(1+u)^{2N+1-i}} {u^{N+1}}
\\
&\kern7cm \cdot \frac {(1+z)^{2N+j+1}} {z^{2N+j}}\frac {x^{2N+j+1}}
{(1+xz)^j}
\,du\,dz\,dx\\
&=\frac {1} {(2\pi \sqrt{-1})^2} \int _{\tilde C} ^{} \int _{C}
^{}\int _{0} ^{1} \frac {d} {dz}\left(\frac {z^{j-1}} {1+\frac {z}
{1+u}}\right) \frac {(1+u)^{2N+1}} {u^{N+1}}
\\
&\kern7cm \cdot \frac {(1+z)^{2N+j+1}} {z^{2N+j}}\frac {x^{2N+j+1}}
{(1+xz)^j}
\,du\,dz\,dx\\
&=\frac {1} {(2\pi \sqrt{-1})^2} \int _{\tilde C} ^{} \int _{C}
^{}\int _{0} ^{1} \frac {(j-1)+(j-2)\frac {z} {1+u}} {\left(1+\frac
{z} {1+u}\right)^2} \frac {(1+u)^{2N+1}} {u^{N+1}}
\\
&\kern7cm \cdot \frac {(1+z)^{2N+j+1}} {z^{2N+2}}\frac {x^{2N+j+1}}
{(1+xz)^j} \,du\,dz\,dx.
\end{align*}
Again we do the substitution $z\to z/(1-z)$. Thereby, we obtain
\begin{align*}
S_4&=\frac {1} {(2\pi \sqrt{-1})^2} \int _{\tilde C} ^{} \int _{C}
^{}\int _{0} ^{1} \frac {(j-1)(1+u)-z(1+(j-1)u)} {\left(1-\frac {uz}
{1+u}\right)^2}
\\
&\kern7cm \cdot \frac {(1+u)^{2N}} {u^{N+1}} \frac {x^{2N+j+1}}
{z^{2N+2}(1-z(1-x))^j}
\,du\,dz\,dx\\
&=\frac {1} {(2\pi \sqrt{-1})^2} \int _{\tilde C} ^{} \int _{C}
^{}\int _{0} ^{1} \frac {(j-1)(1+u)-z(1+(j-1)u)} {\left(1-\frac {uz}
{1+u}\right)^2}
\\
&\kern4.5cm \cdot \frac {(1+u)^{2N}} {u^{N+1}} \frac {x^{2N+j+1}}
{z^{2N+2}} \sum _{h=0}^{2N+1}\binom {h+j-1}h z^h(1-x)^h
\,du\,dz\,dx\\
&=\frac {1} {(2\pi \sqrt{-1})^2} \int _{\tilde C} ^{} \int _{C} ^{}
(j-1)(1+u)-z(1+(j-1)u) \sum _{s=0}^{\infty}(s+1)\left(\frac {uz}
{1+u}\right)^s \frac {(1+u)^{2N}} {u^{N+1}}
\\
&\kern4.5cm \cdot \sum _{h=0}^{2N+1}\binom {h+j-1}h \frac
{(2N+j+1)!\,h!} {(2N+j+h+2)!}\frac {1} {z^{2N+2-h}}
\,du\,dz\\
&=\sum _{h=0}^{2N+1}\binom {h+j-1}h \frac {(2N+j+1)!\,h!}
{(2N+j+h+2)!}
\bigg((j-1)(2N+2-h)\binom {h}{h-N-1}\\
&\kern3cm -(2N+1-h)\binom {h}{h-N} -(j-1)(2N+1-h)\binom {h}{h-N-1}
\bigg)\\
&=\sum _{h=0}^{2N+1}\binom {h+j-1}h \frac {(2N+j+1)!\,h!}
{(2N+j+h+2)!}
\frac {h!} {(h-N)!\,(N+1)!}\\
&\kern3cm \cdot \big((j-1)(h-N) -(2N+1-h)(N+1)
\big)\\
&=\sum _{h=0}^{2N+1} \bigg(
\frac {(h+j-1)!\,h!\,(2N+j+1)!} {(j-1)!\,(h-N-1)!\,(N+1)!\,(2N+j+h+1)!}\\
&\kern3cm -\frac {(h+j)!\,(h+1)!\,(2N+j+1)!}
{(j-1)!\,(h-N)!\,(N+1)!\,(2N+j+h+2)!}
\bigg)\\
&\kern1cm +(j-1)\sum _{h=0}^{2N+1}
\frac {(h+j-1)!\,(h+1)!\,(2N+j+1)!} {(j-1)!\,(h-N)!\,(N+1)!\,(2N+j+h+2)!}\\
&= -\frac {(2N+1+j)!^2\,(2N+2)!} {(j-1)!\,(N+1)!^2\,(4N+j+3)!}
\\
&\kern1cm +(j-1)\sum _{h=0}^{2N+1} \frac
{(h+j-1)!\,(h+1)!\,(2N+j+1)!} {(j-1)!\,(h-N)!\,(N+1)!\,(2N+j+h+2)!}.
\end{align*}

In total, if we substitute our findings in \eqref{eq:1} and
\eqref{eq:2}, then we have shown that
\begin{align*}
S&=S_1-S_2 + \frac {(j-1)\,(2N+j)!\,(2N+j+1)!} {(2N+2)\,(j-1)!\,(4N+j+3)!}\left(2^{2N+1}-\binom {2N+1}{N}\right)\\
&=S_1+\frac {S_3-S_4} {2N+j+1}+ \frac {(j-1)\,(2N+j)!\,(2N+j+1)!} {(2N+2)\,(j-1)!\,(4N+j+3)!}\left(2^{2N+1}-\binom {2N+1}{N}\right)\\
&= \frac {1} {2N+j+1}-\frac {(-1)^N}{2N+2}+ \frac {(j-1)\,(2N+j)!\,(2N+j+1)!} {(2N+2)\,(j-1)!\,(4N+j+3)!}\left(2^{2N+1}-\binom {2N+1}{N}\right)\\
&\quad+\frac {1} {2N+j+1} \Bigg( 2^{2N+2}\frac {(2N+j+1)!^2}
{(4N+j+3)!\,(j-1)!}-1
+\frac {(2N+1+j)!^2\,(2N+2)!}{(j-1)!\,(N+1)!^2\,(4N+j+3)!}\\
&\kern2cm -(j-1)\sum _{h=0}^{2N+1} \frac
{(h+j-1)!\,(h+1)!\,(2N+j+1)!}
{(j-1)!\,(h-N)!\,(N+1)!\,(2N+j+h+2)!}\Bigg).
\end{align*}
A combination of the latter identity with~\eqref{eq:decompo-Dij} leads, after some trivial simplification, to~\eqref{eq:simplification-Dij}.
This ends the proof.
\qed

%%%%%%%%%%%%%%%%%%%%%%%%%%%%%%%%%%%%%%%%%%%%%%%%%%%%%%%%%%%%%%%%%%%%%%%%%%%%%%%%%%%%%%%%%%%%%%%%%%%%%%%%%%%%%%%%%%%%%%%%%%
%%%%%%%%%%%%%%%%%%%%%%%%%%%%%%%%%%%%%%%%%%%%%%%%%%%%%%%%%%%%%%%%%%%%%%%%%%%%%%%%%%%%%%%%%%%%%%%%%%%%%%%%%%%%%%%%%%%%%%%%%%
%%%%%%%%%%%%%%%%%%%%%%%%%%%%%%%%%%%%%%%%%%%%%%%%%%%%%%%%%%%%%%%%%%%%%%%%%%%%%%%%%%%%%%%%%%%%%%%%%%%%%%%%%%%%%%%%%%%%%%%%%%
%
%                         ASYMPTOTIQUE DU RATIO
%
%%%%%%%%%%%%%%%%%%%%%%%%%%%%%%%%%%%%%%%%%%%%%%%%%%%%%%%%%%%%%%%%%%%%%%%%%%%%%%%%%%%%%%%%%%%%%%%%%%%%%%%%%%%%%%%%%%%%%%%%%%
%%%%%%%%%%%%%%%%%%%%%%%%%%%%%%%%%%%%%%%%%%%%%%%%%%%%%%%%%%%%%%%%%%%%%%%%%%%%%%%%%%%%%%%%%%%%%%%%%%%%%%%%%%%%%%%%%%%%%%%%%%
%%%%%%%%%%%%%%%%%%%%%%%%%%%%%%%%%%%%%%%%%%%%%%%%%%%%%%%%%%%%%%%%%%%%%%%%%%%%%%%%%%%%%%%%%%%%%%%%%%%%%%%%%%%%%%%%%%%%%%%%%%

%\newpage
\section{Proof of Corollary~\ref{cor:Ncentered(2n,2x+1)}}

Throughout this section, all asymptotics are taken as $n\to\infty$ and $x\sim an$. By Theorem~\ref{thm:Ncentered(2n,2x+1)}, the probability
that a random vertically symmetric rhombus tiling of a $(2n,2x+1,2n)$ hexagon is
centered is equal to $R(n,x)$. Using expressions~\eqref{eq:def R(x,n)} and~\eqref{eq:def-Ux}, we see that $R(n,x)$ can be written in
hypergeometric notation (after reversing the order of  summation in~\eqref{eq:def-Ux} and
dividing the sum in four parts) as
\begin{align*}
R&(n,x)=2^{3n-2}\frac{(2x+2)!(x+2n)!}{n!(x+1)!(2x+4n)!}\\
       \times&\bigg{(} \frac{(2n-1)!!(\frac{3}{2}-n)_{2n-1}(x+1)_{n-1}(x+n+1)_{n}}{(n-1)!n!}
                         {}_{4}F_3\left[{{1,n+\frac{1}{2},1-n,-n-x}\atop {n+1,\frac{3}{2}-n, 1-n-x}};-1\right]\\
      &- \frac{(2n-1)!!(\frac{3}{2}-n)_{2n-1}(x+1)_{n}(x+n+2)_{n-1}}{(n-1)!n!}
                         {}_{4}F_3\left[{{1,x+n+1,n+\frac{1}{2},1-n}\atop {n+1,\frac{3}{2}-n}, x+n+2};-1\right]\\
      &+\frac{(-1)^{n+1}(2n)!!(\frac{3}{2}-n)_{2n-1}(x+1)_{n-1}(x+n+1)_{n}}{(n-1)!n!}
                         {}_{4}F_3\left[{{1,n+\frac{1}{2},1-n,-n-x}\atop {n+1,\frac{3}{2}-n, 1-n-x}};1\right]\\
      &- \frac{(-1)^{n+1}(2n)!!(\frac{3}{2}-n)_{2n-1}(x+1)_{n}(x+n+2)_{n-1}}{(n-1)!n!}
                         {}_{4}F_3\left[{{1,x+n+1,n+\frac{1}{2},1-n}\atop {n+1,\frac{3}{2}-n, x+n+2}};1\right]\bigg{)},
\end{align*}
which simplifies to
\begin{align}\label{eq:simplification-ratio}
R(n,x)&=(-1)^{n+1}2^{2n-2}\frac{ (x+2n)}{n(x+n)}
          \frac{\Gamma(2x+2)\Gamma(x+2n)^2\Gamma(n+\frac{1}{2})}{\Gamma(x+1)^2\Gamma(2x+4n)\Gamma(\frac{3}{2}-n)\Gamma(n)^2}\\
       &\quad\times\bigg{(} (-1)^{n+1}\frac{(2n)!}{4^n\,n!^2}
                         {}_{4}F_3\left[{{1,n+\frac{1}{2},1-n,-n-x}\atop {n+1,\frac{3}{2}-n, 1-n-x}};-1\right]\nonumber\\
       &\quad\quad  - (-1)^{n+1}\frac{(2n)!}{4^n\,n!^2}\frac{x+n}{x+n+1}
            {}_{4}F_3\left[{{1,x+n+1,n+\frac{1}{2},1-n}\atop {n+1,\frac{3}{2}-n, x+n+2}};-1\right]\nonumber\\
      &\quad\quad\quad+ {}_{4}F_3\left[{{1,n+\frac{1}{2},1-n,-n-x}\atop {n+1,\frac{3}{2}-n, 1-n-x}};1\right]\\
      &\quad\quad\quad\quad-\frac{x+n}{x+n+1} {}_{4}F_3\left[{{1,x+n+1,n+\frac{1}{2},1-n}\atop {n+1,\frac{3}{2}-n, x+n+2}};1\right]\bigg{)}.\nonumber
\end{align}

Using Stirling's formula, it is a routine matter to determine the asymptotic behavior
of the term in the first row in~\eqref{eq:simplification-ratio}:
\begin{align}\label{eq:asympt-facteur-ratio}
(-1)^{n+1}2^{2n-2}\frac{ (x+2n)}{n(x+n)}
          \frac{\Gamma(2x+2)\Gamma(x+2n)^2\Gamma(n+\frac{1}{2})}{\Gamma(x+1)^2\Gamma(2x+4n)\Gamma(\frac{3}{2}-n)\Gamma(n)^2}
          \sim\frac{\sqrt{a(a+2)}}{\pi (a+1)}\frac{1}{n}.
\end{align}
To deal with the ${}_{4}F_3$-series in~\eqref{eq:simplification-ratio}, we shall use the next lemma.

\begin{lem}\label{lem:asympt-hypergeoseries}
Let $b$ be a real number with $|b|>1$. Then,  for any nonnegative integer $r$
and any sequence $(b_n)_{n\geq 1}$ with $b_n\to b$ as
$n\to\infty$, we have
\begin{align}\label{eq:asympt-hypergeoseries}
\lim_{n\to\infty} \frac{1}{n} \,
{}_{4}F_3\left[{{1,n+\frac{1}{2},1-n, b_n n+r}\atop
{n+1,\frac{3}{2}-n}, b_n n+r+1};1\right]
=\frac{2b}{b+1}\sqrt{\frac{b+1}{b-1}}
\arctan\left(\sqrt{\frac{b-1}{b+1}}\right).
\end{align}
\end{lem}
Before we prove this result, we show how it leads to
Corollary~\ref{cor:Ncentered(2n,2x+1)}. First, Lemma~\eqref{eq:asympt-hypergeoseries}
gives the asymptotic behavior of the last two ${}_{4}F_3$-series in~\eqref{eq:simplification-ratio}. Moreover,
it is easily checked that the absolute value of the first (resp., second) ${}_{4}F_3$-series is less than the third ${}_{4}F_3$-series
(resp., fourth) ${}_{4}F_3$-series  in~\eqref{eq:simplification-ratio} which is
$O(n)$ by Lemma~\ref{lem:asympt-hypergeoseries}. This, combined with the fact that $\frac{(2n)!}{4^n\,n!^2}\sim (\pi n)^{-1/2}$,
shows that the contribution of the first two ${}_{4}F_3$-series in~\eqref{eq:simplification-ratio} is negligible in the limit.
Altogether,  with~\eqref{eq:asympt-facteur-ratio} and Lemma~\ref{lem:asympt-hypergeoseries}, we see that $R(n,x)$ is
asymptotically equivalent to
\begin{align*}
&\frac{\sqrt{a(a+2)}}{\pi (a+1)} \left(
  \frac{2(a+1)}{a}\sqrt{\frac{a}{a+2}} \arctan\left(\sqrt{\frac{a+2}{a}}\right)
-\frac{2(a+1)}{a+2}\sqrt{\frac{a+2}{a}}
\arctan\left(\sqrt{\frac{a}{a+2}}\right)
\right)\\
&=\frac{2}{\pi}\arctan\left(\frac{1}{\sqrt{a(a+2)}}\right)=\frac{2}{\pi}\arcsin\left(\frac{1}{a+1}\right),
\end{align*}
as desired.
%(the first equality is a consequence of the summation formula for $\arctan$ and the last is a corollary of the relation
%$\arctan(x)=\arcsin(x/\sqrt{x^2+1})$)
To conclude the proof of Corollary~\ref{cor:Ncentered(2n,2x+1)}, it remains to prove Lemma~\ref{lem:asympt-hypergeoseries}.\\

\noindent\emph{Proof of Lemma~\ref{lem:asympt-hypergeoseries}}. If we write
the  ${}_{4}F_3$-series in~\eqref{eq:asympt-hypergeoseries}
 explicitly as a sum over $k$, after some simplification, we obtain
\begin{align}
 {}_{4}F_3&\left[{{1,n+\frac{1}{2},1-n, b_n n+r,}\atop
{n+1,\frac{3}{2}-n}, b_n n+r+1};1\right]\nonumber\\
&\qquad=  \sum_{k=0}^{n-1}\frac{(n+\frac{1}{2})_k(1-n)_k (b_n n+r)_k}{(n+1)_k(\frac{3}{2}-n)_k (b_n n+r+1)_k}\nonumber\\
&\qquad=  \sum_{k=0}^{n-1}\frac{(b_n n+r)n(n-\frac{1}{2})}{(b_n n+r+k)}
                                              \frac{ \Gamma(n)^2\Gamma(n+k+\frac{1}{2})\Gamma(n-k-\frac{1}{2})}
                                                        { \Gamma(n+\frac{1}{2})^2\Gamma(n+k+1)\Gamma(n-k)}.\label{eq:hypergeo-explicite_lem}
\end{align}
Let us denote the summand in the latter sum by $F(n,k)$. It is easy to
check  that ${F(n,k)>0}$ for $0\leq k\leq n-1$ and that
\begin{align}\label{eq:asymptoLastTerm}
F(n,0)=1\quad\text{and}\quad F(n,n-1)\sim
\frac{b}{b+1}\sqrt{\frac{\pi }{2}}n^{1/2}  \quad\text{as
$n\to\infty$.}
\end{align}
Moreover, a routine computation shows that $(\frac{\partial^2}{\partial k^2} F(n,k))/F(n,k)$ is equal to
\begin{align*}
&\bigg{(}\psi(n+k+1/2)-\psi(n-k-1/2)-\psi(n+k+1)+\psi(n-k)-\frac{1}{b_n+k+r} \bigg{)}^2\\
&\;+\bigg{(}\psi_1(n+k+1/2)+\psi_1(n-k-1/2)-\psi_1(n+k+1)-\psi_1(n-k)+\frac{1}{(b_n+k+r)^2}
\bigg{)},
\end{align*}
where $\psi$ and $\psi_1$ are the digamma and trigamma functions
defined by $\psi(x)=\frac{d}{dx}\ln(\Gamma(x))$ and
$\psi_1(x)=\frac{d^2}{dx^2}\ln(\Gamma(x))=\frac{d}{dx}\psi(x)$.
Since the trigamma function $\psi_1$ is decreasing, the above expression is positive.
Consequently, for fixed $n\geq 1$, the summand $F(n,k)$ is convex as a function of
$k$. Combined with~\eqref{eq:asymptoLastTerm}, this implies that the sum in~\eqref{eq:hypergeo-explicite_lem} may be approximated by an
integral and
\begin{align}\label{eq:serie-integrale}
\sum_{k=0}^{n-1}F(n,k)&=\int_0^{n-1}F(n,k)\, dk+O(n^{1/2})
                                                    =\int_{n^{1/3}}^{n-n^{1/3}}F(n,k) \,dk+O(n^{1/2+1/3}).
\end{align}
Using the expression~\eqref{eq:hypergeo-explicite_lem} for $F(n,k)$
and the asymptotic approximation
\begin{align*}
\Gamma(z+\tfrac{1}{2})&=z^{-1/2}\Gamma(z+1)\,\left(1+O\left(z^{-1}\right)\right),\qquad  z\to\infty,
\end{align*}
 we see after some manipulation that
\begin{align*}
F(n,k)=\frac{b_n n}{(b_n  n+k) \sqrt{1-\frac{k^2}{n^2}}}\left(1+O(n^{-1/3})\right),
\qquad\text{for $n^{1/3}\leq k\leq
n-n^{1/3}$}.
\end{align*}
Combined with~\eqref{eq:serie-integrale}, this leads to
\begin{align*}
\sum_{k=0}^{n-1}F(n,k) & =\left(\int_{n^{1/3}}^{n-n^{1/3}}\frac{b_n
n}{(b_n  n+k) \sqrt{1-\frac{k^2}{n^2}}} \,dk\right)
     \left(  1+O(n^{-1/3})      \right) +O(n^{5/6}),
\end{align*}
which gives after the substitution $y=k/n$
\begin{align*}
\sum_{k=0}^{n-1}F(n,k) &
=n\,\left(\int_{n^{-2/3}}^{1-n^{-2/3}}\frac{b_n}{(b_n+y)
\sqrt{1-y^2}} \,dy\right)
     \left(  1+O(n^{-1/3})      \right) +O(n^{5/6}).
\end{align*}
Dividing both parts by $n$ and taking the limit, we obtain
\begin{align*}
\lim_{n\to\infty} \frac{1}{n} \,
{}_{4}F_3\left[{{1,n+\frac{1}{2},1-n, b_n n+r,}\atop
{n+1,\frac{3}{2}-n}, b_n n+r+1};1\right]
=\int_0^1\frac{b}{b+y}\frac{1}{\sqrt{1-y^2}} dy.
\end{align*}
To finish the proof of the lemma,  it remains to check that the
above integral is equal  to the right-hand side of~\eqref{eq:asympt-hypergeoseries},
which amounts to a routine computation.
\qed

%%%%%%%%%%%%%%%%%%%%%%%%%%%%%%%%%%%%%%%%%%%%%%%%%%%%%%%%%%%%%%%%%%%%%%%%%%%%%%%%%%%%%%%%%%%%%%%%%%%%%%%%%%%%%%%%%%%%%%%%%%
%%%%%%%%%%%%%%%%%%%%%%%%%%%%%                BIBLIOGRAPHIE                       %%%%%%%%%%%%%%%%%%%%%%%%%%%%%%%%%
%%%%%%%%%%%%%%%%%%%%%%%%%%%%%%%%%%%%%%%%%%%%%%%%%%%%%%%%%%%%%%%%%%%%%%%%%%%%%%%%%%%%%%%%%%%%%%%%%%%%%%%%%%%%%%%%%%%%%%%%%%


\begin{thebibliography}{99}
%%%%%%%%%%%%%%%%%%%%%%%%%%%%%%%%%%%%%%%%%%%%%%%%%%%%%%%%%%%%%%%%%%%%%%%%%%%
\bibitem{Andrews} G. E. Andrews. Plane partitions I: The MacMahon conjecture. Studies
in foundations and combinatorics, G.-C. Rota ed., Adv. in Math.
Suppl. Studies, Vol. 1, 1978, pp. 131--150.
%%%%%%%%%%%%%%%%%%%%%%%%%%%%%%%%%%%%%%%%%%%%%%%%%%%%%%%%%%%%%%%%%%%%%%%%%%%
\bibitem{BailAA} W.\,N.~Bailey. {\em Generalized Hypergeometric Series}, Cambridge
University Press, Cambridge, 1935.
 %%%%%%%%%%%%%%%%%%%%%%%%%%%%%%%%%%%%%%%%%%%%%%%%%%%%%%%%%%%%%%%%%%%%%%%%%%%
 \bibitem{Ciu} M.~Ciucu. Enumeration of lozenge tilings of punctured hexagons. J. Combin. Theory Ser. A 83 (1998), 268--272. 
%%%%%%%%%%%%%%%%%%%%%%%%%%%%%%%%%%%%%%%%%%%%%%%%%%%%%%%%%%%%%%%%%%%%%%%%%%%
 \bibitem{Ciu-Zare} M. Ciucu, T. Eisenk\"olbl, C. Krattenthaler and D. Zare. Enumeration of lozenge tilings of hexagons with a central triangular hole. 
 J. Combin. Theory Ser. A 95 (2001), 251--334. 
%%%%%%%%%%%%%%%%%%%%%%%%%%%%%%%%%%%%%%%%%%%%%%%%%%%%%%%%%%%%%%%%%%%%%%%%%%%
\bibitem{CiuKrat-centered} M.~Ciucu and C.~Krattenthaler.
 The number of centered lozenge tilings of a symmetric hexagon.
 J. Combin. Theory Ser. A  86 (1999), no. 1, 103--126.
 %%%%%%%%%%%%%%%%%%%%%%%%%%%%%%%%%%%%%%%%%%%%%%%%%%%%%%%%%%%%%%%%%%%%%%%%%%%
 \bibitem{CiuKrat-factor} M.~Ciucu and C.~Krattenthaler.
  A factorization theorem for classical group characters, with applications to plane partitions and rhombus tilings. In:
 Advances in Combinatorial Mathematics: Proceedings of the Waterloo Workshop in Computer Algebra 2008, I. Kotsireas, E. Zima (eds.), Springer-Verlag, 2010,
 pp. 39--60.
 %%%%%%%%%%%%%%%%%%%%%%%%%%%%%%%%%%%%%%%%%%%%%%%%%%%%%%%%%%%%%%%%%%%%%%%%%%%
\bibitem{CiuKrat-interaction} M.~Ciucu and C.~Krattenthaler.
 The interaction of a gap with a free boundary  in a two dimensional dimer system.
 Comm. Math. Phys.  302 (2011), no. 1, 253--289.
 %%%%%%%%%%%%%%%%%%%%%%%%%%%%%%%%%%%%%%%%%%%%%%%%%%%%%%%%%%%%%%%%%%%%%%%%%%%%%%%%%
\bibitem{CiuKrat-Dual} M.~Ciucu and C.~Krattenthaler.
A dual of MacMahon's theorem on plane partitions. Proc. Natl. Acad. Sci. USA 110 (2013), 4518-4523. 
%%%%%%%%%%%%%%%%%%%%%%%%%%%%%%%%%%%%%%%%%%%%%%%%%%%%%%%%%%%%%%%%%%%%%
 \bibitem{Ther} T. Eisenk\"olbl. Rhombus tilings of a hexagon with two triangles missing on the symmetry axis. Electron. J. Combinat. 6 (1) (1999), \# R30, 19pp.
%%%%%%%%%%%%%%%%%%%%%%%%%%%%%%%%%%%%%%%%%%%%%%%%%%%%%%%%%%%%%%%%%%%%%%%%%%%
\bibitem{Ilse} I.~Fischer. Enumeration of rhombus tilings of a hexagon which contain a fixed rhombus in the centre. J. Combin.
Theory Ser. A 96 (2001), no. 1, 31 -- 88.
%%%%%%%%%%%%%%%%%%%%%%%%%%%%%%%%%%%%%%%%%%%%%%%%%%%%%%%%%%%%%%%%%%%%%%%%%%%
\bibitem{Ilse2} I. Fischer. Another refinement of the Bender-Knuth (ex-)conjecture. European J. Combin. 27 (2006),
290--321.
%%%%%%%%%%%%%%%%%%%%%%%%%%%%%%%%%%%%%%%%%%%%%%%%%%%%%%%%%%%%%%%%%%%%%%%%%%%
\bibitem{FulKrat-I} M.~Fulmek and C.~Krattenthaler.
 The number of rhombus tilings of a symmetric hexagon which contain
 a fixed rhombus on the symmetry axis I.
 Ann. Comb.  2 (1998), no. 1, 19--41.
%%%%%%%%%%%%%%%%%%%%%%%%%%%%%%%%%%%%%%%%%%%%%%%%%%%%%%%%%%%%%%%%%%%%%%%%%%%
\bibitem{FulKrat-II} M.~Fulmek and C.~Krattenthaler.
The number of rhombus tilings of a symmetric hexagon which contain a
fixed rhombus on the symmetry axis II.
 European J. Combin.  21 (2000), no. 5, 601--640.
%%%%%%%%%%%%%%%%%%%%%%%%%%%%%%%%%%%%%%%%%%%%%%%%%%%%%%%%%%%%%%%%%%%%%%%%%%%
\bibitem{GaRa}
G.~Gasper and M.~Rahman. {\em Basic hypergeometric series},
Encyclopedia of Mathematics And Its Applications 35, Cambridge
University Press, Cambridge, 1990.
%%%%%%%%%%%%%%%%%%%%%%%%%%%%%%%%%%%%%%%%%%%%%%%%%%%%%%%%%%%%%%%%%%%%%%%%%%%
\bibitem{Gordon} B. Gordon. A proof of the Bender-Knuth conjecture. Pacific J. Math. 108 (1983), 99--113.
%%%%%%%%%%%%%%%%%%%%%%%%%%%%%%%%%%%%%%%%%%%%%%%%%%%%%%%%%%%%%%%%%%%%%%%%%%%
\bibitem{GospAB} R.~W.~Gosper.  Decision procedure for
indefinite hypergeometric summation. Proc.\ Natl.\ Acad.\ Sci.\ USA
75 (1978), 40--42.
%%%%%%%%%%%%%%%%%%%%%%%%%%%%%%%%%%%%%%%%%%%%%%%%%%%%%%%%%%%%%%%%%%%%%%%%%%%
\bibitem{GeHel}
I.~Gessel and H.~Helfgott. Enumeration of tilings of diamonds and hexagons with defects. Electron. J. Combin.  6 (1999), \# RP 16, 26 pp.
%%%%%%%%%%%%%%%%%%%%%%%%%%%%%%%%%%%%%%%%%%%%%%%%%%%%%%%%%%%%%%%%%%%%%%%%%%%
\bibitem{Ishikawa} M.~Ishikawa and M.~Wakayama,  Minor summation formula of Pfaffians.  
Linear and Multilinear Algebra 39 (1995), no. 3, 285--305.
%%%%%%%%%%%%%%%%%%%%%%%%%%%%%%%%%%%%%%%%%%%%%%%%%%%%%%%%%%%%%%%%%%%%%%%%%%%
\bibitem{Kratt} C. Krattenthaler, The major counting of nonintersecting lattice paths
and generating functions for tableaux. Mem. Amer. Math. Soc. 115,
no. 552, Providence, R. I., 1995.
%%%%%%%%%%%%%%%%%%%%%%%%%%%%%%%%%%%%%%%%%%%%%%%%%%%%%%%%%%%%%%%%%%%%%%%%%%%
\bibitem{Krat} C.~Krattenthaler.  An alternative evaluation of the Andrews-Burge determinant. In: "Mathematical Essays in Honor of Gian-Carlo Rota", 
B. E. Sagan, R. P. Stanley, eds., Progress in Math.,
  vol. 161, Birkh\"auser, Boston, 1998, pp. 263--270.
%%%%%%%%%%%%%%%%%%%%%%%%%%%%%%%%%%%%%%%%%%%%%%%%%%%%%%%%%%%%%%%%%%%%%%%%%%%
\bibitem{Krattt} C. Krattenthaler.  Advanced determinant calculus. S\'eminaire Lotharingien Combin. 42 ``The Andrews Festschrift'') (1999),
Article B42q, 67 pp.
%%%%%%%%%%%%%%%%%%%%%%%%%%%%%%%%%%%%%%%%%%%%%%%%%%%%%%%%%%%%%%%%%%%%%%%%%%%
\bibitem{KrattOkada} C. Krattenthaler and S. Okada. The number of rhombus tilings of a ``puncture'' hexagon and the minor summation formula. Adv. Appl. Math. 21 (1998), 381-404.
%%%%%%%%%%%%%%%%%%%%%%%%%%%%%%%%%%%%%%%%%%%%%%%%%%%%%%%%%%%%%%%%%%%%%%%%%%%
\bibitem{MacM} P. A. MacMahon. {\em  Combinatory Analysis}, volume 2, Cambridge University Press, 1916; (reprinted
Chelsea, New York, 1960.)
%%%%%%%%%%%%%%%%%%%%%%%%%%%%%%%%%%%%%%%%%%%%%%%%%%%%%%%%%%%%%%%%%%%%%%%%%%%
\bibitem{McDon} I. G. Macdonald. {\em Symmetric Functions and Hall Polynomials}, second edition, Oxford University Press,
New York/London, 1995.
%%%%%%%%%%%%%%%%%%%%%%%%%%%%%%%%%%%%%%%%%%%%%%%%%%%%%%%%%%%%%%%%%%%%%%%%%%%
\bibitem{PaScAA} P.~Paule and M.~Schorn.
 A Mathematica version of Zeilberger's algorithm for proving
binomial coefficient identities. J. Symbol.\ Comp.\  20
(1995), 673--698.
%%%%%%%%%%%%%%%%%%%%%%%%%%%%%%%%%%%%%%%%%%%%%%%%%%%%%%%%%%%%%%%%%%%%%%%%%%%
\bibitem{PeWZAA} M.~Petkov\v sek, H.~Wilf and D.~Zeilberger.
{\em A=B}, A.K. Peters, Wellesley, 1996.
%%%%%%%%%%%%%%%%%%%%%%%%%%%%%%%%%%%%%%%%%%%%%%%%%%%%%%%%%%%%%%%%%%%%%%%%%%%
\bibitem{Proct} R. A. Proctor. Bruhat lattices, plane partitions generating functions, and minuscule representations.
Europ. J. Combin. 5 (1984), 331--350.
%%%%%%%%%%%%%%%%%%%%%%%%%%%%%%%%%%%%%%%%%%%%%%%%%%%%%%%%%%%%%%%%%%%%%%%%%%%
\bibitem{Proct2} R. A. Proctor. Odd symplectic groups. Inventiones mathematicae  92 (1988), 307-332.
%%%%%%%%%%%%%%%%%%%%%%%%%%%%%%%%%%%%%%%%%%%%%%%%%%%%%%%%%%%%%%%%%%%%%%%%%%%
\bibitem{Stem} J.~R.~Stembridge. Nonintersecting paths, Pfaffians, and plane partitions.
 Adv. Math.~83 (1990), no. 1, 96--131.
%%%%%%%%%%%%%%%%%%%%%%%%%%%%%%%%%%%%%%%%%%%%%%%%%%%%%%%%%%%%%%%%%%%%%%%%%%%
\bibitem{SlatAC} L.~J.~Slater. {\em Generalized hypergeometric functions},
Cambridge University Press, Cambridge, 1966.
%%%%%%%%%%%%%%%%%%%%%%%%%%%%%%%%%%%%%%%%%%%%%%%%%%%%%%%%%%%%%%%%%%%%%%%%%%%
%%%%%%%%%%%%%%%%%%%%%%%%%%%%%%%%%%%%%%%%%%%%%%%%%%%%%%%%%%%%%%%%%%%%%%%%%%%
\end{thebibliography}
\end{document}